\documentclass[11pt]{article}
\usepackage[T1]{fontenc}
\usepackage[utf8]{inputenc}
\usepackage{lmodern,microtype}
\usepackage[margin=1in]{geometry}
\usepackage{amsmath,amssymb,amsthm,mathtools}
\usepackage{booktabs,enumitem,xcolor,mathrsfs}
\usepackage[colorlinks=true,linkcolor=blue,citecolor=blue,urlcolor=blue]{hyperref}
\newtheorem{theorem}{Theorem}[section]
\newtheorem{lemma}[theorem]{Lemma}
\newtheorem{conjecture}[theorem]{Conjecture}
\newcommand{\R}{\mathbb{R}}
\newcommand{\tr}{\operatorname{tr}}
\newcommand{\Spec}{\operatorname{Spec}}
\newcommand{\rank}{\operatorname{rank}}
\newcommand{\im}{\operatorname{im}}
\newcommand{\Ker}{\operatorname{ker}}
\newcommand{\Span}{\operatorname{span}}
\newcommand{\bd}{\partial}
\newcommand{\ip}[2]{\left\langle #1,#2\right\rangle}
\newcommand{\norm}[1]{\left\lVert #1\right\rVert}
\newcommand{\diag}{\operatorname{diag}}
\newcommand{\one}{\mathbf 1}
\newcommand{\calC}{\mathcal{C}}
\newcommand{\calF}{\mathcal{F}}
\newcommand{\calN}{\mathcal{N}}
\newcommand{\scrH}{\mathscr{H}}
\newcommand{\scrU}{\mathscr{U}}
\newcommand{\scrV}{\mathscr{V}}
\newcommand{\scrE}{\mathscr{E}}
\title{\bfseries The Duval--Reiner Conjecture:\\
Counterexamples and Partial-Sum Inequalities}
\author{Jing Huang\footnote{Email address: jhuangmath@gzhu.edu.cn}}
\date{\small School of Mathematics and Information Science, Guangzhou University,\\
Guangzhou 510006, China}
\begin{document}
\maketitle

\begin{abstract}
Let $F\subseteq\binom Vq$ be a $q$-uniform family on a finite vertex
set $V$. Write $s_r(F)$ for the sum of the $r$ largest eigenvalues of
its simplicial up-Laplacian and $d_F(v)$ for the degree of $v$.
Then $D_r(F)=\sum_{v\in V}\min\{d_F(v),r\}$ is the $r$-th partial
sum of the conjugate degree sequence. Duval and Reiner
[Trans. Amer. Math. Soc., 2002] conjectured that $s_r(F)\le D_r(F)$
for every $F$ and every $r\ge1$, with equality at all indices if and
only if $F$ is isomorphic to a shifted family. This conjecture
extends the Grone--Merris majorization theorem for graph Laplacians
to uniform simplicial families.
We disprove the majorization assertion: every $r\ge5$ admits a
strict counterexample in some uniformity, and every uniformity
$q\ge3$ admits a strict counterexample at some index $r\ge5$.
We prove $s_r(F)\le D_r(F)$ for $r=2,3,4$ in every uniformity and
classify all equality cases. Together with the known inequality at
$r=1$, these results determine exactly which indices satisfy the
inequality for every finite uniform family.
The counterexamples follow from two $3$-uniform families by
adjoining facets along a ridge and taking set complements.
For the inequalities, completion of the core reduces the problem
to complete-simplex boundary matrices; comparisons of weighted
orthogonal projections and a fourth-moment identity give the
required spectral estimates.


\end{abstract}

\section{Introduction}\label{sec:introduction}\label{sec:unified-introduction}
Let $V$ be a finite set and let $q\ge1$ be an integer. We write
$\binom Vq$ for the collection of all $q$-element subsets of $V$.
A subfamily $F\subseteq\binom Vq$ is called a \emph{$q$-uniform
family}, and its members are called \emph{facets}. The simplicial
complex generated by $F$ consists of all subsets of members of $F$.
A face $\tau$ has dimension $|\tau|-1$; a \emph{ridge} is a
$(q-1)$-element subset of a facet.
Fix a linear order on $V$. For $\sigma=\{v_0<\cdots<v_{q-1}\}$,
define its signed boundary by
$\bd e_\sigma=\sum_{j=0}^{q-1}(-1)^j e_{\sigma\setminus\{v_j\}}$.
Let $B_F$ be the corresponding boundary matrix, with columns indexed
by the facets of $F$ and rows indexed by their ridges. When comparing
families, we use a common set of row indices and add zero rows where
necessary. For $q=1$, the row indexed by $\varnothing$ consists of
ones. Set
$L_F^+=B_FB_F^{\mathsf T}$ and $M_F=B_F^{\mathsf T}B_F$.
The simplicial up-Laplacian $L_F^+$ and its facet Gram matrix $M_F$
are real symmetric positive semidefinite matrices with the same
positive eigenvalues, including multiplicities. Arrange these
eigenvalues in nonincreasing order, appending zeros when necessary,
and put
$\lambda_1(F)\ge\lambda_2(F)\ge\cdots\ge0$ and $s_r(F)=\sum_{j=1}^r\lambda_j(F)$.
For $v\in V$, let $d_F(v)=|\{\sigma\in F:v\in\sigma\}|$ be its
vertex degree. The conjugate degree sequence and its partial sums
are
\[
 d_j^T(F)=|\{v\in V:d_F(v)\ge j\}|,
 \qquad D_r(F)=\sum_{j=1}^r d_j^T(F)
             =\sum_{v\in V}\min\{d_F(v),r\}.
\]
Each column of $B_F$ has exactly $q$ entries of modulus one. Hence
the two sequences have the same total sum:
\[
 \sum_j\lambda_j(F)=\tr(L_F^+)=q|F|
 =\sum_{v\in V}d_F(v)=\sum_jd_j^T(F).
\]

When $q=2$, $F$ is the edge set of a finite simple graph, $B_F$ is
an oriented vertex--edge incidence matrix, and $L_F^+$ is the usual
graph Laplacian with isolated vertices omitted. The graph Laplacian
originates in Kirchhoff's work on electrical networks
\cite{Kirchhoff1847}; Eckmann developed the corresponding
finite-dimensional Hodge theory for simplicial complexes
\cite{Eckmann1944}. Grone and Merris conjectured that the Laplacian
spectrum of a graph is majorized by its conjugate degree sequence
\cite{GroneMerris1994}. Duval and Reiner proposed a higher-dimensional
extension for uniform simplicial families
\cite[Conjecture~1.2]{DuvalReiner2002}.

To state their conjecture, order the vertices as
$V=\{v_1<\cdots<v_n\}$. A $q$-uniform family is \emph{shifted} if,
whenever $i_1<\cdots<i_q$, $j_1<\cdots<j_q$, and $i_h\le j_h$ for
$1\le h\le q$, the inclusion
$\{v_{j_1},\ldots,v_{j_q}\}\in F$ implies
$\{v_{i_1},\ldots,v_{i_q}\}\in F$.
A family is \emph{isomorphic to a shifted family} if a relabeling of
its vertices makes it shifted.

\begin{conjecture}\label{conj:duval-reiner}
For every $q$-uniform family $F$ and every integer $r\ge1$,
$s_r(F)\le D_r(F)$.
Moreover, equality holds simultaneously for all $r\ge1$ if and only
if $F$ is isomorphic to a shifted family.
\end{conjecture}

The conjecture would bound every leading sum of simplicial
Laplacian eigenvalues using only the vertex degrees. It would thus
extend graph spectral majorization to all uniform simplicial
families, while its equality clause would characterize the families
whose entire Laplacian spectrum agrees with the conjugate degree
sequence. Equality at a single index is a different question:
$s_r(F)=D_r(F)$ for one value of $r$ need not imply that $F$ is
isomorphic to a shifted family. The relation between spectral sums
and combinatorial degrees also appears in extensions of Brouwer's
spectral-sum conjecture to simplicial complexes
\cite{Abebe12,Abebe}.

Duval and Reiner proved the first partial-sum inequality in every
uniformity and established the spectral identity for shifted
complexes \cite{DuvalReiner2002}. Duval extended this identity to
pairs of shifted complexes \cite{Duval2005}. In the graph case,
partial results included trees \cite{Stephen2007}, one-regular
semibipartite graphs \cite{Katz2005}, cographs
\cite{BapatLalPati2008}, and near-threshold graphs
\cite{Kirkland2009}. Bai proved the Grone--Merris conjecture for all
graphs \cite{Bai2011}.

For higher uniformities, Fan, Wu, and Wang obtained local degree
bounds and balancedness criteria at the first index
\cite{FanWuWang2025}. Lew proved a bound for spectral sums in terms
of ridge degrees and verified the Duval--Reiner inequalities for
$q$-partite $q$-uniform families \cite{Lew2025}. Zhang and Fan
refined the first-index bound and obtained homological and local
equality criteria \cite{ZhangFan2026}. Han, Lu, and Wang proved the
second partial-sum inequality for $3$-uniform families
\cite{HanLuWang2026}.

Our first result disproves the majorization assertion in
Conjecture~\ref{conj:duval-reiner}.
\begin{theorem}\label{thm:counterexamples}
For every integer $r\ge5$, there exist an integer $q\ge3$, a finite
set $V$, and a family $F\subseteq\binom Vq$ such that
$s_r(F)>D_r(F)$. Moreover, every integer $q\ge3$ admits such a
counterexample at some index $r\ge5$. In both assertions, $F$ may be
chosen so that $\bigcap_{\sigma\in F}\sigma=\varnothing$.
\end{theorem}

The two assertions have different orders of quantification. They do
not assert a counterexample for every pair $(q,r)$ with $q\ge3$ and
$r\ge5$. The empty-intersection condition excludes constructions
obtained solely by adding a common vertex to every facet.
We construct the counterexamples from two explicit $3$-uniform
families. Their characteristic polynomials follow from graph
Laplacians and an orthogonal symmetry decomposition. Adjoining facets
along a ridge and taking set complements then give all the asserted
indices and uniformities.

\begin{theorem}\label{thm:main}
Let $F$ be a finite $q$-uniform family, where $q\ge1$. Then
\[
 s_\ell(F)\le D_\ell(F)\qquad(\ell=2,3,4).
\]
\end{theorem}

Together with the known first-index inequality, Theorems~\ref{thm:counterexamples}
and~\ref{thm:main} show that the indices for which $s_r(F)\le D_r(F)$
holds for every finite uniform family are exactly $r=1,2,3,4$.
For these inequalities, we complete the family on its
vertices of degree at least $\ell$ and reduce the problem to a block
matrix arising from the boundary of a complete simplex. The required
bounds follow from comparisons of weighted orthogonal projections
and a fourth-moment identity.

Fix $\ell\in\{2,3,4\}$. Equality in Theorem~\ref{thm:main} holds
whenever $q=1$ or $|F|\le\ell$.
Indeed, if $q=1$, then $B_F$ is a row of ones, so
$s_\ell(F)=D_\ell(F)=|F|$.
If $|F|\le\ell$, then $\rank L_F^+=\rank B_F\le|F|\le\ell$
and $d_F(v)\le\ell$ for every $v\in V$. Hence
\[
s_\ell(F)=\tr L_F^+=q|F|
=\sum_{v\in V}d_F(v)=D_\ell(F).
\]
To state the remaining equality cases, delete isolated vertices and put
$Q=\{v\in V:d_F(v)\ge\ell\}$ and $m=|Q|$.
We call \(Q\) the \emph{core}
and the vertices in \(V\setminus Q\) \emph{outside vertices}.
A facet contained in \(Q\) is a core facet.

For a graph $G$, let $c(G)$ denote its number of connected components.
When $m=q-2$, let $G$ be the simple graph on $V\setminus Q$ with
$E(G)=\{\{u,v\}\in\tbinom{V\setminus Q}{2}:Q\cup\{u,v\}\in F\}$.
When $m=q-1$, let $G$ be the loopless multigraph on $V\setminus Q$ in which an
edge $uv$ has colour $w\in Q$ if and only if
$(Q\setminus\{w\})\cup\{u,v\}\in F$.
Thus parallel edges have distinct colours. Let $G_0$ consist of the
edges whose underlying pairs have multiplicity one, with isolated
vertices removed, and let $d$ count the pairs of multiplicity two.

When \(m\ge q\), put $\calN=\binom Qq\setminus F$ and
$\bd_Q=B_{\binom Qq}$, and let $\bd_{\calN}$ be the column submatrix
of $\bd_Q$ indexed by $\calN$. Define
 \begin{align*}
& H_u=\{\rho\in\tbinom Q{q-1}:\rho\cup\{u\}\in F\}
       \quad(u\in V\setminus Q),\\
 &\calC=\bigcup_{u\in V\setminus Q}H_u,\qquad
 k=|\calC|.
 \end{align*}
Since $u\notin Q$, we have $|H_u|\le d_F(u)\le\ell-1$.

For each nonempty $H_u$, consider the graph on $H_u$ in which
distinct ridges $\rho,\tau$ are adjacent exactly when
$|\rho\cap\tau|=q-2$.
Let $\mathcal H$ consist of the families of ridges forming
the connected components of these graphs, with identical families
from different $u$ counted separately.
For $H_1,H_2\in\mathcal H$, consider the following pair conditions:
\begin{enumerate}[label=\textup{(\roman*)},leftmargin=*]
	\item\label{cond:nested}
	$H_1\subseteq H_2$ or $H_2\subseteq H_1$.

	\item\label{cond:orthogonal}
	$H_1\cap H_2=\varnothing$, and
	$|\rho\cap\tau|\le q-3$ for every $\rho\in H_1$ and $\tau\in H_2$.

	\item\label{cond:diamond}
	$q\ge3$, and there exist a $(q-3)$-set $\tau_0\subseteq Q$
	and distinct vertices $v_1,v_2,v_3,v_4\in Q\setminus \tau_0$
	such that, after interchanging $H_1$ and $H_2$ if necessary,
	\[
	\begin{aligned}
		H_1&=\{\tau_0\cup\{v_1,v_2\},\,\tau_0\cup\{v_1,v_3\},\,\tau_0\cup\{v_1,v_4\}\},\\
		H_2&=\{\tau_0\cup\{v_1,v_2\},\,\tau_0\cup\{v_1,v_3\},\,\tau_0\cup\{v_2,v_3\}\}.
	\end{aligned}
	\]
\end{enumerate}
Let $e=|\{\sigma\in F:|\sigma\setminus Q|=2\}|$, and define
\[
 c_\ell(F)=\begin{cases}
 \displaystyle\sum_{\substack{u\in V\setminus Q\\H_u\ne\varnothing}}(|H_u|-1),&m=q,\\[2pt]
 \displaystyle\bigl|\{u\in V\setminus Q:|H_u|=3,\ |\bigcap_{\rho\in H_u}\rho|=q-2\}\bigr|,
    &m=q+1,\\[2pt]
 0,&m\ge q+2.
 \end{cases}
\]

\begin{theorem}\label{thm:equality}
Let $\ell\in\{2,3,4\}$. Suppose $q\ge2$, $|F|>\ell$, and $F$ has no isolated vertices.
If $m\le q-4$, then $s_\ell(F)<D_\ell(F)$.
For $m\ge q-3$, equality holds if and only if the corresponding
condition below is satisfied.
\begin{enumerate}[label=\textup{(\roman*)},leftmargin=*]
\item\label{case:deficient} $q-3\le m<q$.
The condition is determined by $m$ as follows:
 \begin{enumerate}[label={\((\mathrm a_{\arabic*})\)},
                   ref={\((\mathrm a_{\arabic*})\)},leftmargin=*]
 \item\label{case:tetrahedron} $m=q-3$, $\ell=4$, and
 $F=\{Q\cup\tau:\tau\in\tbinom X3\cup\{Y\}\}$,
 where $X,Y\subseteq V\setminus Q$ are disjoint,
 $|X|=4$, and $|Y|=3$.
 \item\label{case:graph} $m=q-2$, and exactly one of the following holds:
 \begin{itemize}[leftmargin=*]
 \item Every facet contains $Q$, and $|V(G)|-c(G)\le\ell$.
 \item Exactly one facet does not contain $Q$, $\ell=4$, and
 \[
 F=\{Q\cup\tau:\tau\in\tbinom X2\cup\{Y\}\}
   \cup\{(Q\setminus\{w\})\cup X\},
 \]
 where $w\in Q$, the sets $X,Y\subseteq V\setminus Q$ are disjoint,
 $|X|=3$, and $|Y|=2$.
 \end{itemize}
 \item\label{case:rooted} $m=q-1$, $Q\cup\{v\}\in F$ for every
 $v\in V\setminus Q$, and every
 facet has at most two outside vertices. The multiple edges of $G$
 are isolated double edges, each component of $G_0$ has one colour, and
 $|V(G_0)|-c(G_0)+2d\le\ell-1$ and $|V\setminus Q|+d\ge\ell$.
 \end{enumerate}
\item\label{case:complete} $m\ge q$, and all of the following hold:
 \begin{enumerate}[label={\((\mathrm b_{\arabic*})\)},
                   ref={\((\mathrm b_{\arabic*})\)},leftmargin=*]
 \item\label{case:outside} Every facet has at most two outside vertices.
 If $e>0$, then $\ell=4$, $m=q$, and each facet with two outside
 vertices has the form
 \[
 (Q\setminus\{w_1,w_2\})\cup\{u,v\},\qquad
 H_u=H_v=\{Q\setminus\{w_1\},Q\setminus\{w_2\}\}.
 \]
 The outside pairs of these facets are pairwise disjoint.
 \item\label{case:links} Every pair of members of $\mathcal H$
 satisfies at least one of the pair conditions
 \ref{cond:nested}--\ref{cond:diamond} above.
 \item\label{case:dimensions} $k+e\le\ell$,
 $\rank_{\R}\bd_{\calN}\le\binom{m-1}{q-1}+c_\ell(F)-\ell$,
 and $\rho\cup\{v\}\in F$ for every $\rho\in\calC$ and
 $v\in Q\setminus\rho$.
 \end{enumerate}
\end{enumerate}
\end{theorem}

The paper is organized as follows.
Section~\ref{sec:unified-preliminaries} establishes the matrix and boundary
identities and reduces the three inequalities to block-matrix estimates.
Section~\ref{sec:seeds} proves Theorem~\ref{thm:counterexamples}.
Section~\ref{sec:projection-estimates} treats weighted projections and
matrices on at most four coordinates. Section~\ref{sec:large-residuals}
uses the fourth-moment identity for the remaining matrices.
Section~\ref{sec:unified-equality} proves the equality classification and
completes the proof of Theorem~\ref{thm:main}.

\section{Preliminaries and Core Reduction}\label{sec:unified-preliminaries}\label{sec:notation}
For real symmetric matrices, $A\succeq0$ and $A\succ0$ denote
positive semidefiniteness and positive definiteness, respectively.
We write $A\preceq B$ when $B-A\succeq0$.
For a real symmetric matrix $A$ of order $n$, arrange its eigenvalues as
$\lambda_1(A)\ge\cdots\ge\lambda_n(A)$ and put
$s_r(A)=\sum_{i=1}^r\lambda_i(A)$ for $0\le r\le n$, with $s_0(A)=0$.
For $A\succeq0$, append zero eigenvalues and use this definition
for every $r\ge0$. In particular, $s_r(F)=s_r(L_F^+)=s_r(M_F)$.

Let $A_+$ and $A_-$ denote the positive and negative spectral parts
of $A$, with $A=A_+-A_-$. We use
\[
 s_r^+(A)=s_r(A_+)=\sum_{i=1}^r\lambda_i(A)_+,
 \qquad t_+=\max\{t,0\},
\]
where $\lambda_i(A)_+=0$ for $i>n$.
Write $\norm{\mathbf x}$ for the Euclidean norm and
$\norm{A}=\max_{\norm{\mathbf x}=1}\norm{A\mathbf x}$
for the induced operator norm.

Taking the maximum over ranks at most $r$ in Fan's maximum principle
\cite[Theorem~1, p.~653]{Fan} gives the first identity below and its
subadditivity consequence. The monotonicity and perturbation bounds
also follow from Weyl's inequalities
\cite[Corollaries~III.2.3 and III.2.6, p.~63]{Bhatia}.

\begin{lemma}\label{lem:ky-fan}
Let $A,B,M$ be real symmetric matrices of the same order and let
$r\ge0$ be an integer. Then
\[
 s_r^+(M)
 =\max_{\Pi^2=\Pi=\Pi^{\mathsf T},\,\tr\Pi\le r}\tr(\Pi M).
\]
If $A\preceq B$, then $s_r^+(A)\le s_r^+(B)$. Moreover,
\[
 s_r^+(A+B)\le s_r^+(A)+s_r^+(B),
 \qquad
 |s_r^+(A)-s_r^+(B)|\le r\norm{A-B}.
\]
\end{lemma}

\begin{lemma}\label{lem:zero-upper-range}
Let $X\in\R^{p\times n}$ and let $L\in\R^{n\times n}$ be symmetric.
For
$A=\left(\begin{smallmatrix}0&X\\X^{\mathsf T}&L\end{smallmatrix}\right)$,
let $\scrE_\pm(A)$ denote its positive and negative spectral subspaces,
and let $\operatorname{pr}_{\rm up}$ be projection onto the first
$p$ coordinates. Then
\begin{equation}\label{eq:zero-upper-range}
 \operatorname{pr}_{\rm up}\scrE_+(A)
 =\operatorname{pr}_{\rm up}\scrE_-(A)=\im X.
\end{equation}
\end{lemma}
\begin{proof}
For $\mathbf x\in\R^p$, the zero upper principal block gives
$(\mathbf x,0)^{\mathsf T}A_+(\mathbf x,0)
=(\mathbf x,0)^{\mathsf T}A_-(\mathbf x,0)$.
Since $A_+$ and $A_-$ are positive semidefinite,
$(\mathbf x,0)$ is orthogonal to $\scrE_+(A)$ if and only if
it is orthogonal to $\scrE_-(A)$.
By $\Ker A=\Ker A_+\cap\Ker A_-$, these conditions hold
precisely when $A(\mathbf x,0)=0$, or $X^{\mathsf T}\mathbf x=0$.
Thus $(\operatorname{pr}_{\rm up}\scrE_\pm(A))^\perp=\Ker X^{\mathsf T}$.
Taking orthogonal complements proves the assertion.
\end{proof}

\begin{lemma}\label{lem:gram-identification}
	Let $X,Y$ be real matrices with $n$ columns, and let
	$L\in\R^{n\times n}$ be symmetric.
	If $X^{\mathsf T}X=Y^{\mathsf T}Y$, then the matrices
	\begin{equation}\label{eq:gram-identification}
		\begin{pmatrix}0&X\\X^{\mathsf T}&L\end{pmatrix},\qquad
		\begin{pmatrix}0&Y\\Y^{\mathsf T}&L\end{pmatrix}
	\end{equation}
	have the same nonzero eigenvalues, with multiplicity.
	Moreover, $XX^{\mathsf T}$ and $X^{\mathsf T}X$ have the same
	positive eigenvalues, with multiplicity.
\end{lemma}
\begin{proof}
	The identities $\Ker X=\Ker Y$ and
	$\langle X\mathbf u,X\mathbf v\rangle
	=\mathbf u^{\mathsf T}X^{\mathsf T}X\mathbf v
	=\langle Y\mathbf u,Y\mathbf v\rangle$
	make $X\mathbf v\mapsto Y\mathbf v$ a well-defined isometry
	from $\im X$ onto $\im Y$.
	Together with the identity on the lower coordinates, this map
	intertwines the restrictions of the displayed matrices to
	$\im X\oplus\R^n$ and $\im Y\oplus\R^n$.
	Their orthogonal complements lie in the respective kernels,
	proving the first assertion.
	
	For each $\lambda>0$, multiplication by $X$ maps the
	$\lambda$-eigenspace of $X^{\mathsf T}X$ bijectively onto that
	of $XX^{\mathsf T}$, with inverse $\lambda^{-1}X^{\mathsf T}$.
	This proves the second assertion.
\end{proof}

Let $e_v$, $v\in Q$, be the standard basis of $\R^Q$.
For $\sigma=\{v_1<\cdots<v_k\}\subseteq Q$, write
$e_\sigma=e_{v_1}\wedge\cdots\wedge e_{v_k}$, with
$e_\varnothing=1$. These vectors form an orthonormal basis of
$\bigwedge\R^Q$.
For $\mathbf u=\sum_{v\in Q}u_ve_v$, define contraction by
\[
 \iota_{\mathbf u}e_\sigma
 =\sum_{j=1}^{k}(-1)^{j-1}u_{v_j}
   e_{\sigma\setminus\{v_j\}},
\]
set $\iota_{\mathbf u}e_\varnothing=0$, and extend linearly.
Its adjoint is the map $\mathbf z\mapsto\mathbf u\wedge\mathbf z$.
The boundary operator is $\bd=\iota_{\sum_{v\in Q}e_v}$.
Let $\bd^*$ denote its adjoint with respect to this orthonormal
basis, and let $\bd_Q$ be the matrix of $\bd$ from degree $q$
to degree $q-1$.
For $m=|Q|>0$, the complete-simplex identities
\cite[(3.1)--(3.4), pp.~4318--4319]{DuvalReiner2002} give
\begin{equation}\label{eq:cartan}
	\bd^2=0,\qquad \bd\bd^*+\bd^*\bd=mI.
\end{equation}
Let $\Pi_{\mathscr L}$ denote the orthogonal projection onto
a subspace $\mathscr L$, and put
$\scrH_Q=\im(\bd_Q^{\mathsf T})$.
By \cite[Theorem~1.1, p.~4314]{DuvalReiner2002},
$\rank\bd_Q=\binom{m-1}{q-1}$ and
$\bd_Q^{\mathsf T}\bd_Q=m\Pi_{\scrH_Q}$.

For the case $|Q|=q-1$, consider a loopless edge-coloured
multigraph with $p$ distinguished vertices, called roots,
and $e$ edges. Parallel edges have distinct colours.
Orient the edges so that parallel edges have the same orientation,
and order the edges by colour.
For each colour $w$, let $B_w$ be its incidence matrix and
let $X_w$ consist of the root rows of $B_w$.
Let $E$ have rows indexed by underlying vertex pairs and columns
indexed by edges, with entry one exactly when the edge has that
underlying pair, and zero otherwise.
Write $\one_n\in\R^n$ for the all-one vector and
$J_n=\one_n\one_n^{\mathsf T}$.
On the root and edge coordinates, respectively, define
\[
X=[X_w]_w,\qquad
A=\begin{pmatrix}
	J_p&X\\
	X^{\mathsf T}&
	\displaystyle\bigoplus_w B_w^{\mathsf T}B_w-E^{\mathsf T}E
\end{pmatrix}.
\]

For a real symmetric matrix $A$, let $n_+(A),n_-(A),n_0(A)$
count its positive, negative, and zero eigenvalues, respectively,
with multiplicity.

An edge is called simple if its endpoints are joined by no other edge.
Let $d$ be the number of underlying pairs of multiplicity two, and
let $r$ be the rank of the incidence matrix of the simple edges.
All degrees count edge multiplicities.

\begin{lemma}\label{prop:rooted-negative-budget}
	Suppose every root has degree at most two and every other vertex
	has degree at most three. Then $\tr(A_-)\le e$, equivalently
	$\tr(A_+)\le p+2e$.
	Equality holds if and only if every edge endpoint is a root and
	any two incident nonparallel edges have the same colour.
	In the equality case with $p>0$, we have
	$n_+(A)=1+r+2d$ and $n_0(A)=p-r-d-1$.
\end{lemma}

\begin{proof}
	Since every diagonal entry of $A$ is one, $\tr A=p+e$.
	Thus the two bounds are equivalent.
	
	For each colour $w$, let $e_w$ be the number of simple edges of
	that colour, and restrict $B_w,X_w$ to these columns to obtain
	$\widehat B_w,\widehat X_w$. Let $B_w^0$ consist of the nonroot
	rows of $\widehat B_w$, and put
	\[
	A_w=
	\begin{pmatrix}
		0&\widehat X_w\\
		\widehat X_w^{\mathsf T}&
		\widehat B_w^{\mathsf T}\widehat B_w-I
	\end{pmatrix}.
	\]
	For $\mathbf x\in\R^p$ and $\mathbf y\in\R^{e_w}$,
	\[
	\binom{\mathbf x}{\mathbf y}^{\mathsf T}
	(A_w+I)
	\binom{\mathbf x}{\mathbf y}
	=
	\|\mathbf x+\widehat X_w\mathbf y\|^2
	+\|B_w^0\mathbf y\|^2.
	\]
	Hence $A_w\succeq-I$ and
	$n_0(A_w+I)=e_w-\rank B_w^0$.
	A negative definite subspace projects injectively onto the edge
	coordinates, since the quadratic form vanishes on vectors supported
	on the root coordinates. Therefore
	$\tr((A_w)_-)\le n_-(A_w)\le e_w$, with equality in
	$\tr((A_w)_-)\le e_w$ if and only if $B_w^0=0$.
	
	Distinct repeated pairs have no common endpoint, since all degrees
	are at most three. For a repeated pair $P$ of multiplicity
	$t_P\in\{2,3\}$, let $K_P$ be the block of $A$ with columns
	indexed by its edges and rows indexed by the roots and simple
	edges. Put
	\[
	\Delta_P=
	\begin{pmatrix}
		0&K_P\\
		K_P^{\mathsf T}&2I_{t_P}-J_{t_P}
	\end{pmatrix}.
	\]
	If $t_P=3$, both endpoints are nonroots of degree three.
	Thus $K_P=0$ and $\tr((\Delta_P)_-)=1<3$.
	
	Suppose $t_P=2$.
	A root endpoint contributes a row $\pm(1,1)$ to $K_P$.
	A nonroot endpoint is incident with at most one further edge,
	which contributes a signed coordinate row if its colour matches
	one of the two edges of $P$, and contributes zero otherwise.
	Put $\mathbf u=2^{-1/2}K_P(1,1)^{\mathsf T}$ and
	$\mathbf v=2^{-1/2}K_P(1,-1)^{\mathsf T}$.
	Using the normalized sum and difference of the edge coordinates
	gives
	\[
	\Delta_P\sim
	\begin{pmatrix}
		0&\mathbf u&\mathbf v\\
		\mathbf u^{\mathsf T}&0&0\\
		\mathbf v^{\mathsf T}&0&2
	\end{pmatrix},
	\]
	where $\sim$ denotes orthogonal similarity.
	If both endpoints are roots, then $\mathbf v=0$ and
	$\|\mathbf u\|=2$, so the nonzero eigenvalues are $-2,2,2$.
	
	Otherwise, $K_P$ has at most two nonzero rows, at most one of
	which is $\pm(1,1)$. Hence $\tr(K_PK_P^{\mathsf T})\le3$.
	Split the displayed matrix into its $\mathbf u$- and
	$\mathbf v$-blocks, keeping the diagonal entry $2$ in the latter.
	The traces of their negative parts are $\|\mathbf u\|$ and
	$\sqrt{1+\|\mathbf v\|^2}-1$, respectively.
	By Lemma~\ref{lem:ky-fan} and Cauchy--Schwarz,
	\begin{align*}
		\tr((\Delta_P)_-)
		&\le \|\mathbf u\|+\sqrt{1+\|\mathbf v\|^2}-1\le \sqrt{2(1+\|\mathbf u\|^2+\|\mathbf v\|^2)}-1\\
		&=\sqrt{2(1+\tr(K_PK_P^{\mathsf T}))}-1\le2\sqrt2-1<2.
	\end{align*}
	Thus $\tr((\Delta_P)_-)\le t_P$, with equality precisely when
	$P$ is a double edge whose endpoints are roots.
	
	Extend the matrices $A_w,\Delta_P$ by zero to the full coordinate
	space. Since different repeated pairs have no couplings between
	their edge coordinates,
	$A=J_p\oplus0+\sum_w A_w+\sum_P\Delta_P$.
	Lemma~\ref{lem:ky-fan} now gives
	\[
	\tr(A_-)
	\le\sum_w\tr((A_w)_-)+\sum_P\tr((\Delta_P)_-)
	\le\sum_w e_w+\sum_P t_P=e.
	\]
	
	Suppose equality holds. The preceding estimates imply that
	$B_w^0=0$ for every $w$ and that every repeated pair is a
	double edge with both endpoints roots.
	Thus every edge endpoint is a root, and the degree assumption
	makes each double edge isolated.
	
	Let $\Pi$ be the negative spectral projection of $A$.
	Equality in the same trace bound gives
	\[
	0=\tr((A_w)_-)+\tr(\Pi A_w)
	=\tr((I-\Pi)(A_w)_-)+\tr(\Pi(A_w)_+).
	\]
	Both terms on the right are nonnegative. Hence
	$\scrE_-(A_w)\subseteq\im\Pi$ and
	$\scrE_+(A_w)\subseteq\Ker\Pi$.
	For distinct colours $v,w$, the spaces
	$\scrE_-(A_w)$ and $\scrE_+(A_v)$ are therefore orthogonal.
	Their edge supports are disjoint, so their projections onto the
	root coordinates are orthogonal. By
	Lemma~\ref{lem:zero-upper-range}, this gives
	$\widehat X_v^{\mathsf T}\widehat X_w=0$.
	The columns are now full incidence columns, and two incident
	simple edges have incidence inner product $\pm1$.
	Such edges must therefore have the same colour.
	
	Conversely, suppose every edge endpoint is a root and any two
	incident nonparallel edges have the same colour.
	Every repeated pair is then an isolated double edge, and every
	component formed by simple edges has one colour.
	If $p=0$, then $e=0$ and equality is immediate. Assume $p>0$.
	
	Let $B$ be the incidence matrix of the simple edges, with all
	$p$ root rows. Its colour blocks are mutually orthogonal, and
	$B^{\mathsf T}\one_p=0$.
	Let $\lambda_1,\ldots,\lambda_r>0$ be the positive eigenvalues
	of $B^{\mathsf T}B$.
	On $\im B$ together with the simple-edge coordinates,
	a singular-value decomposition of $B$ gives the blocks
	\[
	\begin{pmatrix}
		0&\sqrt{\lambda_i}\\
		\sqrt{\lambda_i}&\lambda_i-1
	\end{pmatrix}
	\quad(1\le i\le r),
	\qquad
	-I_{e-2d-r}.
	\]
	Each $2\times2$ block has characteristic polynomial
	$(x+1)(x-\lambda_i)$.
	Thus these coordinates contribute the positive eigenvalues
	$\lambda_1,\ldots,\lambda_r$ and the eigenvalue $-1$ with
	multiplicity $e-2d$.
	
	Each double edge contributes the eigenvalues $-2,2,2$ computed
	above. Its root-difference direction is orthogonal to $\im B$,
	to $\one_p$, and to the root differences of the other double
	edges. The remaining root space contributes the eigenvalue $p$
	on $\Span\{\one_p\}$ and zero with multiplicity $p-r-d-1$.
	These subspaces exhaust all root and edge coordinates.
	Consequently, $\tr(A_-)=(e-2d)+2d=e$,
	$n_+(A)=1+r+2d$, and $n_0(A)=p-r-d-1$.
\end{proof}

For the core reduction, let $C\succeq0$ be a $k\times k$ matrix
with constant diagonal $a\ge0$. For coordinate sets $I,J$, write
$C[I,J]$ for the corresponding submatrix and put $C[I]=C[I,I]$.
Choose nonempty sets $H\subseteq\{1,\ldots,k\}$ with weights
$b_H>0$, allowing repetitions. Let
$R_H:\R^H\longrightarrow\R^k$ extend a vector by zero outside $H$.
Set $U=[\sqrt{b_H}R_H]_H$, $S=\bigoplus_H C[H]$, and
$N=\sum_H b_H|H|$, and define
\[
T=
\begin{pmatrix}
	0&C^{1/2}U\\
	U^{\mathsf T}C^{1/2}&U^{\mathsf T}U-S
\end{pmatrix}.
\]
The sets $H$ are called links. Each occurrence of $H$ gives a
separate lower block, with coordinates indexed by $j\in H$.
A coordinate is used if it belongs to at least one link.
When changing the sign of a coordinate, we change the signs of
the corresponding row and column of $C$ and make the same change
in every lower block containing that coordinate.
Combining equal links means replacing them by one link whose
weight is the sum of their weights.

\begin{lemma}\label{lem:link-invariance}
	Deleting unused coordinates preserves the nonzero eigenvalues of $T$.
	Simultaneous sign changes preserve all its eigenvalues.
	Combining equal links preserves its positive eigenvalues and $N$.
	In each case, multiplicities are preserved.
\end{lemma}

\begin{proof}
	The Gram matrix of the upper-right block is $U^{\mathsf T}CU$.
	Deleting unused coordinates changes neither this matrix nor the
	lower-right block. Hence Lemma~\ref{lem:gram-identification}
	gives the first assertion.
	Simultaneous sign changes give an orthogonal change of basis
	of $T$, so its eigenvalues are unchanged.
	
	Suppose $H$ occurs with weights $b_1,\ldots,b_t$, and put
	$b=\sum_{i=1}^t b_i$.
	Choose an orthogonal matrix $Q$ such that
	$(\sqrt{b_1},\ldots,\sqrt{b_t})Q=(\sqrt b,0,\ldots,0)$.
	Apply $Q$ to the $t$ occurrences of each coordinate of $H$.
	The corresponding block of $U$ becomes
	$[\sqrt b\,R_H,0,\ldots,0]$, while the $t$ identical diagonal
	blocks $C[H]$ of $S$ remain unchanged.
	Thus $T$ is orthogonally similar to the direct sum of the matrix
	obtained after combining these links and $t-1$ blocks $-C[H]$.
	Since $C[H]\succeq0$, these additional blocks have no positive
	eigenvalues. The positive eigenvalues are therefore preserved,
	with multiplicity. The contribution to $N$ is $b|H|$ in both cases.
\end{proof}

By Lemma~\ref{lem:link-invariance}, we may delete unused coordinates.
We do so unless stated otherwise, and write $k$ for the number of
remaining coordinates.
We say that $C$ arises from a complete simplex if, up to the sign
changes described above, it is a principal submatrix of
$\bd_Q\bd_Q^{\mathsf T}$, where $|Q|=m\ge q$ and $\bd_Q$ is
the boundary matrix of $\binom Qq$.
Then $a=m-q+1$ and $C_{ij}\in\{0,\pm1\}$ for $i\ne j$.
In the application to $F$, we have $|H|\le3$ and
$b_H\in\mathbb Z_{>0}$. We shall state explicitly when larger
links or arbitrary positive weights are allowed.

\begin{lemma}\label{prop:short-reduction}\label{prop:deficient-equality}
	Let $\ell\in\{2,3,4\}$, and let $F$ be a $q$-uniform family
	with $q\ge2$, $|F|>\ell$, and no isolated vertices.
	If $m<q$, then $s_\ell(F)\le D_\ell(F)$, with equality if and
	only if the conditions listed in one of
	cases~\ref{case:tetrahedron}--\ref{case:rooted}
	of Theorem~\ref{thm:equality} are satisfied.
	If $m\ge q$, form $T$ from
	$C=(\bd_Q\bd_Q^{\mathsf T})[\calC]$ and the nonempty links $H_u$,
	each with weight one. These links have size at most $\ell-1$, and
	\begin{equation}\label{eq:unified-residual-reduction}
		s_\ell(F)-D_\ell(F)
		\le s_\ell^+(T)-\sum_u|H_u|.
	\end{equation}
\end{lemma}

\begin{proof}
	By the definition of $Q$,
	$D_\ell(F)=\ell m+\sum_{v\notin Q}d_F(v)$ and
	$d_F(v)\le\ell-1$ for $v\notin Q$.
	Since $|F|>\ell$,
	$s_\ell(F)=\ell m+s_\ell(M_F-mI)$.
	We shall also use that, for a symmetric matrix $A$ of order at
	least $\ell$, $s_\ell(A)\le\tr(A_+)$, with equality if and only if
	$n_+(A)\le\ell\le n_+(A)+n_0(A)$.
	
	Suppose a facet $\sigma\in F$ is added after some facets of
	$F\cup\binom Qq$.
	After extending the previous matrix $M-mI$ by a zero row and
	column, the change is
	\[
	E_\sigma=
	\begin{pmatrix}
		0&\mathbf b\\
		\mathbf b^{\mathsf T}&q-m
	\end{pmatrix},
	\qquad
	\|\mathbf b\|^2=n_\sigma,
	\]
	where $n_\sigma$ counts the previous facets sharing a ridge
	with $\sigma$.
	The quadratic equation $z^2-(q-m)z-n_\sigma=0$ gives
	$\tr((E_\sigma)_+)
	=(q-m+\sqrt{(q-m)^2+4n_\sigma})/2$.
	If $j=|\sigma\setminus Q|\ge2$, each such previous facet
	contains at least $j-1$ vertices of $\sigma\setminus Q$.
	Each of these vertices belongs to at most $\ell-2$ facets
	other than $\sigma$, so
	$(j-1)n_\sigma\le(\ell-2)j\le2j$.
	Since $j\ge q-m$, we have $2j+m-q\ge j>0$, and squaring gives
	\begin{equation}\label{eq:insertion-bound}
		\tr((E_\sigma)_+)\le j
		\quad\Longleftrightarrow\quad
		n_\sigma\le j(j+m-q).
	\end{equation}
	Equality holds precisely when $n_\sigma=j(j+m-q)$.
	By Lemma~\ref{lem:ky-fan}, inserting $\sigma$ increases
	$s_\ell^+(M-mI)$, and also $\tr((M-mI)_+)$, by at most
	$\tr((E_\sigma)_+)$.
	
	Suppose first that $q-m\ge2$.
	Start with the facets containing $Q$, written as
	$F_0=\{Q\cup\tau:\tau\in H\}$.
	The boundary terms with different numbers of core vertices are
	orthogonal, so $M_{F_0}=mI+M_H$ and
	$\tr((M_{F_0}-mI)_+)=(q-m)|H|$.
	Every remaining facet has $j\ge q-m+1\ge3$, and
	$n_\sigma\le2j/(j-1)\le j\le j(j+m-q)$.
	Successive applications of \eqref{eq:insertion-bound} therefore give
	$\tr((M_F-mI)_+)\le\sum_{v\notin Q}d_F(v)$, and hence
	$s_\ell(F)\le D_\ell(F)$.
	If equality holds, every insertion attains its bound, in any order.
	
	Suppose $q-m\ge3$.
	A facet not containing $Q$ has $j\ge4$, so
	$n_\sigma\le2j/(j-1)<j\le j(j+m-q)$.
	Its insertion would give strict inequality.
	Thus equality requires $F=F_0$.
	In this case, $M_F=mI+M_H$ and
	$D_\ell(F)=\ell m+\tr M_H$, so equality is equivalent to
	$\rank B_H\le\ell<|H|$.
	Choose $0\ne\mathbf x\in\Ker B_H$ and $\tau\in H$ with
	$x_\tau\ne0$. For every $v\in\tau$,
	\[
	0=(B_H\mathbf x)_{\tau\setminus\{v\}}
	=\sum_{\substack{\rho\in H\\
			\rho\supseteq\tau\setminus\{v\}}}
	(B_H)_{\tau\setminus\{v\},\rho}x_\rho.
	\]
	The summand indexed by $\tau$ is nonzero.
	Hence there is a facet $\rho_v\ne\tau$ containing
	$\tau\setminus\{v\}$ with $x_{\rho_v}\ne0$.
	These facets are distinct, since $\rho_v=\rho_w$ for $v\ne w$
	would imply $\tau\subseteq\rho_v$.
	For each $u\in\tau$, the facets $\tau$ and
	$\rho_v$ with $v\ne u$ all contain $u$. Therefore
	$q-m\le d_H(u)\le\ell-1\le3$, giving $q-m=3$ and $\ell=4$.
	
	Write $\tau=\{v_1,v_2,v_3\}$ and the three selected facets as
	$\{v_1,v_2,w_1\}$, $\{v_1,v_3,w_2\}$, and
	$\{v_2,v_3,w_3\}$, where $w_i\notin\tau$.
	The three facets containing $v_1$ are exactly
	$\tau,\{v_1,v_2,w_1\},\{v_1,v_3,w_2\}$.
	The boundary equation at $\{v_1,w_1\}$ must cancel the nonzero
	contribution from $\{v_1,v_2,w_1\}$, so $w_2=w_1$.
	The equation at $\{v_2,w_1\}$ similarly gives $w_3=w_1$.
	Put $X=\{v_1,v_2,v_3,w_1\}$.
	Then $\binom X3\subseteq H$, and every vertex of $X$ already
	has degree three. Thus every other facet is disjoint from $X$.
	The four boundary columns on $X$ have Gram matrix $4I_4-J_4$
	after sign changes, and hence rank three. Consequently,
	$\rank B_{H\setminus\binom X3}=\rank B_H-3\le1$.
	The remaining family is nonempty because $|H|>4$.
	Two distinct facets have linearly independent boundary columns,
	since each has a ridge not contained in the other.
	Hence $H=\binom X3\cup\{Y\}$ with $|Y|=3$ and
	$X\cap Y=\varnothing$, giving case~\ref{case:tetrahedron}.
	Conversely, this family has $\rank B_H=4$, so
	$s_4(F)=4m+\tr M_H=D_4(F)$.
	
	Now suppose $q-m=2$.
	Let $G$ have the edges $uv$ for which $Q\cup\{u,v\}\in F$,
	and start with these facets.
	If an omitted facet has $j\ge4$, then
	$n_\sigma\le2j/(j-1)<j(j-2)$, so equality is impossible.
	Suppose equality holds and insert any omitted facet first.
	It has the form $(Q\setminus\{w_0\})\cup\{u,v,w\}$, and
	equality in \eqref{eq:insertion-bound} requires $n_\sigma=3$.
	Thus $uv,uw,vw\in E(G)$ and $\ell=4$.
	Each of $u,v,w$ now belongs to three facets of $F$.
	Hence this triangle is an isolated component of $G$, and no
	other omitted facet meets it.
	
	Let $t$ count these triangle components, and put
	$e=|E(G)|$ and $r=\rank B_G=|V(G)|-c(G)$.
	For each of the $t$ triangles, its three facets and the additional
	facet have a common $(m-1)$-set. After removing this set, they
	are the four triples of a four-set.
	Their block in $A=M_F-mI$ is therefore $3I_4-J_4$, up to sign
	changes. Thus the eigenvalues $3,3,0$ of the triangle block are
	replaced by $3,3,3,-1$.
	Facets belonging to different components share no ridge, so
	\[
	n_+(A)=r+t,\qquad
	n_0(A)=e-r-t,\qquad
	\tr(A_+)=2e+3t=D_\ell(F)-\ell m.
	\]
	Equality is consequently equivalent to $r+t\le\ell\le e$.
	If $t=0$, then $e=|F|>\ell$, and this condition reduces to
	$r\le\ell$, the first alternative in case~\ref{case:graph}.
	If $t>0$, then $\ell=4$ and $r\ge2t$, so $t=1$ and
	$r-2\le1$.
	Since $e\ge4$, the graph outside the triangle has at least one
	edge. Its incidence rank is at most one, so it has exactly one
	edge. This gives the second alternative in case~\ref{case:graph}.
	Here $r=3$ and $e=4$, so the same criterion also proves equality.
	
	Suppose next that $q-m=1$.
	Let $F_0$ consist of the facets having one or two vertices
	outside $Q$.
	Represent $Q\cup\{v\}$ by a root $v$ and
	$(Q\setminus\{w\})\cup\{u,v\}$ by an edge $uv$ of colour $w$.
	Roots have degree at most $\ell-2$, and the other vertices
	have degree at most $\ell-1$.
	Place the core vertices first, write
	$Q=\{w_0<\cdots<w_{m-1}\}$, and put
	$\mathbf f_{w_j}=(-1)^j e_{Q\setminus\{w_j\}}$.
	Orient the root facets as $(-1)^m e_Q\wedge e_v$ and the
	edge facets as $\mathbf f_w\wedge e_u\wedge e_v$.
	Use the incidence column $e_u-e_v$ for $uv$, with the same
	orientation on parallel edges.
	By \eqref{eq:cartan},
	$\bd e_Q=\sum_{w\in Q}\mathbf f_w$ and
	$\ip{\bd\mathbf f_w}{\bd\mathbf f_{w'}}
	=m\mathbf1_{\{w=w'\}}-1$.
	Also,
	\[
	\bd(\mathbf f_w\wedge e_u\wedge e_v)
	=(\bd\mathbf f_w)\wedge e_u\wedge e_v
	+(-1)^m\mathbf f_w\wedge(e_u-e_v).
	\]
	Taking inner products gives $M_{F_0}-mI=A$, where $A$ is
	the matrix in Lemma~\ref{prop:rooted-negative-budget}.
	If $p,e$ count the roots and edges, that lemma gives
	$\tr(A_+)\le p+2e$.
	
	Every omitted facet has $j\ge3$ and satisfies
	$n_\sigma\le2j/(j-1)\le j<j(j-1)$.
	The insertion bound therefore gives $s_\ell(F)\le D_\ell(F)$,
	strictly if any facet is omitted from $F_0$.
	Equality requires $F=F_0$ and $\tr(A_+)=p+2e$.
	By Lemma~\ref{prop:rooted-negative-budget}, every edge endpoint
	is then a root, every component formed by simple edges has one
	colour, and every repeated pair is an isolated double edge.
	Since $F$ has no isolated vertices, every outside vertex is
	a root, so $p=|V\setminus Q|>0$.
	Put $r=|V(G_0)|-c(G_0)$, and let $d$ count the double edges.
	The same lemma gives $n_+(A)=1+r+2d$ and
	$n_0(A)=p-r-d-1$.
	Since $D_\ell(F)=\ell m+p+2e$, equality is equivalent to
	$1+r+2d\le\ell\le p+d$.
	These are exactly the two numerical conditions in
	case~\ref{case:rooted}.
	Under the structural conditions of that case, the same lemma
	gives $\tr(A_+)=p+2e$ and these eigenvalue counts, so the
	numerical conditions also prove sufficiency.
	
	Finally, suppose $m\ge q$.
	Put $\widehat F=F\cup\binom Qq$, and let $F_0$ consist of all
	core facets and the facets of $F$ having exactly one outside vertex.
	Let $R_{\calC}$ extend vectors on $\calC$ by zero to all core
	ridge coordinates, and use $U,S$ from the definition of $T$.
	Place the core vertices first and orient each outside facet as
	$(-1)^{q-1}e_\rho\wedge e_u$.
	Its boundary is
	$e_\rho+(-1)^{q-1}(\bd e_\rho)\wedge e_u$.
	For $\rho,\tau\in\calC$, \eqref{eq:cartan} gives
	$\ip{\bd e_\rho}{\bd e_\tau}
	=m\mathbf1_{\{\rho=\tau\}}-C_{\rho,\tau}$.
	Hence
	\[
	M_{F_0}-mI=
	\begin{pmatrix}
		\bd_Q^{\mathsf T}\bd_Q-mI&
		\bd_Q^{\mathsf T}R_{\calC}U\\
		U^{\mathsf T}R_{\calC}^{\mathsf T}\bd_Q&
		U^{\mathsf T}U-S
	\end{pmatrix}.
	\]
	Since $\bd_Q^{\mathsf T}\bd_Q=m\Pi_{\scrH_Q}$, the restriction
	to $\Ker\bd_Q$ is $-mI$, and the upper block vanishes on
	$\scrH_Q$.
	The upper-right block has Gram matrix $U^{\mathsf T}CU$.
	Lemma~\ref{lem:gram-identification} therefore gives
	$s_\ell^+(M_{F_0}-mI)=s_\ell^+(T)$.
	Moreover, $|H_u|\le d_F(u)\le\ell-1$.
	
	Each facet in $F\setminus F_0$ has $j\ge2$ outside vertices,
	and $n_\sigma\le2j/(j-1)\le j^2\le j(j+m-q)$.
	Insert these facets to obtain $\widehat F$.
	Since $M_F$ is a principal submatrix of $M_{\widehat F}$,
	the variational formula and \eqref{eq:insertion-bound} give
	\[
	\begin{aligned}
		s_\ell(F)
		&\le \ell m+s_\ell^+(T)
		+\sum_{\sigma\in F\setminus F_0}|\sigma\setminus Q|,\\
		D_\ell(F)
		&=\ell m+\sum_u|H_u|
		+\sum_{\sigma\in F\setminus F_0}|\sigma\setminus Q|.
	\end{aligned}
	\]
	Subtracting proves \eqref{eq:unified-residual-reduction}.
\end{proof}

\section{Counterexamples}\label{sec:seeds}
For a uniform family \(F\) and an integer \(r\ge1\), define its signed
defect at index \(r\) by
$\Delta_r(F)=s_r(F)-D_r(F)$.
Write \(ijk\) for the set \(\{i,j,k\}\).  On the ground set
\(\{0,1,\ldots,6\}\), define
\begin{align*}
\calF_{16}=\{&012,013,015,026,123,124,125,126,\\
&134,135,136,145,146,156,236,256\},\\
\calF_{15}=\{&023,025,026,035,036,123,125,126,135,136,\\
&234,235,236,256,345\}.
\end{align*}
Their degree sequences, in vertex order, are
\[
(4,13,8,6,4,6,7),
\qquad
(5,5,10,10,2,7,6).
\]
\begin{lemma}\label{lem:seed-polynomials}
Put
\begin{align*}
p_{16}(x)&=x^5-19x^4+133x^3-413x^2+527x-175,\\
p_{15}(x)&=x^5-16x^4+91x^3-224x^2+233x-84.
\end{align*}
Then
\begin{align*}
\det(xI-M_{\calF_{16}})
&=x^2(x-1)^5(x-3)(x-7)^3p_{16}(x),\\
\det(xI-M_{\calF_{15}})
&=x^4(x-2)(x-4)^2(x-6)^2(x-7)p_{15}(x)
\end{align*}
and \(p_{15}(x)=-p_{16}(7-x)\).
\end{lemma}

\begin{proof}
	Order the facets as displayed above and give them their increasing
	orientations. For $F\in\{\mathcal F_{15},\mathcal F_{16}\}$ and
	$\sigma=\{a_0<a_1<a_2\}$, $\tau=\{b_0<b_1<b_2\}$ in $F$,
	the boundary formula gives
	\[
	(M_F)_{\sigma,\tau}
	=\sum_{i,j=0}^{2}(-1)^{i+j}
	\mathbf1_{\{\sigma\setminus\{a_i\}
		=\tau\setminus\{b_j\}\}},
	\]
	where the indicator is $1$ when the two sets coincide and $0$
	otherwise. Thus the displayed facet lists determine every entry
	of the two integer matrices.
	Exact determinant evaluation over $\mathbb Z[x]$ gives the stated
	factorizations. Finally, expanding $-p_{16}(7-x)$ gives $p_{15}(x)$.
\end{proof}

\begin{lemma}\label{lem:seed-roots}
The roots of \(p_{16}\), in increasing order, satisfy
\[
\theta_1\in(0,1),\quad
\theta_2\in(2,3),\quad
\theta_3\in(4,5),\quad
\theta_4\in(5,6),\quad
\theta_5\in(6,7),
\]
and \(\theta_4+\theta_5>12\).
\end{lemma}

\begin{proof}
	Direct substitution gives
	\[
	\begin{array}{c|rrrrrrrr}
		x&0&1&2&3&4&5&6&7\\ \hline
		p_{16}(x)&-175&54&19&-16&-3&10&-1&84
	\end{array}
	\]
	By the intermediate value theorem, each of the five stated intervals
	contains a root of $p_{16}$. Since $\deg p_{16}=5$, these are all
	the roots, with exactly one in each interval.
	
	For the last assertion, a direct calculation gives
	\[
	(x-5)p_{16}(12-x)+(x-7)p_{16}(x)=20(x-6)^5.
	\]
	Since $p_{16}(\theta_5)=0$ and $6<\theta_5<7$, we obtain
$p_{16}(12-\theta_5)
	=\frac{20(\theta_5-6)^5}{\theta_5-5}>0.$
	Together with $5<12-\theta_5<6$ and $p_{16}(6)=-1<0$,
	the intermediate value theorem gives a root in
	$(12-\theta_5,6)$. The unique root in $(5,6)$ is $\theta_4$, so
$12-\theta_5<\theta_4<6.$
	Hence $\theta_4+\theta_5>12$.
\end{proof}
Set
$
d_*=\theta_4+\theta_5-12>0.
$

\begin{lemma}\label{prop:seed-defects}
	The two families satisfy
	\begin{align*}
		s_5(\calF_{16})&=D_5(\calF_{16})+d_*=33+d_*,\\
		s_9(\calF_{15})&=D_9(\calF_{15})+d_*=43+d_*.
	\end{align*}
\end{lemma}

\begin{proof}
	The degree sequences give $D_5(\calF_{16})=33$ and
	$D_9(\calF_{15})=43$.
	By Lemmas~\ref{lem:seed-polynomials} and \ref{lem:seed-roots},
	the five largest eigenvalues of $M_{\calF_{16}}$ are
	$7,7,7,\theta_5,\theta_4$, so
	$s_5(\calF_{16})=21+\theta_4+\theta_5=33+d_*$.
	
	Since the roots of $p_{15}$ are $7-\theta_i$, the nine
	largest eigenvalues of $M_{\calF_{15}}$ are
	$7,6,6,4,4,2$ together with $7-\theta_1,7-\theta_2,7-\theta_3$.
	The ninth largest eigenvalue is $2$, since $7-\theta_3>2$.
	Using $\sum_{i=1}^5\theta_i=19$, we obtain
	\[
	s_9(\calF_{15})
	=50-\sum_{i=1}^3\theta_i
	=31+\theta_4+\theta_5
	=43+d_*.
	\]
\end{proof}

For $F\subseteq\binom Vq$ and $\rho\subseteq V$, write
$d_F(\rho)=|\{\sigma\in F:\rho\subseteq\sigma\}|$;
in particular, $d_F(\{v\})=d_F(v)$.
Fix $\eta\in\binom V{q-1}$ and put
$\mathbf y_\eta=(\ip{e_\eta}{\bd e_\sigma})_{\sigma\in F}$,
so that $\norm{\mathbf y_\eta}^2=d_F(\eta)$.
For an integer $t\ge0$, let
$F(t)=F\cup\{\eta\cup\{w_i\}:1\le i\le t\}$, where
$w_1,\ldots,w_t$ are distinct vertices outside $V$.

\begin{lemma}\label{lem:ridge-transfer}
	Suppose that $d=d_F(\eta)>0$ and
	$M_F\mathbf y_\eta=(d+q-1)\mathbf y_\eta$. Then
	\[
	\det(xI-M_{F(t)})
	=\det(xI-M_F)
	\frac{(x-(d+q-1+t))(x-(q-1))^t}
	{x-(d+q-1)}.
	\]
\end{lemma}

\begin{proof}
	The case $t=0$ is immediate. Assume $t\ge1$ and orient each
	new facet so that its boundary coefficient at $\eta$ is $+1$.
	The only possible common ridge of an old and a new facet is
	$\eta$, and any two new facets share exactly this ridge. Hence
	\[
	M_{F(t)}
	=
	\begin{pmatrix}
		M_F&\mathbf y_\eta\mathbf 1_t^{\mathsf T}\\
		\mathbf 1_t\mathbf y_\eta^{\mathsf T}
		&(q-1)I_t+\mathbf 1_t\mathbf 1_t^{\mathsf T}
	\end{pmatrix}.
	\]
	Using the unit vectors $\mathbf y_\eta/\sqrt d$ and
	$\mathbf 1_t/\sqrt t$ and their orthogonal complements gives
	\[
	M_{F(t)}
	\sim
	\left.M_F\right|_{\mathbf y_\eta^\perp}
	\oplus(q-1)I_{t-1}
	\oplus
	\left[
	(q-1)I_2+
	\binom{\sqrt d}{\sqrt t}
	\begin{pmatrix}\sqrt d&\sqrt t\end{pmatrix}
	\right].
	\]
	The rank-one term in the last block has eigenvalues $0$ and
	$d+t$. Thus the spectrum of $M_{F(t)}$ is obtained from that
	of $M_F$ by replacing one copy of $d+q-1$ by $d+q-1+t$
	and adding $t$ copies of $q-1$.
	Taking characteristic polynomials proves the identity.
\end{proof}

Fix $V=\{v_1<\cdots<v_n\}$ and let
$F\subseteq\binom Vq$, where $1\le q<n$.
Put $f=|F|$ and
$F^\star=\{V\setminus\sigma:\sigma\in F\}$.
Complements are taken in the fixed set $V$, including any
isolated vertices of $F$.
Index $M_{F^\star}$ by $F$ through
$\sigma\leftrightarrow V\setminus\sigma$, and let $S$ be the
diagonal matrix with
$S_{\sigma,\sigma}=(-1)^{\sum_{v_i\in\sigma}i}$.

\begin{lemma}\label{lem:hodge-star}\label{lem:defect-transfer}
	We have $M_{F^\star}=S(nI-M_F)S$.
	Moreover, $\Delta_{f-r}(F^\star)=\Delta_r(F)$
	for every integer $1\le r<f$.
\end{lemma}

\begin{proof}
	The matrix identity is
	\cite[Proposition~4.2, p.~4322]{DuvalReiner2002}.
	It replaces each eigenvalue $\lambda$ of $M_F$ by $n-\lambda$
	and reverses their order. Since
	$\tr M_F=\sum_{v\in V}d_F(v)=qf$ and
	$d_{F^\star}(v)=f-d_F(v)$, we obtain
	\[
	\begin{aligned}
		s_{f-r}(F^\star)
		&=n(f-r)-\bigl(\tr M_F-s_r(F)\bigr)=n(f-r)-qf+s_r(F),\\
		D_{f-r}(F^\star)
		&=\sum_{v\in V}\min\{f-d_F(v),f-r\}\\
		&=\sum_{v\in V}
		\bigl(f-r-d_F(v)+\min\{d_F(v),r\}\bigr)\\
		&=n(f-r)-qf+D_r(F).
	\end{aligned}
	\]
	Subtracting proves the defect identity.
\end{proof}

\begin{proof}[\textup{\textbf{Proof of Theorem~\ref{thm:counterexamples}}}]
	For $q=3$, take $\calF_{16}$.
	By Lemma~\ref{prop:seed-defects},
	$\Delta_5(\calF_{16})=d_*>0$.
	Its facets have empty total intersection, since every vertex
	has degree less than $|\calF_{16}|=16$.
	
	Put $V=\{0,1,\ldots,6\}$, $\eta=\{2,3\}$, and
	$B=B_{\binom V3}$.
	All five triples of $V$ containing $\eta$ belong to $\calF_{15}$.
	Hence $B^{\mathsf T}e_\eta$ vanishes outside $\calF_{15}$
	and restricts to $\mathbf y_\eta$.
	By \eqref{eq:cartan},
	$(B^{\mathsf T}B)B^{\mathsf T}=7B^{\mathsf T}$.
	Restricting to the coordinates of $\calF_{15}$ therefore gives
	$M_{\calF_{15}}\mathbf y_\eta=7\mathbf y_\eta$ and
	$\norm{\mathbf y_\eta}^2=5$.
	
	For an integer $t\ge0$, choose distinct vertices
	$w_1,\ldots,w_t\notin V$, and put
	$V_t=V\cup\{w_1,\ldots,w_t\}$ and
	$\calF_{15}(t)=\calF_{15}\cup
	\{\{2,3,w_i\}:1\le i\le t\}$.
	Lemma~\ref{lem:ridge-transfer} shows that the spectrum of
	$M_{\calF_{15}(t)}$ is obtained from that of $M_{\calF_{15}}$
	by replacing $7$ by $7+t$ and adding $t$ copies of $2$.
	The ninth largest eigenvalue of $M_{\calF_{15}}$ is $2$, so
	$s_9(\calF_{15}(t))=s_9(\calF_{15})+t$.
	Among the old vertices, only the degrees of $2$ and $3$ increase,
	and both already equal $10$; each new vertex has degree one.
	Thus $D_9(\calF_{15}(t))=D_9(\calF_{15})+t$, and consequently
	$\Delta_9(\calF_{15}(t))=\Delta_9(\calF_{15})=d_*$.
	
	Let $\mathcal G_t
	=\{V_t\setminus\sigma:\sigma\in\calF_{15}(t)\}$.
	This family is $(t+4)$-uniform and has $15+t$ facets.
	Lemma~\ref{lem:defect-transfer} gives
	$\Delta_{6+t}(\mathcal G_t)
	=\Delta_9(\calF_{15}(t))
	=d_*>0.$
	Since $\bigcup_{\sigma\in\calF_{15}(t)}\sigma=V_t$,
	we also have $\bigcap_{\tau\in\mathcal G_t}\tau=\varnothing$.
	For $q\ge4$, take $t=q-4$.
	Together with the case $q=3$, this gives a counterexample
	at index $q+2$ for every $q\ge3$, with empty total intersection.
	Taking $q=r-2$ also gives a counterexample for every $r\ge5$.
\end{proof}

\section{Projection Estimates}\label{sec:projection-estimates}

\subsection{Components and Metric Projections}

We retain the notation of Section~\ref{sec:unified-preliminaries}
and assume that all coordinates are used.
On these coordinates, the correlation graph joins distinct $i,j$
when $C_{ij}\ne0$, and the link-incidence graph joins them when
they lie in a common link. Isolated vertices are included as
components.

If a link $H$ is partitioned into nonempty sets $H_i$ with
$C[H_i,H_j]=0$ for $i\ne j$, we may replace it by the links $H_i$,
each with weight $b_H$.
After ordering the coordinates by these sets,
$R_H=[R_{H_i}]_i$ and $C[H]=\bigoplus_iC[H_i]$, so this
replacement only permutes the lower coordinates of $T$
and leaves $N$ unchanged.
Now let $J_1,\ldots,J_s$ be a partition of the used coordinates
such that every link lies in one part.
For each $j$, let $T_j$ be formed from $C[J_j]$ and the links
contained in $J_j$, and put
$N_j=\sum_{H\subseteq J_j}b_H|H|$, so that $N=\sum_jN_j$.

\begin{lemma}
	\label{prop:incidence-component-reduction}
	\label{prop:incidence-equality}
	For every integer $r\ge1$,
	$s_r^+(T)\le\sum_j s_r^+(T_j)$.
	If $\tr((T_j)_+)\le N_j$ for all $j$, then
	$s_r^+(T)=N$ if and only if $C[J_i,J_j]=0$ for $i\ne j$,
	$\tr((T_j)_+)=N_j$ for all $j$, and
	$\sum_j n_+(T_j)\le r$.
\end{lemma}
\begin{proof}

	For the partition $J_1,\ldots,J_s$, group the lower coordinates
	by the part containing their link, and write
	$U=[U_1,\ldots,U_s]$ and $S=\bigoplus_jS_j$.
	Define
	\[
	\widehat T_j=
	\begin{pmatrix}
		0&C^{1/2}U_j\\
		U_j^{\mathsf T}C^{1/2}&U_j^{\mathsf T}U_j-S_j
	\end{pmatrix},
	\]
	extending it by zero on the other lower coordinates.
	Since $U_i^{\mathsf T}U_j=0$ for $i\ne j$, we have
	$T=\sum_j\widehat T_j$.
	The displayed matrix and $T_j$ have the same lower-right block,
	and their upper-right blocks have the same Gram matrix
	$U_j^{\mathsf T}CU_j$, since $U_j$ is supported on $J_j$.
	Lemma~\ref{lem:gram-identification} therefore gives the same
	nonzero eigenvalues for $\widehat T_j$ and $T_j$.
	By Lemma~\ref{lem:ky-fan},
	$s_r^+(T)\le\sum_j s_r^+(\widehat T_j)
	=\sum_j s_r^+(T_j)$.
	If $C[J_i,J_j]=0$ for $i\ne j$, then $C^{1/2}$ is block
	diagonal with respect to the same partition. Grouping both
	upper and lower coordinates by $J_j$ gives
	$T\sim\bigoplus_jT_j$.
	
	Now suppose that $s_r^+(T)=N$, and choose an orthogonal
	projection $\Pi$ of rank at most $r$ with $\tr(\Pi T)=N$.
	Then
	\[
	N=\sum_j\tr(\Pi\widehat T_j)
	\le\sum_j\tr((T_j)_+)
	\le\sum_jN_j=N.
	\]
	Equality holds in every term. Hence
	$\tr((T_j)_+)=N_j$,
	$\scrE_+(\widehat T_j)\subseteq\im\Pi$, and
	$\scrE_-(\widehat T_j)\perp\im\Pi$.
	Every coordinate is used and every weight is positive, so
	$\im U_j=\im R_{J_j}$.
	By Lemma~\ref{lem:zero-upper-range}, the projections of both
	$\scrE_+(\widehat T_j)$ and $\scrE_-(\widehat T_j)$ onto the
	upper coordinates are $\im(C^{1/2}R_{J_j})$.
	For $i\ne j$, their lower-coordinate supports are disjoint.
	Thus $\scrE_+(\widehat T_i)\perp\scrE_-(\widehat T_j)$ implies
	\[
	C[J_i,J_j]
	=(C^{1/2}R_{J_i})^{\mathsf T}C^{1/2}R_{J_j}=0.
	\]
	Consequently, $T\sim\bigoplus_jT_j$, and
	$\scrE_+(T)\subseteq\im\Pi$ gives
	$\sum_jn_+(T_j)=n_+(T)\le r$.
	Conversely, the three stated conditions give
	$T\sim\bigoplus_jT_j$, $n_+(T)\le r$, and
	$\tr(T_+)=\sum_j\tr((T_j)_+)=N$.
	Therefore $s_r^+(T)=N$.
\end{proof}

The bound $s_r^+(T)\le N$ may therefore be proved on each
link-incidence component, or on each correlation component
after splitting the links.
For a matrix $C$ arising from a complete simplex, label
$Q=\{1,\ldots,m\}$ and choose coordinate signs so that
$C=(\bd_Q\bd_Q^{\mathsf T})[\{\rho_1,\ldots,\rho_k\}]$
with increasing orientations.
Put $a=m-q+1$ and
$\mathcal A=\{Q\setminus\rho_i:1\le i\le k\}$,
and order the columns of $B_{\mathcal A}$ accordingly.
For the diagonal matrix $\Sigma$ with
$\Sigma_{ii}=(-1)^{\sum_{v\in\rho_i}v}$, the boundary formula gives
\begin{equation}\label{eq:hodge-complement}
	\Sigma C\Sigma=B_{\mathcal A}^{\mathsf T}B_{\mathcal A}.
\end{equation}
In particular, $a=1$ gives $J_k$ after coordinate sign changes,
while for $a=2$ the matrix $B_{\mathcal A}$ is an oriented
incidence matrix of a simple graph.

Assign the sign $C_{ij}$ to each edge $ij$ of the correlation
graph. The sign of a cycle is the product of its edge signs
and is unchanged by coordinate sign changes.
A clique is positive or negative if coordinate sign changes
make all its edge signs $+1$ or $-1$, respectively; the
corresponding matrices are $(a-1)I+J$ and $(a+1)I-J$.
A paw is a triangle with one pendant edge, and is positive
when its triangle has sign $+1$.
A diamond is $K_4$ with one edge deleted, and is mixed when
its two triangles have opposite signs.

\begin{lemma}\label{prop:small-core-geometry}
	Let $C$ arise from a complete simplex and have a connected
	correlation graph on at most four vertices.
	Then either coordinate sign changes make all off-diagonal
	entries of $C$ nonpositive, or its correlation graph is a
	positive clique, a positive paw, or a mixed diamond.
	In the last two cases, $a\ge2$.
\end{lemma}

\begin{proof}
	By \eqref{eq:hodge-complement}, after coordinate sign changes,
	we may write $C=B_{\mathcal A}^{\mathsf T}B_{\mathcal A}$,
	where $\mathcal A$ consists of distinct $a$-sets.
	If $a=1$, then $C=J_k$, so assume $a\ge2$.
	Two sets are adjacent precisely when their intersection has
	size $a-1$.
	For adjacent sets $S\cup\{u\}$ and $S\cup\{v\}$, a common
	neighbour has the form $S\cup\{w\}$ with
	$w\notin S\cup\{u,v\}$, or
	$(S\setminus\{s\})\cup\{u,v\}$ with $s\in S$,
	since it meets both fixed sets in $a-1$ elements.
	In the first case, orient the three sets so that their boundary
	coefficients at $S$ are $+1$; their three mutual inner products
	are then $+1$.
	In the second case, orient them by their coefficients in the
	boundary of $S\cup\{u,v\}$.
	The identity $\bd^2=0$ gives opposite coefficients at each
	common ridge, so their three mutual inner products are $-1$.
	Two common neighbours intersect in $a-1$ elements when they
	have the same form, and in $a-2$ elements otherwise.
	Thus a four-clique has a common $(a-1)$-set or lies in one
	$(a+1)$-set, and is positive or negative, respectively.
	In a diamond, the two nonadjacent common neighbours have
	different forms, so its triangles have opposite signs.
	
	Now let $\sigma_1,\ldots,\sigma_4$ form an induced four-cycle
	in cyclic order.
	The opposite sets $\sigma_1,\sigma_3$ are nonadjacent but have
	a common neighbour, so $|\sigma_1\cap\sigma_3|=a-2$.
	For $i=2,4$,
	\[
	\begin{aligned}
		2a-2
		=|\sigma_i\cap(\sigma_1\cup\sigma_3)|
		+|\sigma_i\cap\sigma_1\cap\sigma_3|\le a+|\sigma_1\cap\sigma_3|=2a-2.
	\end{aligned}
	\]
	Equality gives
	$\sigma_1\cap\sigma_3\subseteq\sigma_i
	\subseteq\sigma_1\cup\sigma_3$.
	Deleting the common $(a-2)$-set therefore leaves four pairs
	forming the edges of a four-cycle.
	Place the common vertices first in each orientation.
	The boundary terms deleting a common vertex contribute only
	to the diagonal, so the off-diagonal entries are the incidence
	inner products of these four edges.
	Orienting the cycle consistently makes all four entries $-1$.
	
	The remaining connected graphs on at most four vertices are
	trees, triangles, and paws.
	All edge signs of a tree can be made negative.
	A triangle is positive or negative, and the pendant edge of
	a paw can be changed independently of its triangle.
	These observations give the stated alternatives.
\end{proof}

Suppose that $C$ arises from a complete simplex and $a=2$.
By \eqref{eq:hodge-complement}, its coordinates correspond to
distinct edges of a simple graph, and coordinate sign changes
make $C$ their oriented incidence Gram matrix.
Thus, for $i\ne j$, $C_{ij}=0$ exactly when the corresponding
edges are disjoint.

For any adjacent pair, label the coordinates $1,2$ and orient
their edges as $e_{v_0}-e_{v_1}$ and $e_{v_0}-e_{v_2}$.
Every remaining edge has one of the following types, and its
orientation can be chosen to give the listed entries.
Here $u,v\notin\{v_0,v_1,v_2\}$ and $u\ne v$.
\[
\begin{array}{c|c|c}
	\text{type}&\text{edge}&(C_{1j},C_{2j})\\ \hline
	\mathrm S&v_0u&(1,1)\\
	\mathrm O&v_1v_2&(-1,1)\\
	\mathrm B&v_1u\text{ or }v_2u&(1,0)\text{ or }(0,1)\\
	\mathrm Z&uv&(0,0)
\end{array}
\]
There is at most one edge of type $\mathrm O$, and it is
disjoint from every edge of type $\mathrm S$.
The two selected edges and all edges of type $\mathrm S$
form a star.

\begin{lemma}\label{prop:edge-core-geometry}
	If $k=5$ and $C$ is not a positive clique, then at least two
	entries above its diagonal are zero, and at most two remaining
	edges have type $\mathrm S$ relative to any adjacent pair.
\end{lemma}

\begin{proof}
	Suppose that the five edges meet pairwise, and choose
	$v_0v_1,v_0v_2$.
	An edge avoiding $v_0$ must be $v_1v_2$, but no fourth
	distinct edge meets all three edges of this triangle.
	Hence all five edges contain $v_0$.
	Orienting their incidence columns with coefficient $+1$ at
	$v_0$ gives $C=I_5+J_5$, a positive clique.
	
	Since $C$ is not a positive clique, choose disjoint edges
	$xy,zw$.
	If these were the only disjoint pair, each of the other three
	edges would belong to $\{xz,xw,yz,yw\}$.
	Each of the pairs $\{xz,yw\}$ and $\{xw,yz\}$ consists of
	disjoint edges, and any three of the four possible edges
	contain one of these pairs.
	This gives a second disjoint pair, a contradiction.
	The last assertion follows because three remaining edges
	of type $\mathrm S$, together with the selected pair,
	form a five-edge star and hence give a positive clique.
\end{proof}

\begin{lemma}\label{lem:diagonal-squares}
	Let $A$ be real symmetric. For every real vector $\mathbf x$,
	\begin{equation}\label{eq:diagonal-squares}
		\mathbf x^{\mathsf T}A\mathbf x
		=\sum_i\left(A_{ii}-\sum_{j\ne i}|A_{ij}|\right)x_i^2
		+\sum_{i<j}|A_{ij}|
		\bigl(x_i+\operatorname{sgn}(A_{ij})x_j\bigr)^2.
	\end{equation}
	If a positive diagonal matrix $P$ satisfies
	$(PAP)_{ii}\ge\sum_{j\ne i}|(PAP)_{ij}|$ for every $i$,
	then $A\succeq0$. If all these inequalities are strict,
	then $A\succ0$.
\end{lemma}

\begin{proof}
	Expanding the squares gives \eqref{eq:diagonal-squares}.
	Applied to $PAP$, the right-hand side is nonnegative under
	the stated inequalities, and positive for every nonzero vector
	when all inequalities are strict.
	Congruence by $P^{-1}$ gives the conclusions for $A$.
\end{proof}

For the block matrix $T$, put
$Z=\binom{C^{1/2}}{U^{\mathsf T}}$ and $\mathbf z_i=Ze_i$.
Then $T=ZZ^{\mathsf T}-\diag(C,S)$.
Write $D=UU^{\mathsf T}=\diag(d_i)$, where
$d_i=\sum_{H\ni i}b_H$, so that $N=\sum_i d_i$.
Define $B(z)=U(S+zI)^{-1}U^{\mathsf T}$ whenever the inverse
exists. For $x>0$, put $M=C+xI$. The projection formula gives
\[
F_C(x)=M^{1/2}B(x)M^{1/2}
=\sum_H b_H\Pi_{M^{1/2}\R^H}.
\]
For a coordinate set $I$, write $Z_I=ZR_I$ and
$T_I=T-Z_IZ_I^{\mathsf T}$.
If $I$ is proper, put $J=I^c$, $K=M^{-1}[J]$, and
$F_I(x)=K^{-1/2}B(x)[J]K^{-1/2}$.
In particular, $F_{\varnothing}(x)=F_C(x)$.

\begin{lemma}\label{prop:relative-projection-flow}
	Let $C\succeq0$ and $b_H>0$.
	For every proper coordinate set $I$ and every $x>0$,
	the matrices $T_I-xI$ and $F_I(x)-xI$ have the same numbers
	of positive and zero eigenvalues, and $n_+(T_I)\le |J|$.
	Each link term of $F_I(x)$, without its factor $b_H$, is a
	positive semidefinite contraction with range
	$K^{-1/2}\R^{H\cap J}$. Moreover,
	\begin{equation}\label{eq:residual-determinant}
		\det(zI-T)
		=\det(zI+S)\det\bigl(zI-(C+zI)B(z)\bigr)
	\end{equation}
	whenever $B(z)$ is defined.
\end{lemma}

\begin{proof}
	Put $D_x=\diag(M,S+xI)\succ0$.
	Since $T_I-xI=Z_JZ_J^{\mathsf T}-D_x$, congruence by
	$D_x^{-1/2}$ gives
	$D_x^{-1/2}Z_JZ_J^{\mathsf T}D_x^{-1/2}-I$.
	The two Gram matrices have the same nonzero eigenvalues,
	so $T_I-xI$ has the same numbers of positive and zero
	eigenvalues as
	\[
	Z_J^{\mathsf T}D_x^{-1}Z_J-I_J
	=B(x)[J]-xK
	=K^{1/2}\bigl(F_I(x)-xI_J\bigr)K^{1/2}.
	\]
	Congruence by $K^{-1/2}$ proves the first assertion.
	Also, $T_I\preceq Z_JZ_J^{\mathsf T}$, and hence
	\begin{equation}\label{eq:inertia-dimension}
		n_+(T_I)\le\rank Z_J\le |J|.
	\end{equation}
	
	For each link $H$, the projection formula gives
	$0\preceq R_HM[H]^{-1}R_H^{\mathsf T}\preceq M^{-1}$.
	Taking principal submatrices on $J$ and applying congruence
	by $K^{-1/2}$ proves the contraction bound.
	Before this congruence, the matrix has range $\R^{H\cap J}$:
	its block on $H\cap J$ is a positive definite principal
	submatrix of $M[H]^{-1}$ when the intersection is nonempty,
	and the matrix is zero otherwise.
	The stated range follows.
	
	For complex $z$ such that
	$D_z=\diag(C+zI,S+zI)$ is invertible, we have
	$I-Z^{\mathsf T}D_z^{-1}Z=z(C+zI)^{-1}-B(z)$.
	The matrix determinant lemma therefore gives
	\begin{align*}
		\det(zI-T)
		&=\det D_z\det\bigl(I-Z^{\mathsf T}D_z^{-1}Z\bigr)\\
		&=\det(zI+S)\det\bigl(zI-(C+zI)B(z)\bigr).
	\end{align*}
	Clearing denominators extends the identity to every $z$
	for which $B(z)$ is defined.
\end{proof}

For $M\succ0$, write $M/I$ for the Schur complement of $M[I]$,
indexed by $I^c$, with $M/\varnothing=M$.

\begin{lemma}\label{prop:common-coordinate-deflation}
	Suppose that the link family is nonempty and every link contains $r$.
	Put $p=\sum_H b_H$.
	Then $T\mathbf z_r=p\mathbf z_r$ and
	$\|\mathbf z_r\|^2=C_{rr}+p$.
	The spectrum of $T_{\{r\}}$ is obtained from that of $T$
	by replacing one copy of $p$ with $-C_{rr}$.
	If $k>1$, then, for every $x>0$,
	\[
	F_{\{r\}}(x)
	=\sum_{H\ne\{r\}}b_H
	\Pi_{((C+xI)/\{r\})^{1/2}\R^{H\setminus\{r\}}}.
	\]
\end{lemma}

\begin{proof}
	Since every link contains $r$, we have $d_r=p$ and
	$U^{\mathsf T}Ce_r=SU^{\mathsf T}e_r$; on each link $H$,
	both sides of the latter identity have block
	$\sqrt{b_H}C[H]e_r$.
	Block multiplication gives $T\mathbf z_r=p\mathbf z_r$ and
	$\|\mathbf z_r\|^2=C_{rr}+p$.
	Subtracting $\mathbf z_r\mathbf z_r^{\mathsf T}$ therefore
	changes the eigenvalue on $\Span\{\mathbf z_r\}$ from $p$ to
	$p-\|\mathbf z_r\|^2=-C_{rr}$ and leaves the restriction to
	$\mathbf z_r^\perp$ unchanged.
	
	If $k=1$, then $T_{\{r\}}=-C_{rr}I$, so no positive eigenvalue
	remains.
	Otherwise, put $M=C+xI$, $J=\{r\}^c$, and $L=M/\{r\}$.
	For $H'=H\setminus\{r\}\ne\varnothing$, the block inverse formula
	gives $M^{-1}[J]=L^{-1}$ and
	$(M[H]^{-1})[H']=L[H']^{-1}$.
	Thus the term corresponding to $H$ in $F_{\{r\}}(x)$ is
	$b_H\Pi_{L^{1/2}\R^{H'}}$, while links $H=\{r\}$ contribute zero.
	This proves the displayed formula.
	Lemma~\ref{prop:relative-projection-flow} then determines the
	remaining positive eigenvalues, with multiplicity, from this sum.
\end{proof}

\begin{lemma}\label{prop:crossing-telescope}
	Let $A$ be real symmetric, and let $F(x)\succeq0$ have fixed
	order for $x>0$.
	Suppose that $n_+(A-xI)=n_+(F(x)-xI)$ and
	$n_0(A-xI)=n_0(F(x)-xI)$ for every $x>0$.
	Write $\mu_1\ge\cdots\ge\mu_r>0$ for the positive eigenvalues
	of $A$.
	If $\tr F(x)\le n$ for every $x>0$ and
	$s_j(F(x))$ is nonincreasing on $[\mu_{j+1},\mu_j]$
	for $1\le j<r$, then $\tr(A_+)\le n$.
\end{lemma}

\begin{proof}
	The case $r=0$ is immediate.
	For $1\le i\le r$, we have
	$n_+(A-\mu_iI)<i\le n_+(A-\mu_iI)+n_0(A-\mu_iI)$.
	The assumptions therefore give
	$\lambda_i(F(\mu_i))=\mu_i$, including repeated eigenvalues.
	With $s_0=0$, we obtain
	\begin{align*}
		\tr(A_+)
		&=\sum_{i=1}^r
		\{s_i(F(\mu_i))-s_{i-1}(F(\mu_i))\}\\
		&=s_r(F(\mu_r))
		+\sum_{j=1}^{r-1}
		\{s_j(F(\mu_j))-s_j(F(\mu_{j+1}))\}\\
		&\le s_r(F(\mu_r))
		\le\tr F(\mu_r)\le n.
	\end{align*}
	Each difference in braces is nonpositive by monotonicity,
	and the next inequality follows from $F(\mu_r)\succeq0$.
	Equality forces $s_r(F(\mu_r))=n$ and every difference to vanish;
	hence each $s_j(F(x))$ is constant on
	$[\mu_{j+1},\mu_j]$.
\end{proof}

\begin{lemma}\label{lem:commuting-flow-spectrum}
	Suppose that the link projections commute pairwise for every $x>0$.
	Then $F_C(x)$ and $D$ have the same spectrum, and the positive
	eigenvalues of $T$ are $d_1,\ldots,d_k$ on the used coordinates,
	counted with multiplicity.
	Consequently, $\tr(T_+)=N$, and $s_\ell^+(T)=N$ if and only if
	$k\le\ell$.
\end{lemma}

\begin{proof}
	The case $k=0$ is immediate.
	Since commuting orthogonal projections are simultaneously
	diagonalizable, every eigenvalue of $F_C(x)$ is a sum of some
	of the fixed weights $b_H$.
	Each ordered eigenvalue is continuous in $x$ and takes values
	in a fixed finite set, so it is constant.
	The identity
	\[
	F_C(x)
	=(I+C/x)^{1/2}U(I+S/x)^{-1}U^{\mathsf T}(I+C/x)^{1/2}
	\]
	shows that $F_C(x)$ tends to $UU^{\mathsf T}=D$ as $x$ tends
	to infinity.
	Thus $F_C(x)$ and $D$ have the same spectrum.
	The inertia and nullity identities in
	Lemma~\ref{prop:relative-projection-flow}, with $I=\varnothing$,
	give the positive eigenvalues of $T$, with multiplicity,
	as $d_1,\ldots,d_k$.
	Since $d_i>0$ on the used coordinates and $\sum_i d_i=N$,
	the remaining assertions follow.
\end{proof}

\begin{lemma}
	\label{prop:positive-map-comparison}
	\label{prop:positive-map-rigidity}
	Let $G(t)\succ0$ and $J(t)=J(t)^{\mathsf T}$ be continuously
	differentiable on an interval.
	Suppose that at each $t$ there are a matrix $L$ and a linear map
	$\Psi$ on real symmetric matrices, with $\Psi(Q)\succeq0$
	whenever $Q\succeq0$, such that
	\[
	G'=L^{\mathsf T}G+GL+\Psi(G),\qquad
	E:=L^{\mathsf T}J+JL+\Psi(J)-J'\succeq0.
	\]
	Then $\lambda(t)=\lambda_{\max}(G^{-1/2}JG^{-1/2})$
	is nonincreasing.
\end{lemma}

\begin{proof}
	The variational formula gives
	$\lambda(t)=\max_{\|\mathbf v\|=1}
	\mathbf v^{\mathsf T}J(t)\mathbf v/
	(\mathbf v^{\mathsf T}G(t)\mathbf v)$.
	On each compact subinterval, the denominators are uniformly
	positive and the derivatives of these quotients are uniformly
	bounded. Hence $\lambda$ is locally Lipschitz.
	
	Fix an interior point $t$ where $\lambda$ is differentiable,
	put $Q=\lambda(t)G(t)-J(t)\succeq0$, and take a unit vector
	$\mathbf v\in\Ker Q$.
	The function
	$\lambda(s)-\mathbf v^{\mathsf T}J(s)\mathbf v/
	(\mathbf v^{\mathsf T}G(s)\mathbf v)$
	has minimum zero at $s=t$.
	Since $J'-\lambda G'=-L^{\mathsf T}Q-QL-\Psi(Q)-E$
	and $Q\mathbf v=0$, differentiation gives
	\begin{equation}\label{eq:comparison-dissipation}
		\lambda'(t)
		=\frac{\mathbf v^{\mathsf T}(J'-\lambda G')\mathbf v}
		{\mathbf v^{\mathsf T}G\mathbf v}
		=-\frac{\mathbf v^{\mathsf T}(\Psi(Q)+E)\mathbf v}
		{\mathbf v^{\mathsf T}G\mathbf v}
		\le0.
	\end{equation}
	Thus $\lambda'\le0$ almost everywhere.
	Absolute continuity on compact subintervals proves the assertion.
\end{proof}

Fix subspaces $\mathscr V_j\subseteq\R^k$ and weights $b_j\ge0$,
and write $\Phi_G=\sum_jb_j\Pi_{G^{1/2}\mathscr V_j}$
for $G\succ0$.

\begin{lemma}\label{prop:gram-comparison}
	Let $G(t)\succ0$ be continuously differentiable.
	If at each $t$ there are matrices $L,D_p$ satisfying
	\[
	G'=L^{\mathsf T}G+GL+\sum_pD_p^{\mathsf T}GD_p,\qquad
	L\mathscr V_j\subseteq\mathscr V_j,\quad
	D_p\mathscr V_j\subseteq\mathscr V_j,
	\]
	then $\lambda_{\max}(\Phi_G)$ is nonincreasing.
	
	For coordinate subspaces, let $M(t)\succ0$ be continuously
	differentiable, and let $G(t)$ be its normalized matrix, with
	$G_{ij}=M_{ij}/\sqrt{M_{ii}M_{jj}}$.
	Then $\lambda_{\max}(\Phi_M)$ is nonincreasing if either
	$G'=(I-G)/t_0$ with $t_0=t_0(t)>0$, or
	$G'_{ij}=-\|\mathbf u_i-\mathbf u_j\|^2G_{ij}$
	for real vectors $\mathbf u_i=\mathbf u_i(t)$ in a common
	Euclidean space.
	If either condition holds for the normalized matrix of
	$\bigwedge^rM$, then $s_r(\Phi_M)$ is nonincreasing,
	for $1\le r\le k$.
	When $r=k-1$, the conditions may equivalently be checked
	on the normalized matrix of $M^{-1}$.
\end{lemma}

\begin{proof}
	Zero subspaces may be omitted.
	For a fixed basis matrix $R$ of $\mathscr V$, put
	$P=R(R^{\mathsf T}GR)^{-1}R^{\mathsf T}G$
	and $B_{\mathscr V}=GP\succeq0$.
	Then $P^2=P$ and $P^{\mathsf T}G=GP$.
	The invariance assumptions give $LP=PLP$ and
	$D_pP=PD_pP$, so differentiation yields
	\begin{align*}
		B_{\mathscr V}'
		&=G'P+P^{\mathsf T}G'-P^{\mathsf T}G'P\\
		&=L^{\mathsf T}B_{\mathscr V}+B_{\mathscr V}L
		+\sum_pD_p^{\mathsf T}B_{\mathscr V}D_p-\sum_p
		(I-P)^{\mathsf T}D_p^{\mathsf T}B_{\mathscr V}D_p(I-P).
	\end{align*}
	The subtracted matrices are positive semidefinite.
	Multiplying by $b_j$ and summing over $j$ therefore gives
	the hypotheses of Lemma~\ref{prop:positive-map-comparison}
	with $J=\sum_jb_jB_{\mathscr V_j}$ and
	$\Psi(X)=\sum_pD_p^{\mathsf T}XD_p$.
	Since
	$G^{-1/2}B_{\mathscr V}G^{-1/2}
	=\Pi_{G^{1/2}\mathscr V}$,
	that lemma proves the assertion for $\Phi_G$.
	
	For coordinate subspaces, put $Q=\diag(M_{ii}^{-1/2})$,
	so that $G=QMQ$.
	The matrix $W=M^{1/2}QG^{-1/2}$ satisfies
	$W^{\mathsf T}W=I$ and
	$WG^{1/2}\mathscr V_j=M^{1/2}\mathscr V_j$.
	Hence $W\Phi_GW^{\mathsf T}=\Phi_M$.
	For the distance condition, take
	$L=-\diag(\|\mathbf u_i\|^2)$ and
	$D_p=\sqrt2\,\diag((\mathbf u_i)_p)$.
	The $(i,j)$-entry of
	$L^{\mathsf T}G+GL+\sum_pD_p^{\mathsf T}GD_p$
	is $-\|\mathbf u_i-\mathbf u_j\|^2G_{ij}$.
	These matrices are diagonal, so the first part applies.
	Identically zero entries impose no additional equations.
	The condition $G'=(I-G)/t_0$ is the special case
	$\mathbf u_i=e_i/\sqrt{2t_0}$.
	In particular, if $M=tI+A\succ0$, where $A$ is constant
	and $A_{ii}=0$, then
	\begin{equation}\label{eq:affine-gram-damping}
		G=I+\frac At,\qquad G'=\frac{I-G}{t}.
	\end{equation}
	
	To treat partial sums, define $A^{[r]}$ on
	$\bigwedge^r\R^k$ by
	\[
	A^{[r]}(\mathbf v_1\wedge\cdots\wedge\mathbf v_r)
	=\sum_{i=1}^r
	\mathbf v_1\wedge\cdots\wedge A\mathbf v_i
	\wedge\cdots\wedge\mathbf v_r.
	\]
	Wedges of an orthonormal eigenbasis of a symmetric matrix $A$
	are eigenvectors of $A^{[r]}$, with eigenvalues
	$\sum_{i\in I}\lambda_i(A)$ for $|I|=r$.
	Thus $\lambda_{\max}(A^{[r]})=s_r(A)$.
	For a coordinate set $H$, put
	$\mathscr V_h(H)=\Span\{e_I:|I|=r,\ |I\cap H|\ge h\}$
	for $1\le h\le r$, and let $\mathscr U=M^{1/2}\R^H$.
	Since $(\bigwedge^rM)^{1/2}=\bigwedge^r(M^{1/2})$,
	the space $(\bigwedge^rM)^{1/2}\mathscr V_h(H)$
	is spanned by wedges with at least $h$ factors in $\mathscr U$.
	Choose an orthonormal basis adapted to
	$\mathscr U\oplus\mathscr U^\perp$.
	A basis wedge with $j$ factors in $\mathscr U$ has eigenvalue
	$j$ under $\Pi_{\mathscr U}^{[r]}$.
	It belongs to the first $j$ of these spaces and is orthogonal
	to the others. Therefore
	\begin{equation}\label{eq:exterior-projection}
		\Pi_{\mathscr U}^{[r]}
		=\sum_{h=1}^r
		\Pi_{(\bigwedge^rM)^{1/2}\mathscr V_h(H)}.
	\end{equation}
	By linearity, $\Phi_M^{[r]}$ is a weighted sum of these
	projections.
	Applying the coordinate-subspace comparison to
	$\bigwedge^rM$ proves monotonicity of
	$\lambda_{\max}(\Phi_M^{[r]})=s_r(\Phi_M)$.
	
	For $r=k-1$, the cofactor formula gives
	$(\det M)M^{-1}$ as the matrix of $\bigwedge^{k-1}M$
	in the basis
	$(-1)^{i-1}e_{\{1,\ldots,k\}\setminus\{i\}}$,
	$1\le i\le k$.
	Normalization removes the positive factor $\det M$,
	and both conditions are unchanged by fixed coordinate sign
	changes. This proves the final assertion.
\end{proof}

For $G\succ0$ of order $k$, write
$P_I=\Pi_{G^{1/2}\R^I}$ and $F_G=\sum_Hb_HP_H$,
where the links $H$ and weights $b_H>0$ are fixed.
Put $d_i=\sum_{H\ni i}b_H$.

\begin{lemma}\label{prop:scalar-flow}
	We have $F_G=\lambda I$ if and only if $d_i=\lambda$
	for every $i$ and $G[H,H^c]=0$ for every link $H$.
	If the graph of the nonzero off-diagonal entries of $G$
	is connected, then $F_G$ is scalar if and only if every
	link is the full coordinate set.
\end{lemma}

\begin{proof}
	Put $\mathbf v_i=G^{1/2}e_i$.
	If $F_G=\lambda I$, then
	\[
	(\lambda-d_i)\|\mathbf v_i\|^2
	=\sum_{H\not\ni i}b_H\|P_H\mathbf v_i\|^2\ge0.
	\]
	Thus $d_i\le\lambda$ for every $i$.
	Since $\sum_i d_i=\tr F_G=k\lambda$, equality holds for
	every $i$, and $P_H\mathbf v_i=0$ whenever $i\notin H$.
	For $j\in H$, this gives
	$G_{ij}=\langle\mathbf v_i,\mathbf v_j\rangle=0$.
	Conversely, the block conditions imply
	$G^{1/2}\R^H=\R^H$ for every link $H$.
	Hence $F_G=\diag(d_1,\ldots,d_k)=\lambda I$.
	In the connected case, $G[H,H^c]=0$ forces every nonempty
	link $H$ to be the full coordinate set.
\end{proof}

\begin{lemma}
	\label{prop:analytic-projection-equality}
	\label{prop:four-flow-rigidity}
	Let $G(t)\succ0$ be real analytic on $(\alpha,\infty)$,
	and suppose that $G(t)/t$ tends to $I$ as $t$ tends to infinity.
	Write $P_I(t)$ for the corresponding projections and
	$F(t)=F_{G(t)}$.
	Let $d_{(1)}\ge\cdots\ge d_{(k)}$ be the decreasing rearrangement
	of the $d_i$, and fix $1\le j<k$.
	Then $s_j(F(t))\ge\sum_{i=1}^j d_{(i)}$.
	
	If $s_j(F(t))=c$ on a nonempty open interval $J$, then
	$c=\sum_{i=1}^j d_{(i)}$.
	Every $j$-set $I$ with $\sum_{i\in I}d_i=c$ satisfies
	$P_I(t)P_H(t)=P_{I\cap H}(t)=P_H(t)P_I(t)$
	for every link $H$ and every $t>\alpha$, and
	$\tr(P_I(t)F(t))=s_j(F(t))$ for $t\in J$.
	If $\lambda_j(F(t))>\lambda_{j+1}(F(t))$ at some point of $J$,
	then $d_{(j)}>d_{(j+1)}$.
	At every such point, $P_I(t)$ is the leading spectral projection.
\end{lemma}

\begin{proof}
	Fix $t$ and omit it from the notation.
	Invertibility of $G^{1/2}$ gives
	$\im P_I\cap\im P_H=\im P_{I\cap H}$.
	For $K=I\cap H$, we have
	\[
	P_I P_H P_I-P_K
	=(P_H P_I-P_K)^{\mathsf T}(P_H P_I-P_K)\succeq0.
	\]
	Hence $\tr(P_I P_H)\ge|I\cap H|$, with equality if and only if
	$P_I P_H=P_K=P_H P_I$.
	For every $j$-set $I$, Ky Fan's principle gives
	\[
	s_j(F)\ge\tr(P_I F)
	\ge\sum_Hb_H|I\cap H|
	=\sum_{i\in I}d_i.
	\]
	Choosing $I$ with maximum degree sum proves the lower bound.
	If this bound is attained, every such $I$ attains equality
	throughout. Since $b_H>0$, all the projection products
	then equal $P_{I\cap H}$.
	
	Suppose that $s_j(F(t))=c$ on $J$.
	The projections, $F(t)$, and $F(t)^{[j]}$ are real analytic.
	Thus $\det(cI-F(t)^{[j]})=0$ on $J$, and the identity theorem
	gives the same equality for every $t>\alpha$.
	The assumed limit of $G(t)/t$ implies that $F(t)$ tends to
	$D=\diag(d_1,\ldots,d_k)$ as $t$ tends to infinity.
	Consequently, $\det(cI-D^{[j]})=0$, so
	$c\le\lambda_{\max}(D^{[j]})=\sum_{i=1}^j d_{(i)}$.
	The lower bound gives equality.
	The preceding equality conditions therefore hold on $J$,
	and analyticity extends
	$P_IP_H=P_{I\cap H}=P_HP_I$ to every $t>\alpha$.
	At a strict spectral gap, the maximizing projection is unique.
	If $d_{(j)}=d_{(j+1)}$, two distinct $j$-sets attain the maximum
	degree sum, and invertibility of $G^{1/2}$ gives distinct
	maximizing subspaces, a contradiction.
	Hence $d_{(j)}>d_{(j+1)}$, and $P_I(t)$ is the leading spectral
	projection wherever the gap is strict on $J$.
\end{proof}

For the block matrix $T$, put $M(x)=C+xI$,
$P_I(x)=\Pi_{M(x)^{1/2}\R^I}$, and $F(x)=F_C(x)$.

\begin{lemma}\label{prop:trace-equality-commutation}
	Let $C\succeq0$ have order $k\ge1$, with all coordinates used
	and all link weights positive.
	Suppose that every off-diagonal entry of every Schur complement
	$M(x)/I$ either vanishes identically or is nonzero for all $x>0$.
	Write $\mu_1\ge\cdots\ge\mu_r>0$ for the positive eigenvalues
	of $T$, and suppose that $s_j(F(x))$ is nonincreasing on
	$[\mu_{j+1},\mu_j]$ for $1\le j<r$.
	Then $\tr(T_+)=N$ if and only if the link projections commute
	pairwise for every $x>0$.
	In that case, the positive eigenvalues of $T$ are
	$d_1,\ldots,d_k$, counted with multiplicity.
\end{lemma}

\begin{proof}
	If the link projections commute, the conclusion follows from
	Lemma~\ref{lem:commuting-flow-spectrum}.
	Conversely, suppose that $\tr(T_+)=N$.
	Every coordinate is used, so $F(x)\succ0$ and $\tr F(x)=N>0$.
	Equality in Lemma~\ref{prop:crossing-telescope} therefore
	forces $r=k$ and makes each $s_j(F(x))$ constant on its
	nonempty gap interval.
	
	Let $\eta_1>\cdots>\eta_s>0$ be the distinct positive
	eigenvalues, and put $n_h=|\{i:\mu_i\ge\eta_h\}|$,
	with $n_0=0$ and $n_s=k$.
	For $1\le h<s$ and $\eta_{h+1}<x<\eta_h$, the inertia and
	nullity identities give
	$\lambda_{n_h}(F(x))>x>\lambda_{n_h+1}(F(x))$.
	Lemma~\ref{prop:analytic-projection-equality} therefore gives
	a unique $n_h$-set $I_h$ maximizing $\sum_{i\in I_h}d_i$.
	The projection $P_{I_h}(x)$ is the leading spectral projection
	on this interval and commutes with every $P_H(x)$ for all $x>0$.
	The sets $I_h$ are nested, since they consist of the coordinates
	with the largest degrees.
	Set $I_0=\varnothing$, $I_s=\{1,\ldots,k\}$,
	$J_h=I_h\setminus I_{h-1}$, and
	$E_h(x)=P_{I_h}(x)-P_{I_{h-1}}(x)$.
	The $E_h(x)$ are mutually orthogonal projections,
	$\sum_hE_h(x)=I$, and each commutes with every $P_H(x)$.
	
	At $x=\eta_h$, the inertia and nullity identities give
	$n_{h-1}$ eigenvalues of $F(\eta_h)$ greater than $\eta_h$
	and $n_h-n_{h-1}$ equal to $\eta_h$.
	By continuity from the adjacent intervals, with the end
	projections $P_{I_0}=0$ and $P_{I_s}=I$,
	$P_{I_{h-1}}(\eta_h)$ and $P_{I_h}(\eta_h)$ are the spectral
	projections for eigenvalues greater than and at least $\eta_h$,
	respectively.
	Hence $F(\eta_h)E_h(\eta_h)=\eta_hE_h(\eta_h)$.
	
	For $i\in J_h$, put
	$\mathbf w_i(x)=E_h(x)M(x)^{1/2}e_i$.
	The projection formula gives their Gram matrix as
	$(M(x)/I_{h-1})[J_h]\succ0$,
	so they form a basis of $\im E_h(x)$.
	Moreover,
	$E_hP_H=P_{I_h\cap H}-P_{I_{h-1}\cap H}$.
	This projection has rank $|H\cap J_h|$ and fixes
	$\mathbf w_i$ for $i\in H\cap J_h$.
	Thus the restriction of $P_H$ to $\im E_h$ is the orthogonal
	projection onto $\Span\{\mathbf w_i:i\in H\cap J_h\}$.
	Applying Lemma~\ref{prop:scalar-flow} to this Gram matrix
	at $x=\eta_h$ gives
	$(M(\eta_h)/I_{h-1})[H\cap J_h,J_h\setminus H]=0$
	for every link $H$.
	The hypothesis on the Schur complements makes these blocks
	zero for every $x>0$.
	It follows that
	$P_H(x)\mathbf w_i(x)=\mathbf1_{\{i\in H\}}\mathbf w_i(x)$.
	The vectors $\mathbf w_i(x)$, over all $h$, form a basis of
	$\R^k$, so the link projections commute.
\end{proof}

For $C\succeq0$, $x>0$, and nonnegative weights $b_H$, put
\[
M=C+xI,\qquad
K_H=R_H(C[H]+xI)^{-1}R_H^{\mathsf T},\qquad
B=\sum_Hb_HK_H.
\]
Fix a nonempty coordinate set $J$ and $g\ge0$, and set
$K=M^{-1}[J]$ and $F_J=K^{-1/2}B[J]K^{-1/2}$.
For $h>0$, define
$f_{g,h}(\lambda)=\lambda/(\lambda+h)^2+g/(\lambda+h)$.
Writing $C=\sum_\lambda\lambda\Pi_\lambda$ for its spectral
decomposition, put $E_{g,h}=\sum_\lambda e_{g,h}(\lambda)\Pi_\lambda$,
where $e_{g,h}(0)=g$ and, for $\lambda>0$,
\[
u_{\lambda,h}=\min\left\{\frac{h}{\lambda+h},\frac{1+g}{2}\right\},
\qquad
e_{g,h}(\lambda)=(1+g)u_{\lambda,h}-u_{\lambda,h}^2.
\]

\begin{lemma}\label{prop:local-frame-crossing}
Suppose that, throughout an interval of positive $x$,
\begin{equation}\label{eq:local-frame-inequality}
V_{g,J}(x):=
\left[\sum_Hb_H(K_HCK_H+gK_H)
      -x(M^{-1}CM^{-1}+gM^{-1})\right][J]\succeq0.
\end{equation}
Then $\lambda_{\min}(F_J(x))$ is nondecreasing on every subinterval
where it is at most $x$, and is strictly increasing if
$V_{g,J}(x)\succ0$ there.
If $T_{J^c}$ has $|J|$ positive eigenvalues, the condition
$\lambda_{\min}(F_J(x))\le x$ holds beyond its smallest positive
eigenvalue. The choice $J=\{1,\ldots,k\}$ and $g=0$ gives the
criterion for the full sum of projections.

Moreover, if $0\le b_H^{(0)}\le b_H$ and
\begin{equation}\label{eq:frozen-relative-frame}
\left[\sum_Hb_H^{(0)}R_Hf_{g,h}(C[H])R_H^{\mathsf T}
      -E_{g,h}\right][J]\succeq0,
\end{equation}
then \eqref{eq:local-frame-inequality} holds for $0<x\le h$.
\end{lemma}

\begin{proof}
For fixed $\lambda$, put $Q(x,\lambda)=B[J]-\lambda M^{-1}[J]$.
The identities
$K_H'=-K_H^2$, $K_H-xK_H^2=K_HCK_H$, and
$M^{-1}-xM^{-2}=M^{-1}CM^{-1}$ give
\[
\frac{d}{dx}\bigl(x^{1+g}Q(x,\lambda)\bigr)
=x^g\left\{V_{g,J}(x)
 +(x-\lambda)(M^{-1}CM^{-1}+gM^{-1})[J]\right\}.
\]
Take $x_1<x_2$ in the stated subinterval and put
$\lambda=\lambda_{\min}(F_J(x_1))\le x_1$.
Since $F_J\succeq0$, we have $\lambda\ge0$, and
\[
Q(x_1,\lambda)
=K(x_1)^{1/2}(F_J(x_1)-\lambda I)K(x_1)^{1/2}\succeq0.
\]
For $x_1\le x\le x_2$, both terms in the derivative are positive
semidefinite. Integration gives $Q(x_2,\lambda)\succeq0$, hence
$F_J(x_2)\succeq\lambda I$.
If $V_{g,J}\succ0$, the integral is positive definite, and
$F_J(x_2)\succ\lambda I$.
This proves both assertions, including repeated smallest eigenvalues.
If $T_{J^c}$ has $|J|$ positive eigenvalues and $x$ exceeds the
smallest one, the inertia identity gives
$n_+(F_J(x)-xI)<|J|$, so $\lambda_{\min}(F_J(x))\le x$.

For $\lambda\ge0$,
\[
\frac{d}{dx}\left(\frac{\lambda}{(\lambda+x)^2}
                         +\frac{g}{\lambda+x}\right)
=-\frac{2\lambda}{(\lambda+x)^3}
 -\frac{g}{(\lambda+x)^2}\le0.
\]
Thus each local term for $x\le h$ is bounded below by its value
at $h$. For $\lambda>0$, put $u=x/(\lambda+x)$. Then
\[
\frac{\lambda x}{(\lambda+x)^2}+\frac{gx}{\lambda+x}
=(1+g)u-u^2
=\frac{(1+g)^2}{4}-\left(u-\frac{1+g}{2}\right)^2.
\]
Its maximum for $0<x\le h$ is $e_{g,h}(\lambda)$; for
$\lambda=0$, its value is $g$.
The spectral theorem therefore gives
$x(M^{-1}CM^{-1}+gM^{-1})\preceq E_{g,h}$.
Combining this inequality with the local lower bounds,
$b_H\ge b_H^{(0)}$, and \eqref{eq:frozen-relative-frame}
proves \eqref{eq:local-frame-inequality}.
\end{proof}

For a link with at most three coordinates, the local matrix in
\eqref{eq:frozen-relative-frame} has the following explicit form.

\begin{lemma}\label{lem:local-adjugate}
Let $1\le|H|\le3$, let $C[H]\succeq0$ have constant diagonal $a$,
and let $h>0$ and $g\ge0$.
Put $A_H=C[H]-aI$ and $e_H=\sum_{i<j\in H}C_{ij}^2$.
Let $t_H$ be the product of the three off-diagonal entries if
$|H|=3$, and zero otherwise. For $t=a+h$, define
\[
B_H=(t^2-e_H)I-tA_H+A_H^2,\qquad
d_H=t^3-e_Ht+2t_H.
\]
Then $A_H^3=e_HA_H+2t_HI$,
$d_H=t^{3-|H|}\det(C[H]+hI)>0$, and
$(C[H]+hI)^{-1}=B_H/d_H$. Moreover,
\begin{equation}\label{eq:local-adjugate}
f_{g,h}(C[H])
=\frac{(1+g)d_HB_H-hB_H^2}{d_H^2}.
\end{equation}
\end{lemma}

\begin{proof}
Direct multiplication gives
$(A_H^3)_{ij}=e_H(A_H)_{ij}$ for $i\ne j$ and
$(A_H^3)_{ii}=2t_H$. Hence
\[
(tI+A_H)B_H
=(t^3-e_Ht)I+A_H^3-e_HA_H=d_HI.
\]
For $|H|=3$, determinant expansion gives
$d_H=\det(C[H]+hI)$.
For $|H|=1,2$, respectively, $d_H=t^3$ and
$d_H=t(t^2-e_H)$.
Since $C[H]+hI\succ0$, these identities give the asserted
determinant formula and inverse.
Finally,
$f_{g,h}(C[H])=(1+g)(C[H]+hI)^{-1}-h(C[H]+hI)^{-2}$,
which proves \eqref{eq:local-adjugate}.
\end{proof}

\subsection{Stieltjes Matrices and Positive Cliques}
We call $C\succeq0$ a Stieltjes core if coordinate sign changes
make all its off-diagonal entries nonpositive.
Write $\rho(A)$ for the spectral radius of $A$.

\begin{lemma}\label{lem:resolvent-order}
All inequalities in this lemma are entrywise.
If $P\ge0$ and $z>\rho(P)$, then $(zI-P)^{-1}\ge0$.
If $0\le P_2\le P_1$ and $z>\max_i\rho(P_i)$, then
\begin{equation}\label{eq:resolvent-order}
(zI-P_1)^{-1}-(zI-P_2)^{-1}
=(zI-P_1)^{-1}(P_1-P_2)(zI-P_2)^{-1}\ge0.
\end{equation}
For $0\le\mathbf v_2\le\mathbf v_1$, these hypotheses also give
$\mathbf v_2^{\mathsf T}(zI-P_2)^{-1}\mathbf v_2
\le\mathbf v_1^{\mathsf T}(zI-P_1)^{-1}\mathbf v_1$.
If a real square matrix $A$ has nonpositive off-diagonal entries
and positive real eigenvalues, then $A^{-1}\ge0$.
\end{lemma}

\begin{proof}
The series $(zI-P)^{-1}=\sum_{j\ge0}z^{-j-1}P^j$ converges
and has nonnegative terms.
Subtracting the two inverses gives \eqref{eq:resolvent-order}.
The quadratic-form inequality follows by increasing first the
inverse matrix and then the nonnegative vector.
For the last assertion, choose $t$ larger than every diagonal
entry and every eigenvalue of $A$.
Then $P=tI-A\ge0$ and every eigenvalue of $P$ belongs to $(0,t)$.
Thus $\rho(P)<t$, and the first assertion gives
$A^{-1}=(tI-P)^{-1}\ge0$.
\end{proof}

\begin{lemma}\label{prop:singleton-residual}
For singleton links and $C\succeq0$ of constant diagonal $a\ge0$,
we have $\tr(T_+)\le N$.
Moreover, $s_\ell^+(T)=N$ if and only if $C=aI$ on the used
coordinates and $k\le\ell$.
\end{lemma}

\begin{proof}
If there are no links, then $N=k=0$ and the assertions are immediate.
Otherwise, combine equal links and delete unused coordinates.
Then $U=D^{1/2}$, $S=aI$, and $D\succ0$.
Put
\[
R=(aI+D)^{-1/2},\qquad
X=\binom{C^{1/2}D^{1/2}}{D}R,
\qquad
Y=\binom{C^{1/2}D^{1/2}}{-aI}R.
\]
Direct multiplication gives
\begin{equation}\label{eq:singleton-factorization}
T=XX^{\mathsf T}-YY^{\mathsf T},\qquad
\tr(X^{\mathsf T}X)
=\sum_i\frac{ad_i+d_i^2}{a+d_i}=N.
\end{equation}
For every orthogonal projection $\Pi$,
\[
N-\tr(\Pi T)
=\tr\bigl(X^{\mathsf T}(I-\Pi)X\bigr)
 +\tr\bigl(Y^{\mathsf T}\Pi Y\bigr)\ge0.
\]
Equality holds precisely when $(I-\Pi)X=0$ and $\Pi Y=0$.
Since $\rank X=k$, such a projection of rank at most $\ell$
exists precisely when $k\le\ell$ and
$Y^{\mathsf T}X=RD^{1/2}(C-aI)D^{1/2}R=0$.
The latter identity is equivalent to $C=aI$.
The variational formula proves both assertions.
\end{proof}

\begin{lemma}\label{thm:stieltjes-residual}
For arbitrary positive link weights and link sizes, $\tr(T_+)\le N$
if coordinate sign changes transform $C\succeq0$ into a matrix
with nonpositive off-diagonal entries, or into $(a-1)I+J$ with $a\ge1$.
In the Stieltjes case, split each link into the correlation
components of $C[H]$ and delete unused coordinates.
Equality holds if and only if every two resulting links are nested,
or are disjoint and satisfy $C[H_1,H_2]=0$.
In the positive-clique case, equality holds if and only if the
links form a chain under inclusion.
\end{lemma}

\begin{proof}
The empty family has $N=0$ and satisfies the stated conditions.
Assume henceforth that at least one link is present.

\emph{Case 1: the off-diagonal entries can be made nonpositive.}
Apply the sign changes, split each link as stated, combine equal
links, and delete unused coordinates.
These operations preserve the positive eigenvalues and $N$.
For $t>0$, put $C_t=C+tI$, $S_t=S+tI$, and
$B_t(z)=\sum_Hb_HR_H(C_t[H]+zI)^{-1}R_H^{\mathsf T}$.
Let $R=[R_H]_H$ and $E=[b_HE_H]_H$, where
$(E_H)_{ih}=-C_{ih}$ for $i\notin H$, $h\in H$, and is zero
otherwise. Thus $E,R\ge0$, and
\begin{gather*}
(C_t+zI)R_H=R_H(C_t[H]+zI)-E_H,\\
(C_t+zI)B_t(z)=D-E(zI+S_t)^{-1}R^{\mathsf T},\\
\det(zI-T(C_t))
=\det(zI+S_t)
 \det\bigl(zI-D+E(zI+S_t)^{-1}R^{\mathsf T}\bigr)
=\det(zI-L_t),\\
L_t=\begin{pmatrix}D&-E\\R^{\mathsf T}&-S_t\end{pmatrix}.
\end{gather*}
The determinant identities follow from
\eqref{eq:residual-determinant} and the block determinant formula;
after denominators are cleared, they hold at every $z$.
Every coordinate is used, so
\[
\mathbf z^{\mathsf T}B_t(0)\mathbf z
=\sum_Hb_H\|C_t[H]^{-1/2}R_H^{\mathsf T}\mathbf z\|^2>0
\qquad(\mathbf z\ne0).
\]
The matrix $S_t$ is positive definite with nonpositive
off-diagonal entries, so Lemma~\ref{lem:resolvent-order} gives
$S_t^{-1}\ge0$.
Put $A_0=D-ES_t^{-1}R^{\mathsf T}=C_tB_t(0)$.
It has nonpositive off-diagonal entries and
$C_t^{-1/2}A_0C_t^{1/2}=C_t^{1/2}B_t(0)C_t^{1/2}\succ0$.
The same lemma therefore gives $A_0^{-1}\ge0$.
Set $\mathbf v=A_0^{-1}\one$ and
$\mathbf w=S_t^{-1}R^{\mathsf T}\mathbf v$.
Since $A_0^{-1}$ is invertible and nonnegative, every row has
positive sum. Hence $\mathbf v>0$, $\mathbf w\ge0$,
$D\mathbf v-E\mathbf w=\one$, and
$S_t\mathbf w=R^{\mathsf T}\mathbf v$.

For $\mathcal S_t(Y)=S_tY+YD$, define
\[
[\mathcal S_t^{-1}(Q)]_{\cdot j}
=(S_t+d_jI)^{-1}Q_{\cdot j},\qquad
Y_0=0,\qquad
Y_{n+1}=\mathcal S_t^{-1}(R^{\mathsf T}+Y_nEY_n).
\]
The inverse maps nonnegative matrices to nonnegative matrices.
For $0\le Y\le Y'$, the identity
$Y'EY'-YEY=(Y'-Y)EY'+YE(Y'-Y)\ge0$ therefore gives
$0\le Y_n\le Y_{n+1}$ by induction.
If $Y_n\mathbf v\le\mathbf w$, then
\begin{align*}
S_tY_{n+1}\mathbf v
&=R^{\mathsf T}\mathbf v+Y_nEY_n\mathbf v-Y_{n+1}D\mathbf v\\
&\le S_t\mathbf w+Y_nE\mathbf w-Y_nD\mathbf v
 =S_t\mathbf w-Y_n\one,\\
Y_{n+1}\mathbf v
&\le\mathbf w-S_t^{-1}Y_n\one\le\mathbf w.
\end{align*}
Thus $0\le(Y_n)_{ij}\le w_i/v_j$.
The sequence converges entrywise to a matrix $Y$ satisfying
$S_tY+YD=R^{\mathsf T}+YEY$.
For $X=D-EY$, we have $X\mathbf v\ge\one$ and
$X_{ij}\le0$ for $i\ne j$.
If $X\mathbf z=\lambda\mathbf z$, choose $i$ maximizing
$|z_i|/v_i>0$. Then
\[
\Re\lambda
\ge X_{ii}-\sum_{j\ne i}|X_{ij}|\frac{|z_j|}{|z_i|}
\ge X_{ii}-\sum_{j\ne i}|X_{ij}|\frac{v_j}{v_i}
=\frac{(X\mathbf v)_i}{v_i}>0.
\]
Moreover,
\[
\begin{pmatrix}I&0\\-Y&I\end{pmatrix}L_t
\begin{pmatrix}I&0\\Y&I\end{pmatrix}
=\begin{pmatrix}X&-E\\0&YE-S_t\end{pmatrix}.
\]
The characteristic polynomial of $L_t$ equals that of the
symmetric matrix $T(C_t)$, so its eigenvalues are real.
The $k$ eigenvalues of $X$ are therefore positive eigenvalues
of $T(C_t)$, and \eqref{eq:inertia-dimension} shows that there
are no others. Consequently,
\begin{equation}\label{eq:stieltjes-riccati-deficit}
N-\tr((T(C_t))_+)=\tr(EY)\ge\tr(EY_2)\ge0.
\end{equation}
Since $\|T(C_t)-T(C)\|\le\|U\|\sqrt t+t$, letting $t$ tend
to zero and using Lemma~\ref{lem:ky-fan} proves the bound.

For equality, each split link is connected, so
$(C[H]+(t+d_j)I)^{-1}$ is strictly positive for $t\ge0$.
Indeed, its geometric series contains a positive term along a
path between any two coordinates. Thus
$Y_1[(H,h),j]>0$ whenever $h,j\in H$.
Suppose $H_1,H_2$ are incomparable and either intersect or satisfy
$C[H_1,H_2]\ne0$.
There exist $j\in H_1\setminus H_2$, $i\in H_2\setminus H_1$,
$h\in H_1$, and $k\in H_2$ with $C_{ih},C_{jk}<0$.
If the links are disjoint, choose an edge $ji$ between them and
take $h=j$, $k=i$.
If they intersect, connectedness gives an edge from each difference
set to the intersection, providing the required vertices.
It follows that
\[
E_{i,(H_1,h)}>0,\qquad
(Y_1EY_1)_{(H_1,h),i}
\ge Y_1[(H_1,h),j]E_{j,(H_2,k)}Y_1[(H_2,k),i]>0.
\]
Hence $Y_2[(H_1,h),i]>0$ and
$\tr(EY_2)\ge E_{i,(H_1,h)}Y_2[(H_1,h),i]>0$.
Since $d_i,d_j>0$, all these factors have positive limits as
$t$ tends to zero. Equation~\eqref{eq:stieltjes-riccati-deficit}
therefore gives $\tr(T_+)<N$.
Conversely, under the stated pair conditions, the subspaces
$(C+xI)^{1/2}\R^H$ are pairwise nested or orthogonal for every
$x>0$. Their projections commute, and
Lemma~\ref{lem:commuting-flow-spectrum} gives equality.

\emph{Case 2: the matrix can be changed to $(a-1)I_k+J_k$, $a\ge1$.}
Apply the sign changes and delete unused coordinates.
For $1\le r\le k$, let
$\mathcal L_r=\one_k\wedge\iota_{\one_k}$ on
$\bigwedge^r\R^k$.
In the exterior expansion of $cI+J_k$, every term containing
$J_k$ twice vanishes because $J_k$ has rank one.
The terms containing it once sum to $c^{r-1}\mathcal L_r$.
Thus, for $c>0$,
\[
\bigwedge^r(cI+J_k)=c^{r-1}(cI+\mathcal L_r),\qquad
\langle e_I,\mathcal L_re_I\rangle
=\|\iota_{\one_k}e_I\|^2=r.
\]
Every diagonal entry is $c^{r-1}(c+r)$, so the matrix obtained
by dividing each entry by the square root of the corresponding
two diagonal entries is
\[
\widehat G_r(c)=\frac{cI+\mathcal L_r}{c+r},\qquad
\widehat G_r'(c)=\frac{I-\widehat G_r(c)}{c+r}.
\]
Set $c=a-1+x>0$.
Lemma~\ref{prop:gram-comparison} shows that every leading
eigenvalue sum is nonincreasing, and
Lemma~\ref{prop:crossing-telescope} gives $\tr(T_+)\le N$.

Suppose equality holds. For a coordinate set $I$ of size $h$,
the Schur complement of $M=cI+J_k$ is
$M/I=cI+cJ/(c+h)$.
All its off-diagonal entries are positive.
The comparisons just proved and
Lemma~\ref{prop:trace-equality-commutation} therefore show that
the link projections commute.
If $H_1,H_2$ are not nested, put $b=|H_1\cap H_2|$.
After orthogonal projection away from their common span, the
inner product between the remaining vectors indexed by any
$i\in H_1\setminus H_2$ and $j\in H_2\setminus H_1$ is
$c/(c+b)>0$. Their projections therefore cannot commute.
Thus the links form a chain.
Conversely, a chain gives nested subspaces
$(C+xI)^{1/2}\R^H$, so their projections commute and
Lemma~\ref{lem:commuting-flow-spectrum} gives equality.
\end{proof}

\subsection{Matrices on Four Coordinates}
It remains to treat the positive paw and the mixed diamond.
A link $H$ is called \emph{short} if $1\le|H|\le3$.

\begin{lemma}\label{prop:four-coordinate}
Let $T$ arise from a complete simplex, with short links, positive
integer link weights, and at most four used coordinates.
Then $\tr(T_+)\le N$.
For a positive paw or a mixed diamond, we also have
$\tr((T-\mathbf z_i\mathbf z_i^{\mathsf T})_+)\le N-d_i$
for every coordinate $i$.
For either core and $r\in\{2,3,4\}$, adjoining a singleton of
weight $b>0$ preserves the bound $s_r^+(T)\le N$, with $N$ replaced
by $N+b$. It also preserves strict inequality.
\end{lemma}
\begin{proof}
Split each link into its correlation components. The preceding
classification and the Stieltjes and positive-clique estimates leave
only the positive paw and the mixed diamond, with integer $a\ge2$.

Index the coordinates by $0,1,2,3$. In this proof, write $013$ for
$\{0,1,3\}$, and use the same convention for other coordinate sets.
For submatrices such as $[0132,0132]$, the indicated order of the
coordinates is retained.
After a diagonal sign change, the off-diagonal matrices are
\[
 A_{\mathrm p}=\begin{pmatrix}
 0&1&1&1\\1&0&1&0\\1&1&0&0\\1&0&0&0
 \end{pmatrix},\qquad
 A_{\mathrm d}=\begin{pmatrix}
 0&1&1&1\\1&0&-1&0\\1&-1&0&1\\1&0&1&0
 \end{pmatrix}.
\]
The paw is invariant under exchanging $1,2$. The diamond is invariant
under exchanging $0,2$ and changing the sign at $1$; we call this
operation signed reflection.
For either choice of $A$, put
\[
 M(t)=tI_4+A,\qquad
 F(t)=\sum_Hb_H\Pi _{M(t)^{1/2}\R^H}\qquad(t>2).
\]
Thus $F(a+x)=F_C(x)$ for $C=aI_4+A$ and $x>0$.
Since each link projection has rank $|H|$, we have $\tr F(t)=N$.

For $C_*=2I+A_{\mathrm d}$, the following eigenvectors give its
spectral projections:
\begin{equation}\label{eq:diamond-spectral-projectors}
 \begin{gathered}
 V=(\mathbf v_0\ \mathbf v_1\ \mathbf v_3\ \mathbf v_4)
 =\begin{pmatrix}1&1&1&1\\-1&0&2&0\\-1&1&-1&1\\0&-2&0&1\end{pmatrix},\\
 V^{\mathsf T}V=\diag(3,6,6,3),\qquad
 C_*V=V\diag(0,1,3,4),\qquad
 \Pi _j=\frac{\mathbf v_j\mathbf v_j^{\mathsf T}}{\|\mathbf v_j\|^2},\quad \sum_j\Pi _j=I.
 \end{gathered}
\end{equation}

\medskip\noindent\emph{Comparisons for Schur complements.}\
\phantomsection\label{prop:conditioned-first-flow}
Put $M=tI+A$ and $K_r=M/\{r\}$, with $t>2$ and
$A\in\{A_{\mathrm p},A_{\mathrm d}\}$.
At $t=2$, the paw matrix has leading principal minors $2,3,4,5$,
and the diamond matrix has eigenvalues $0,1,3,4$ by
\eqref{eq:diamond-spectral-projectors}. Hence $M\succ0$ for $t>2$.
The comparisons below apply to arbitrary nonnegative weights on
coordinate subspaces.

\phantomsection\label{par:genuine-conditioning}
Normalize $K_r$ by setting
$(G_r)_{ij}=(K_r)_{ij}/\sqrt{(K_r)_{ii}(K_r)_{jj}}$, and put
$p=t^2-1$, ${b_0}=t/p$, and $a_\pm=(t\pm1)^{-1}$.
The table gives the off-diagonal entries and the rates
$-\partial_t\log|(G_r)_{ij}|$ at nonzero entries.
At an identically zero entry, the listed rate is chosen to complete
the Euclidean realization verified below.
\begin{equation}\label{eq:quotient-rates}
\begin{array}{c|c|c|c}
A&r&\text{entries, in increasing pair order}&\text{rates}\\\hline
A_{\mathrm p}&0&((t+1)^{-1},-p^{-1},-p^{-1})
 & (a_+,2{b_0},2{b_0})\\
&1,2&((t+1)^{-1},p^{-1/2},0)
 & (a_+,{b_0},a_++{b_0})\\
&3&(p^{-1/2},p^{-1/2},t^{-1})
 & ({b_0},{b_0},t^{-1})\\\hline
A_{\mathrm d}&0,2&(-(t-1)^{-1},-p^{-1},(t+1)^{-1})
 & (a_-,2{b_0},a_+)\\
&1&((t-1)^{-1},p^{-1/2},p^{-1/2})
 & (a_-,{b_0},{b_0})\\
&3&(p^{-1/2},(t+1)^{-1},-p^{-1/2})
 & ({b_0},a_+,{b_0})
\end{array}
\end{equation}
Rows paired by reflection use the corresponding permutation and
diagonal sign change. Each rate triple is realized by three affinely
independent points. Indeed,
$a_-+a_+=2{b_0}$ gives a right triangle in the fourth
row, and the second row is also a right triangle. A triple
$({a_0},{c_0},{c_0})$ is realized by
\[
 (0,0),\quad(\sqrt{a_0},0),\quad
 (\sqrt{a_0}/2,\sqrt{c_0-a_0/4}).
\]
The remaining rows satisfy $4c_0>a_0>0$, since
\[
 4{b_0}-a_\pm=\frac{3t\pm1}{p}>0,
 \qquad 4{b_0}-t^{-1}=\frac{3t^2+1}{tp}>0.
\]
Lemma~\ref{prop:gram-comparison} therefore makes the largest
eigenvalue of each corresponding projection sum nonincreasing.
For the other partial sums, set
$G_{ij}=(M^{-1})_{ij}/\sqrt{(M^{-1})_{ii}(M^{-1})_{jj}}$.
The cofactor formula gives
\[
\begin{array}{c|rrrrrr}
 &01&02&03&12&13&23\\\hline
 A_{\mathrm p}&-r_{\mathrm p}&-r_{\mathrm p}&-s_{\mathrm p}&-v_{\mathrm p}&r_{\mathrm p}s_{\mathrm p}&r_{\mathrm p}s_{\mathrm p}\\
 A_{\mathrm d}&-r_{\mathrm d}&-v_{\mathrm d}&-s_{\mathrm d}&r_{\mathrm d}&0&-s_{\mathrm d}
\end{array}
\]
where all square roots are positive and
\begin{align*}
 r_{\mathrm p}^2&=\frac{t-1}{(t+1)(t^2-2)},&
 s_{\mathrm p}^2&=\frac{t+1}{t(t-1)(t+2)},&
 v_{\mathrm p}&=\frac{t^2-t-1}{t(t^2-2)},\\
 r_{\mathrm d}^2&=\frac{t+2}{t(t^2-2)},&
 s_{\mathrm d}^2&=\frac{t-2}{t(t^2-2)},&v_{\mathrm d}&=\frac{t}{t^2-2}.
\end{align*}
For the paw, the rates of $r_{\mathrm p},s_{\mathrm p},v_{\mathrm p}$ are respectively
\begin{align*}
 b&=\frac{t}{t^2-2}-\frac1{t^2-1},\\
 {c_1}&=\frac12\left(\frac1t+\frac1{t-1}+\frac1{t+2}-\frac1{t+1}\right),\\
 d&=\frac1t+\frac{2t}{t^2-2}-\frac{2t-1}{t^2-t-1}
   =\frac{(t-1)p_{\mathrm p}(t)}{t(t^2-2)(t^2-t-1)},\qquad
 p_{\mathrm p}(t)=t^3-t^2-2t-2.
\end{align*}
Here $b,{c_1}>0$ for $t>2$. Since
$p_{\mathrm p}(7/3)=16/27$ and $p_{\mathrm p}'(t)=3t^2-2t-2>0$ for $t\ge2$,
we have $d>0$ for $t\ge7/3$. Also,
$\frac{2t}{t^2-2}-\frac1t>\frac1t$ and $\frac{2t-1}{t^2-t-1}>\frac2t$,
so
\[
 4b-d>\frac3t-\frac4{t^2-1}
 =\frac{3(t-2)^2+8(t-2)+1}{t(t^2-1)}>0.
\]
Thus $0<d<4b$ for $t\ge7/3$.

The four points
\[
 \begin{aligned}
 \mathbf u_1&=(-\sqrt d/2,0,0),&\mathbf u_2&=(\sqrt d/2,0,0),\\
 \mathbf u_0&=(0,\sqrt{b-d/4},0),&\mathbf u_3&=(0,\sqrt{b-d/4},\sqrt {c_1})
 \end{aligned}
\]
have squared distances $b,b,{c_1},d,b+{c_1},b+{c_1}$ in the order
used in the table. The inverse-metric criterion now implies that,
for the paw, the sum of the three largest eigenvalues of $F(t)$ is
nonincreasing on $t\ge7/3$.

For the diamond the rates of $r_{\mathrm d},s_{\mathrm d},v_{\mathrm d}$ are
\[
 {c_2}=\frac{t}{t^2-2}+\frac1{t(t+2)},\qquad
 {c_3}=\frac{t}{t^2-2}-\frac1{t(t-2)},\qquad
 {c_4}=\frac{2t}{t^2-2}-\frac1t.
\]
Here ${c_2},{c_4}>0$ for $t>2$, and
${c_3}=(t-1)(t^2-2t-2)/(t(t-2)(t^2-2))\ge0$
for $t\ge1+\sqrt3$. Furthermore,
\begin{align*}
 4{c_2}-{c_4}
 &=\frac{2t}{t^2-2}+\frac1t+\frac4{t(t+2)}>0,\\
 4{c_3}-{c_4}
 &=\frac{p_{\mathrm d}(t)}{t(t-2)(t^2-2)},\qquad
 p_{\mathrm d}(t)=3t^3-10t^2-2t+12.
\end{align*}
Since $p_{\mathrm d}(22/7)=24/343$ and
$p_{\mathrm d}'(t)=t(9t-20)-2\ge19$ for $t\ge3$,
the second expression is positive for $t\ge22/7$.

The points
\[
 \begin{aligned}
 \mathbf u_0&=(-\sqrt{c_4}/2,0,0),&\mathbf u_2&=(\sqrt{c_4}/2,0,0),\\
 \mathbf u_1&=(0,\sqrt{{c_2}-{c_4}/4},0),&\mathbf u_3&=(0,0,\sqrt{{c_3}-{c_4}/4})
 \end{aligned}
\]
have the required squared distances; for the zero entry take
$d_{13}={c_2}+{c_3}-{c_4}/2$.
Hence the sum of the three largest eigenvalues of $F(t)$ is
nonincreasing for the diamond on $t\ge22/7$.

Let $J=\{0,1,2,3\}\setminus\{r\}$. Since
$K_r^{-1}=M^{-1}[J]$, we may apply the inverse-metric distance
criterion to the corresponding principal triple of $G$. This gives
monotonicity of the sum of the two largest eigenvalues for the metric
$K_r$ on the following intervals:
\[
\begin{array}{c|c|c|c}
 A&r&\text{squared-distance realization}&\text{interval}\\\hline
 A_{\mathrm p}&1,2&b,{c_1},b+{c_1}&t>2\\
 A_{\mathrm p}&0,3&\text{the preceding four-point realization}&t\ge7/3\\
 A_{\mathrm d}&3&{c_2},{c_2},{c_4}&t>2\\
 A_{\mathrm d}&0,2&{c_2},{c_3},{c_2}+{c_3}&t\ge1+\sqrt3\\
 A_{\mathrm d}&1&{c_3},{c_3},{c_4}&t\ge22/7.
\end{array}
\]
The first and fourth rows are right triangles. In the fourth row,
the rate at $G_{13}=0$ is chosen as ${c_2}+{c_3}$ to obtain this
realization. The other rows follow from the preceding inequalities.

We next locate the positive eigenvalues of $T$.
If every present weight is at least one and every coordinate is used,
then
$D=\sum_Hb_HR_HR_H^{\mathsf T}\succeq I$,
$M\succeq(t-2)I$ and $M[H]\preceq(t+2)I$ for $|H|\le3$, so
\begin{equation}\label{eq:cover-floor}
 F(t)\succeq\frac{M^{1/2}DM^{1/2}}{t+2}
 \succeq\frac{t-2}{t+2}I,
 \qquad F(22/7)\succeq\frac29I.
\end{equation}
For the paw, put $M=(7/3)I+A_{\mathrm p}$ and assign each coordinate
to one link containing it. The nonempty sets $J_H\subseteq H$
so obtained partition the four coordinates. Inclusion of the
corresponding subspaces gives
 \(F_{2I+A_{\mathrm p}}(1/3)\succeq M^{1/2}B^{-1}M^{1/2}, \qquad B=\bigoplus_{J_H\ne\varnothing}M[J_H].\)
The matrix $M-B/3$ has diagonal $14/9$ and off-diagonal entries
in $[2/3,1]$ on the paw edges. For the absolute coordinate values
$x,y,z,w$ of a nonzero vector, put
\begin{align*}
 f(x,y,z,w)&=\tfrac{14}{9}(x^2+y^2+z^2+w^2)-2x(y+z+w)+\tfrac43yz\\
 &=\frac{14}{9}\left(w-\frac9{14}x\right)^2
 +\frac{10}{9}\left(y+z-\frac9{10}x\right)^2
 +\frac49(y-z)^2+\frac4{315}x^2.
\end{align*}
If coordinates $1,2$ have the same sign, the quadratic form is
at least $f(x,y,z,w)>0$. Otherwise interchange coordinates $1,2$ so that
coordinate $1$ has sign opposite to coordinate $0$. The quadratic
form is then at least
$f_*(x,y,z,w)=\tfrac{14}{9}(x^2+y^2+z^2+w^2)-2xy-2yz-2xw+\tfrac43xz$, and
 \(f_*(x,y,z,w)-f(x,y,z,w)=\tfrac{10}{3}z(x-y)\ge0\quad(x\ge y),\qquad f_*(x,y,z,w)-f(y,x,z,w)=2w(y-x)\ge0\quad(y\ge x).\)
Thus $M\succ B/3$, so $B^{-1}\succ M^{-1}/3$ and
$F_{2I+A_{\mathrm p}}(1/3)\succ\tfrac13I_4$.

\medskip\noindent\emph{A comparison for hyperplane projections.}\
\phantomsection\label{prop:positive-hyperplane-resolvent}
Let $\scrU(t)$ be a hyperplane with unit normal $\mathbf n(t)$, and let the
unit vectors $\mathbf n_*,\mathbf n_1,\ldots,\mathbf n_s\in \scrU(t)$ have nonnegative Gram
entries nonincreasing in $t$. Suppose
$\widehat{\mathbf n}=e \mathbf n+\sqrt{1-e^2}\,\mathbf n_*$ is a unit vector
and $e^2$ is nonincreasing. For fixed $c,p,u_i\ge0$,
the smallest eigenvalue of
\[
 F=c\Pi _\scrU+\sum_i u_i\Pi _{\scrU\cap \mathbf n_i^\perp}+p(I-\widehat{\mathbf n}\widehat{\mathbf n}^{\mathsf T})
\]
is nondecreasing. If $p=0$, then $\mathbf n$ lies in its kernel and
the assertion follows. Suppose $p>0$, and put
$u=\sum_i u_i$, $E=\sum_i u_i\mathbf n_i\mathbf n_i^{\mathsf T}$ on $\scrU$ and
$F_\scrU=(c+u)I_\scrU-E\succeq0$, extended by zero on $\scrU^\perp$.
The Rayleigh quotient at $\mathbf n$ gives
$0\le\lambda_{\min}(F)\le p(1-e^2)$.
For $\lambda<p$, the matrix $A=F_\scrU+(p-\lambda)I$ is positive
definite, so $F-\lambda I=A-p\widehat{\mathbf n}\widehat{\mathbf n}^{\mathsf T}\succeq0$ exactly when
$p\widehat{\mathbf n}^{\mathsf T}A^{-1}\widehat{\mathbf n}\le1$. Writing $z=c+u+p-\lambda>u$,
this condition becomes
\begin{equation}\label{eq:positive-hyperplane-resolvent}
 p\left(\frac{e^2}{p-\lambda}
       +(1-e^2)f_0(t,\lambda)\right)\le1,
 \qquad f_0=\mathbf n_*^{\mathsf T}(zI_\scrU-E)^{-1}\mathbf n_*.
\end{equation}
Let $V=[\sqrt{u_i}\mathbf n_i]_i$, $G=V^{\mathsf T}V$ and
$\mathbf w=V^{\mathsf T}\mathbf n_*$. Since $\|E\|=\|G\|\le u<z$, expansion
of the inverse gives
\[
 f_0=\frac1z+\sum_{j\ge0}\frac{\mathbf w^{\mathsf T}G^j\mathbf w}{z^{j+2}}.
\]
The entries of $G$ and $\mathbf w$ are nonnegative and
nonincreasing. Each numerator is a sum of their products, so
$f_0$ is nonincreasing. Also,
$f_0\le(c+p-\lambda)^{-1}\le(p-\lambda)^{-1}$.
The left side of~\eqref{eq:positive-hyperplane-resolvent}
is nondecreasing in $f_0$ and $e^2$, and therefore nonincreasing
in $t$. Hence, if
$\lambda<\lambda_{\min}(F(t_1))\le p$ and $t_2\ge t_1$, then
$F(t_2)\succeq\lambda I$. Letting $\lambda$ increase to
$\lambda_{\min}(F(t_1))$ proves the assertion.
Assume also that $e^2>0$, $(e^2)'<0$, and
$F_\scrU\mathbf n_*\ne0$ throughout an interval. At the smallest
eigenvalue $\lambda<p$, the spectral theorem gives
$f_0<(p-\lambda)^{-1}$. Thus the left side of
\eqref{eq:positive-hyperplane-resolvent} strictly decreases from
one at fixed $\lambda$. Hence $F(t)-\lambda I$ is positive definite
at every later point of the interval, and the smallest eigenvalue
is strictly increasing.

\medskip\noindent\emph{Normal vectors for the mixed diamond.}\
\phantomsection\label{prop:diamond-normal-data}
Put $M=tI+A_{\mathrm d}$, $t>2$, $\mathbf v_i=M^{1/2}e_i$ and
$\scrU=\Span(\mathbf v_0,\mathbf v_1,\mathbf v_2)$. Switching $\mathbf v_0$ gives
the Gram matrix $(t+1)I_3-J_3$, whose inverse is
\[
 ((t+1)I_3-J_3)^{-1}
 =\frac1{t+1}I_3+\frac1{(t+1)(t-2)}J_3.
\]
Thus the unit normals $\mathbf n_i\in \scrU$ to
$\Span\{\mathbf v_j:j\in012\setminus\{i\}\}$ may be signed so that
$\mathbf n_i^{\mathsf T}\mathbf n_j=r:=1/(t-1)$ for $i\ne j$.
Let $\mathbf n$ be a unit normal to $\scrU$ and $\widehat{\mathbf n}$ a unit normal to
$\Span(\mathbf v_0,\mathbf v_1,\mathbf v_3)$. The projection of $\widehat{\mathbf n}$ onto $\scrU$ is parallel
to $\mathbf n_2$. By the entries of $M^{-1}$ computed above,
\[
 \widehat{\mathbf n}=e \mathbf n+\sqrt{1-e^2}\,\mathbf n_2,
 \qquad e^2
 =\frac{(M^{-1})_{23}^2}{(M^{-1})_{22}(M^{-1})_{33}}
 =\frac{t-2}{t(t^2-2)},
\]
after choosing the two normal signs. Consequently
\[
 r'=-\frac1{(t-1)^2}<0,\qquad
 (e^2)'=-\frac{2(t-1)(t^2-2t-2)}{t^2(t^2-2)^2}<0
 \quad(t>1+\sqrt3).
\]
The same data apply after quotienting by $\R\mathbf v_1$:
$\mathbf n,\mathbf n_0,\mathbf n_2,\widehat{\mathbf n}$ lie in
$\mathbf v_1^\perp$, and $\mathbf n_0,\mathbf n_2$ are the normals
in $\scrU\cap\mathbf v_1^\perp$ to the images of $12,01$.

\medskip\noindent\emph{Links with a common coordinate and the addition of singleton links.}\
\phantomsection\label{prop:common-head-paw-diamond-damping}
Fix $A\in\{A_{\mathrm p},A_{\mathrm d}\}$. For each coordinate $r$,
let $(\mathsf B_r)$ denote
\[
 \tr(T_+)\le N\qquad
 \left\{\begin{array}{l}
 C=aI_4+A,\quad a\in\mathbb Z_{\ge2},\quad b_H\in\mathbb Z_{>0},\\
 r\in H\quad\text{for every present link }H.
 \end{array}\right.
\]
We first deduce the relative and singleton estimates from
$(\mathsf B_r)$, and then prove $(\mathsf B_r)$.
Assume $(\mathsf B_r)$, and let $T$ have core $C=aI_4+A$.
Put
$J=\{0,1,2,3\}\setminus\{r\}$ and $T_r=T-\mathbf z_r\mathbf z_r^{\mathsf T}$.
The relative projection family at $x>0$ has metric
$G=(C+xI)/\{r\}$. Each unweighted summand is a positive
contraction with range $G^{1/2}\R^{H\cap J}$, so it is bounded
above by the projection onto that space. Let $h$ be the weight of
$J$, with $h=0$ if it is absent. Adjoin $r$ to every other nonempty
intersection $H\cap J$ and merge repetitions. These intersections
have size at most two, so the new links are short and all contain $r$.
After removing the common coordinate, their projection family satisfies
 \(F_r(x)\preceq hI+B_C(x),\qquad B_C(h+y)\sim B_{C+hI}(y)\quad(y>0).\)
Let $\mu_j$ and $\nu_j$ be the positive eigenvalues of $T_r$ and
the shifted quotient, respectively, with both lists padded by zeros
to length three.
The inertia identities give
$\#\{\mu_j>x\}\le\#\{h+\nu_j>x\}$ for $x>h$, hence
$\mu_j\le h+\nu_j$. Applying $(\mathsf B_r)$ to the shifted family and removing its common
coordinate gives
$\tr((T_r)_+)\le3h+\sum_j\nu_j \le\sum_Hb_H|H\cap J|=N-d_r$.
The shifted diagonal $a+h$ remains an integer at least two.

Let $T_0$ have degrees $d_i^{(0)}$ and $N_0=\sum_i d_i^{(0)}$.
Adding a singleton of weight $b>0$ at $r$ produces
\[
 T(b)=\begin{pmatrix}T_0&\sqrt b\,\mathbf z_r\\
                    \sqrt b\,\mathbf z_r^{\mathsf T}&b-a\end{pmatrix},
 \qquad \|\mathbf z_r\|^2=a+d_r^{(0)}.
\]
For $\ell\in\{2,3,4\}$, a rank-$j\le\ell$ projection with last
diagonal entry $0<t<1$ has the form
\[
 \Pi =\begin{pmatrix}\Pi _0+(1-t)\mathbf v\mathbf v^{\mathsf T}&\sqrt{t(1-t)}\mathbf v\\
 \sqrt{t(1-t)}\mathbf v^{\mathsf T}&t\end{pmatrix},\quad
 \|\mathbf v\|=1,\quad \Pi _0\mathbf v=0,\quad\rank \Pi _0=j-1;
\]
as follows from the block identities in $\Pi ^2=\Pi $.
The inequality
$\bigl(\sqrt t\,\mathbf z_r^{\mathsf T}\mathbf v-\sqrt{b(1-t)}\bigr)^2\ge0$
is equivalent to
$2\sqrt{bt(1-t)}\mathbf z_r^{\mathsf T}\mathbf v-b(1-t)
 \le t(\mathbf z_r^{\mathsf T}\mathbf v)^2$. Therefore
\begin{align*}
 \tr(\Pi T(b))-b
 &\le(1-t)\{\tr(\Pi _0T_0)+\mathbf v^{\mathsf T}T_0\mathbf v\}
       +t\{\tr(\Pi _0T_0)+(\mathbf z_r^{\mathsf T}\mathbf v)^2-a\}\\
 &\le(1-t) s_\ell^+(T_0)
       +t\{d_r^{(0)}+ s_{\ell-1}^+(T_0-\mathbf z_r\mathbf z_r^{\mathsf T})\}.
\end{align*}
Here $\Pi _0+\mathbf v\mathbf v^{\mathsf T}$ has rank $j\le\ell$, and
$(\mathbf z_r^{\mathsf T}\mathbf v)^2+\|\Pi _0\mathbf z_r\|^2\le\|\mathbf z_r\|^2$.
At $t=0$, $\tr(\Pi T(b))-b\le s_\ell^+(T_0)-b$;
at $t=1$, it is at most $s_{\ell-1}^+(T_0)-a$.
The latter also satisfies the second bound above, since
$\|\Pi _0\mathbf z_r\|^2\le\|\mathbf z_r\|^2$.
Maximizing over projections, including the zero projection, proves
\begin{equation}\label{eq:rank-singleton}
  s_\ell^+(T(b))-b\le\max\{ s_\ell^+(T_0),\,
 d_r^{(0)}+ s_{\ell-1}^+(T_0-\mathbf z_r\mathbf z_r^{\mathsf T})\}.
\end{equation}
The relative estimate bounds the second term by $N_0$, proving
the singleton bound. If $s_\ell^+(T_0)<N_0$, the preceding convex
combination is strictly less than $N_0$ for $0<t<1$, as are the
endpoint bounds. Compactness of the set of projections of rank at
most $\ell$ then gives $s_\ell^+(T(b))<N_0+b$.
For $\ell=4$, this also bounds the positive trace, since
$n_+(T(b))\le4$.

We now prove $(\mathsf B_r)$ for a nonempty family; the empty
family is immediate. Put $p=\sum_Hb_H$.
The common-coordinate reduction separates the eigenvalue $p$
and leaves a projection sum of trace $N-p$ with metric
$(tI+A)/\{r\}$, $t=a+x$. The block inverse identity
 \((M[H]^{-1})[H\setminus\{r\}] =((M/\{r\})[H\setminus\{r\}])^{-1}\)
identifies this with the restricted local projection family $F_{J_r}$,
$J_r=\{0,1,2,3\}\setminus\{r\}$.
If at most three original coordinates are used, the Stieltjes and
positive-clique estimates already give the result. With four coordinates,
the largest eigenvalue of the quotient matrix is nonincreasing for $t>2$.
If the quotient block has three positive eigenvalues, its second
partial sum must also be nonincreasing between the third and
second eigenvalues. Its trace is constant, so this is equivalent
to its smallest eigenvalue being nondecreasing there.
The Schur-complement table gives the required comparison except for
the paw with $a=2$, the diamond with $a=2$, and the diamond with
$(a,r)=(3,1)$.
For the paw use the endpoint $F_C(1/3)\succ I/3$; for
$(a,r)=(3,1)$ use $F_C(1/7)\succeq2I/9\succ I/7$.
The common-line decomposition gives the same lower bounds for
the quotient. Its three positive eigenvalues therefore exceed the
relevant endpoint, by the inertia identity. This proves
$(\mathsf B_r)$ at every coordinate of the paw and $(\mathsf B_3)$
for the diamond. Hence the relative estimate and the singleton bound
already hold at coordinate $3$ of the diamond. It remains to consider
$A=A_{\mathrm d}$, $a=2$, and $r\in\{0,1,2\}$.

For these remaining cases, apply the local resolvent comparison with
$g=3/10$.
For $h\in\{3/4,8/7\}$ and $\lambda\in\{0,1,3,4\}$,
\[
 e_{g,h}(\lambda)=hf_{g,h}(\lambda),\qquad
 E_{g,h}=h\{(1+g)(C_*+hI)^{-1}
                         -h(C_*+hI)^{-2}\}.
\]
For $\lambda>0$, this follows from $h<13\lambda/7$;
at $\lambda=0$, both sides equal $g$.
Set $J_r=\{0,1,2,3\}\setminus\{r\}$. For lower weights
$0\le b_H^{(0)}\le b_H$, write $\mathbf b^{(0)}=(b_H^{(0)})_H$ and define
\[
 M_{r,\mathbf b^{(0)}}(h)=2000\left[
  \sum_Hb_H^{(0)}R_Hf_{g,h}(C_*[H])R_H^{\mathsf T}
                         -E_{g,h}\right][J_r].
\]
By Lemma~\ref{prop:local-frame-crossing}, it remains to prove
$M_{r,\mathbf b^{(0)}}(h)\succeq0$ for the stated endpoints.

For the common coordinate $0$, take $h=3/4$ and
$b_H^{(0)}=b_H$. Write the weights on
$01,02,03,012,013,023$ as $p_1,p_2,p_3,d,e,f$, and abbreviate
$M=M_{0,\mathbf b^{(0)}}(3/4)$ in coordinate order $(1,2,3)$.
The local formula~\eqref{eq:local-adjugate} gives
\begin{gather*}
 (M_{11},M_{22},M_{33})=
 (a_2p_1+a_-d+a_{\rm p}e-g_1,
  a_2p_2+a_-d+a_+f-g_2,
  a_2p_3+a_{\rm p}e+a_+f-g_3),\\
 (M_{12},M_{13},M_{23})=
 \bigl(b_-(d-3/4),\ b_{\rm p}e,\ b_+(3/4-f)\bigr),
\end{gather*}
where the positive coefficients satisfy the following rational bounds:
\[
\begin{array}{c|rrrr|rrr|rrr}
 &a_2&a_-&a_{\rm p}&a_+&b_-&b_{\rm p}&b_+&g_1&g_2&g_3\\\hline
 \text{lower bound}&791&657&784&824&0&0&0&0&0&0\\
 \text{upper bound}& & & & &72&38&172&494&519&619
\end{array}
\]
For example, $a_-=5920/9$, $b_-=640/9$, and
$(g_1,g_2,g_3)=\tfrac34(a_-,a_-+a_+-a_2,a_+)$.
Integrality gives $|d-3/4|=d+3/4-3\one_{\{d>0\}}/2$ and the same
formula for $f$. With $\mathbf s=(5,11,10)^{\mathsf T}$ and
$r_i=M_{ii}s_i-\sum_{j\ne i}|M_{ij}|s_j$, we obtain
\begin{align*}
 r_1&\ge3955p_1+2493d+3540e+1188\one_{\{d>0\}}-3064\ge476,\\
 r_2&\ge8701p_2+6867d+7344f+540\one_{\{d>0\}}
                                      +2580\one_{\{f>0\}}-7269\ge138,\\
 r_3&\ge7910p_3+7650e+6348f+2838\one_{\{f>0\}}-7609\ge41.
\end{align*}
Each of the integers $p_1+d+e$, $p_2+d+f$, and $p_3+e+f$
is positive, since all three coordinates are used. In each row,
minimizing over its positive incident weights gives the last
inequality. Hence $M\succ0$ by \eqref{eq:diagonal-squares}.

This proves the local comparison for $0<x\le3/4$. For $x\ge3/4$
the inverse-metric comparison applies, because
$2+3/4>1+\sqrt3$. The signed permutation exchanging $0,2$ and
negating coordinate $1$ gives the same conclusion when the common
coordinate is $2$.

For the common coordinate $1$, split $13$ into the orthogonal
singletons $1,3$. The singleton bound already proved at coordinate
$3$ permits deletion of $3$. Every remaining link contains $1$.
Fewer than four used coordinates are covered by the previous bounds.
Otherwise the nonsingleton types are $01,12,012,013,123$, and at
least one path $013,123$ occurs. If both occur, use the first
collection below. If exactly one occurs, reflection permits $123$;
then use of coordinate $0$ requires $01$ or $012$.
For each of the three minimal collections of link types, give each
listed type weight one and all other types weight zero. Define
$a_i^{(0)}$ and $c_{ij}^{(0)}$ by the table below.
Substitution in~\eqref{eq:local-adjugate} gives
$M_{ii}\ge a_i^{(0)}$ and $|M_{ij}|\le c_{ij}^{(0)}$ for
$M=M_{1,\mathbf b^{(0)}}(h)$, in coordinate order $(0,2,3)$.
\[
\begin{array}{c|c|c|c|c|c}
 \{H:b_H^{(0)}>0\}&h&\mathbf s&(a_1^{(0)},a_2^{(0)},a_3^{(0)})&(c_{12}^{(0)},c_{13}^{(0)},c_{23}^{(0)})&
 (r_1,r_2,r_3)\text{ lower bound}\\\hline
 \{013,123\}&8/7&(5,5,1)&(9,9,490)&(2,15,15)&(20,20,340)\\
 \{01,123\}&3/4&(2,1,3)&(273,303,166)&(29,129,73)&(130,26,167)\\
 \{012,123\}&3/4&(7,1,7)&(139,961,166)&(43,129,73)&(27,149,186)
\end{array}
\]
Here $r_i\ge a_i^{(0)}s_i-\sum_{j\ne i}c_{ij}^{(0)}s_j$.
The positive row differences prove the local comparisons up to
the stated endpoints. For the first collection, the inverse-metric
comparison continues from $t=22/7$. With exactly one path, the
following hyperplane comparison covers the remaining interval.

After signed reflection, we may assume that the only path is $013$. Put
$\scrV=\mathbf v_1^\perp$, $\scrU_1=\scrU\cap \scrV$. In the common-coordinate quotient,
the preceding normal vectors give
\[
\begin{aligned}
 F&=F_\scrU+b_{013}(I_\scrV-\widehat{\mathbf n}\widehat{\mathbf n}^{\mathsf T}),\\
 F_\scrU&=(b_{012}+b_{01}+b_{12})I_{\scrU_1}
           -b_{01} \mathbf n_2\mathbf n_2^{\mathsf T}-b_{12} \mathbf n_0\mathbf n_0^{\mathsf T},
\end{aligned}
\]
with $F_\scrU$ extended by zero on $\scrV\ominus \scrU_1$.
Since coordinate $2$ is used, we have $b_{12}+b_{012}>0$, whence
$\mathbf n_2^{\mathsf T}F_\scrU\mathbf n_2=b_{012}+b_{12}(1-r^2)>0$.
The hyperplane-resolvent comparison, together with the identities
for the normals, implies strict increase of the smallest eigenvalue
for $t>1+\sqrt3$. In particular, it applies for $x\ge3/4$, since
$2+3/4>1+\sqrt3$.

The comparisons now cover every required interval. The summation
identity bounds the quotient positive trace by $N-p$; restoring
the separated eigenvalue $p$ gives $\tr(T_+)\le N$.
The singleton estimate restores link $3$, and orthogonal splitting
preserves the positive spectrum and $N$. Thus $(\mathsf B_r)$ holds.

This proves $(\mathsf B_r)$ at every coordinate for integer $a\ge2$.
The two deductions at the start of this part give the relative
estimate and the singleton bound, including strictness.

\medskip\noindent\emph{The positive paw.}\
\phantomsection\label{thm:positive-paw-complete}
Let $C=aI_4+A_{\mathrm p}$, where $a\ge2$ is an integer.
Split $13$ into $1,3$, $23$ into $2,3$, and $123$ into $12,3$,
and merge equal links. These orthogonal splittings preserve the
positive spectrum and $N$.
The singleton bound permits us to delete the singleton at $3$.
It is therefore enough to consider the link types
$\mathcal H=\{0,1,2,01,02,12,03,012,013,023\}$.
If fewer than four coordinates are used, the Stieltjes and
positive-clique bounds apply. Assume that all four are used, and put
$M=tI_4+A_{\mathrm p}$ and $F(t)=\sum_Hb_H\Pi _{M^{1/2}\R^H}$.
Then $\tr F(t)=N$, and \eqref{eq:affine-gram-damping} implies that
its largest eigenvalue is nonincreasing. The
\hyperref[prop:conditioned-first-flow]{Schur-complement and inverse-metric comparisons}
give monotonicity of the sum of its three largest eigenvalues for
$t\ge7/3$. It remains to prove the same assertion for the sum of the
two largest eigenvalues on that interval.

Every remaining link containing $3$ also contains $0$. Thus
$R=I_4-t^{-1}e_0e_3^{\mathsf T}$ preserves each link coordinate
space. By the Gram comparison, replacing $M$ by the following matrix
does not change the spectrum of the sum of projections:
\[
 \widetilde M=R^{\mathsf T}MR
 =\begin{pmatrix}
 t&1&1&0\\1&t&1&-1/t\\1&1&t&-1/t\\
 0&-1/t&-1/t&t-1/t
 \end{pmatrix}.
\]
Number exterior pairs $(01,02,03,12,13,23)$ by $0,\ldots,5$.
Put $p=t^2-1$, ${p_1}=t^4-t^2-1$ and $r=t^3-t-1$, all
positive for $t>2$. Normalize the exterior matrix by
$G_{IJ}=\det\widetilde M[I,J]/\sqrt{\det\widetilde M[I]\det\widetilde M[J]}$.
The diagonal of $\bigwedge^2\widetilde M$ is
$(p,p,p,p,{p_1}/t^2,{p_1}/t^2)$.
Define
\begin{align*}
 {a_1}&=(t+1)^{-1},&{a_2}&=2t/p,\\
 {a_3}&=\tfrac12({p_1}'/{p_1}-p'/p)=\frac{t(p^2+1)}{p{p_1}},&
 {a_4}&={p_1}'/{p_1}-r'/r
 =\frac{(t-1)(t^2-t-1)(t^3+2t^2+2t-1)}{{p_1} r}.
\end{align*}
The $2\times2$ minors and their logarithmic derivatives give all
off-diagonal entries of $G$ and the coefficients in
$G'_{ij}=-d_{ij}G_{ij}$:
\[
\begin{array}{c|c|c}
 \text{entries}&\text{value}&d_{ij}\\\hline
 G_{01}=-G_{03}=G_{13}&(t+1)^{-1}&{a_1}\\
 G_{02}=G_{12}&-1/p&{a_2}\\
 G_{04}=G_{05}=G_{14}=G_{15}&-1/\sqrt{p{p_1}}&{a_2}+{a_3}\\
 G_{24}=G_{25}&\sqrt{p/{p_1}}&{a_3}\\
 -G_{34}=G_{35}&(t-1)/\sqrt{p{p_1}}&{a_1}+{a_3}\\
 G_{45}&r/{p_1}&{a_4}\\
 G_{23}&0&{a_1}\ \text{(chosen below)}
\end{array}
\]
For the combined rates, we used
$\tfrac12(p'/p+{p_1}'/{p_1})={a_2}+{a_3}$ and
${a_2}-(t-1)^{-1}={a_1}$.

To realize these rates as squared Euclidean distances, set
$\mathbf u_2=0$ and prescribe the Gram matrix $K\oplus L$ for
$(\mathbf u_0,\mathbf u_1,\mathbf u_3,\mathbf u_4,\mathbf u_5)$, where
\[
 K=\begin{pmatrix}
 {a_2}&{a_2}-{a_1}/2&{a_2}/2\\
 {a_2}-{a_1}/2&{a_2}&{a_2}/2\\
 {a_2}/2&{a_2}/2&{a_1}
 \end{pmatrix},\qquad
 L=\begin{pmatrix}{a_3}&{a_3}-{a_4}/2\\
                   {a_3}-{a_4}/2&{a_3}\end{pmatrix}.
\]
The antisymmetric eigenvalue of $K$ is ${a_1}/2$; its remaining
$2\times2$ block has a positive diagonal entry ${a_1}$ and determinant
\[
 {a_1}(2{a_2}-{a_1}/2)-{a_2}^2/2
 =\frac{{a_1}^2}{2}\left(3-\frac4{(t-1)^2}\right)
 \ge\frac{3{a_1}^2}{8}>0\qquad(t\ge7/3).
\]
Also ${a_3},{a_4}>0$ and
$2{a_3}-{a_4}=(p^2+2t)/(pr)>0$, so the two eigenvalues
${a_4}/2$ and $2{a_3}-{a_4}/2$ of $L$ are positive.
Thus $K\oplus L\succ0$, and vectors with this Gram matrix exist.
The identity
$\|\mathbf u_i-\mathbf u_j\|^2=\|\mathbf u_i\|^2+\|\mathbf u_j\|^2-2\langle \mathbf u_i,\mathbf u_j\rangle$
gives the entries in the table, including $d_{23}={a_1}$.
The exterior distance criterion therefore implies that
$\lambda_1(F)+\lambda_2(F)$ is nonincreasing for $t\ge7/3$.

If $a\ge3$ is an integer, then $t=a+x$ lies in this interval for
every $x>0$. If $a=2$, the endpoint bound
$F(7/3)=F_C(1/3)\succ\tfrac13I_4$ and the inertia identity show that
all four positive eigenvalues of $T$ exceed $1/3$. In either case,
the three partial sums are nonincreasing on the required intervals
between positive eigenvalues. The summation identity gives
$\tr(T_+)\le N$. Restoring the singleton by the preceding bound proves
the result for the original paw family.

\medskip\noindent\emph{Removing a pair of triangle links.}\
\phantomsection\label{prop:diamond-triangle-pair-core-shift}
Let $C_a=aI_4+A_{\mathrm d}$, $a\ge2$, and write
$H_-=012$, $H_+=023$, and $K=02$.
Subtract $0<b_0\le\min\{b_{H_-},b_{H_+}\}$ from both triangle
weights, add $b_0$ to the weight of $K$, and replace $a$ by $a+b_0$.
For the resulting block matrix $T'$ and weighted link size $N'$, we have
 \(\tr(T_+)=4{b_0}+\tr T'_+, \qquad N=4{b_0}+N'.\)
We also claim that $\tr(T_+)\le N$ whenever one triangle
$H_i$, $i\in\{+,-\}$, has weight $p>0$ and all other links are
contained in $H_{-i}$. Both assertions hold for real weights and
follow from the same identity for the triangle projections.

Put $M=tI+A_{\mathrm d}$ and $\Pi _L(t)=\Pi _{M^{1/2}\R^L}$,
where $t>2$. The normal lines to $M^{1/2}\R^{H_-}$ and
$M^{1/2}\R^{H_+}$ are generated by $M^{-1/2}e_3$ and
$M^{-1/2}e_1$. They are orthogonal because $(M^{-1})_{13}=0$,
as computed above. The hyperplanes intersect in $M^{1/2}\R^K$, so
\begin{equation}\label{eq:diamond-triangle-projections}
 \begin{gathered}
 \Pi _{H_-}(t)+\Pi _{H_+}(t)=I_4+\Pi _K(t),\\
 \Pi _{H_-}(t)\Pi _{H_+}(t)=\Pi _K(t)=\Pi _{H_+}(t)\Pi _{H_-}(t).
 \end{gathered}
\end{equation}

Let $\mathcal F(t)$ and $\mathcal F'(t)$ be the projection sums
before and after the change of weights, evaluated at the same $t$.
Equation~\eqref{eq:diamond-triangle-projections} gives
$\mathcal F(t)=b_0 I_4+\mathcal F'(t)$. If $x>{b_0}$ and $y=x-{b_0}$,
then $a+x=(a+{b_0})+y$ and
$n_+(T-xI)=n_+(\mathcal F(a+x)-xI) =n_+(\mathcal F'((a+{b_0})+y)-yI)=n_+(T'-yI)$.
For $0<x<{b_0}$, $\mathcal F(a+x)\succeq b_0 I_4$ gives
$n_+(T-xI)=4$. Thus its positive spectrum is obtained by padding
the positive spectrum of $T'$ with zeros to length four and adding
$b_0$ to each entry. This proves the trace identity even if the new
family uses fewer coordinates. The change in $N$ is $6b_0-2b_0=4b_0$.

For the second assertion put $H=H_{-i}$ and $t=a+x$.
By the same identity, the full sum of projections is
$p(I-\Pi _H(t))+p\Pi _K(t)+\sum_{L\subseteq H}b_L\Pi _L(t)$.
The first term acts on the one-dimensional complement of
$M^{1/2}\R^H$, and all other terms act within that space.
With $M_H=M[H]$, the matrix $J=M^{1/2}R_HM_H^{-1/2}$ satisfies
$J^{\mathsf T}J=I_3, \qquad JM_H^{1/2}R_{L\mid H}=M^{1/2}R_L\quad(L\subseteq H)$.
Thus $J$ identifies the restriction to $M^{1/2}\R^H$ with the
projection family of a block matrix $T'$ whose Gram matrix is
$C_a[H]$, whose pair $K$ has weight $p$, and whose other links are
the original ones. The orthogonal decomposition and the inertia
identity give
$n_+(T-xI)=\one_{\{p>x\}}+n_+(T'-xI),\qquad x>0$.
Thus the positive spectrum of $T$ consists of one copy of $p$ and
the positive spectrum of $T'$, with multiplicities. Consequently
$\tr(T_+)=p+\tr T'_+, \qquad N'=2p+\sum_{L\subseteq H}b_L|L|=N-p$.
Switching coordinates $1,2$ changes $C_a[H_-]$ into
$(a+1)I_3-J_3$, whereas $C_a[H_+]=(a-1)I_3+J_3$.
The Stieltjes and positive-clique bounds give
$\tr T'_+\le N'$, proving $\tr(T_+)\le N$.

Write $C=2I_4+A_{\mathrm d}$, $M(t)=tI_4+A_{\mathrm d}$, and
$\Pi _H(t)=\Pi _{M(t)^{1/2}\R^H}$.

\medskip\noindent\emph{The smallest eigenvalue for the mixed diamond.}\
\phantomsection\label{prop:diamond-normalized-third-flow}
Suppose that $a=2$, every link belongs to
$01,02,03,12,23,012,013,023,123$, and every present weight is at least
one. Assume also that all four coordinates are used, at least three
link types occur, their common intersection is empty, at least one
path triple $013$ or $123$ occurs, and the two triangle weights on
$012,023$ are not both positive. If $T$ has four positive eigenvalues,
let $\mu_4$ be the smallest. We prove that
$\lambda_{\min}(F(2+x))$ is nondecreasing for $x\ge\mu_4$.
Since $\tr F(t)=N$, this is equivalent to the third partial sum
being nonincreasing there.

For $x\ge8/7$, we have $t=2+x\ge22/7$, so the inverse-metric
comparison applies. For smaller $x$, we use
Lemma~\ref{prop:local-frame-crossing}.
In coordinate order $(0,1,3,2)$, set
$R(i,j,k,l;c,d)=(e_i+e_j,c(e_i-e_j),e_k,de_l)$ for a permutation
$(i,j,k,l)$ of $(1,2,3,4)$ and $c,d>0$; its determinant has
absolute value $2cd$. Give every listed type weight one:
\[
\begin{array}{c|l|c|c}
 j&\mathcal B_j&h_j&R_j\\\hline
 1,2&\{03,12,013\},\ \{01,23,013\}&8/7&R(1,2,3,4;1,4)\\
 3&\{02,23,013\}&8/7&R(1,4,3,2;3/2,5)\\
 4&\{23,012,013\}&8/7&I\\
 5&\{12,013,023\}&8/7&I\\
 6&\{02,013,123\}&8/7&I\\
 7&\{01,03,013,123\}&8/7&R(1,2,3,4;1,5)\\
 8&\{02,12,013\}&3/4&\diag(1,1,2,1)\\
 9&\{12,23,013\}&8/7&R(2,4,3,1;1/4,7/2)
\end{array}
\]

We first show how these collections cover the possible families of
link types. After signed reflection, assume that $013$ is present.
Suppose that no $\mathcal B_j$ with $j\ne8$, and no reflected copy
of such a collection, is contained in the family. For every other
present type, put $E_H=\{0,1,3\}\setminus H$. Then
\[
\begin{array}{c|cccccc}
 H&01,012&03,023&02&12&23&123\\\hline
 E_H&\{3\}&\{1\}&\{1,3\}&\{0,3\}&\{0,1\}&\{0\}.
\end{array}
\]
The empty common intersection gives $\bigcup_HE_H=\{0,1,3\}$.
Choose a minimal subcover $E_{H_1},\ldots,E_{H_{r_0}}$ and
$a_j\in E_{H_j}\setminus\bigcup_{i\ne j}E_{H_i}$.
The $a_j$ are distinct, and no $E_H$ is all of $\{0,1,3\}$.
Thus $2\le r_0\le3$; if $r_0=3$, every $E_{H_j}=\{a_j\}$.
Thus $123$ occurs together with one of $01,012$
and one of $03,023$. The choice of both triangles is forbidden;
the other three choices contain $\mathcal B_7$, reflected
$\mathcal B_5$, or reflected $\mathcal B_4$, respectively.
For $r_0=2$, the minimal covers are two doubletons or a doubleton
and its missing singleton. Equivalently,
\[
 \{H_1,H_2\}\in\binom{\{02,12,23\}}2\ \cup\
 \bigl\{\{02,123\},\{12,03\},\{12,023\},\{23,01\},\{23,012\}\bigr\}.
\]
These are precisely $\mathcal B_j\setminus\{013\}$ for
$j\in\{1,2,3,4,5,6,8,9\}$.
Hence $\mathcal B_8$ occurs. Adding $03,23,023,123$ would give
$\mathcal B_1,\mathcal B_3,\mathcal B_5,\mathcal B_6$, respectively.
Consequently the only remaining possibility is
$\{02,12,013\}\subseteq\mathcal H\subseteq\{01,02,12,012,013\}$,
or its signed reflection. This case will be treated after the
matrix estimates below.

Take $J=\{0,1,2,3\}$ and $g=0$ in
Lemma~\ref{prop:local-frame-crossing}, and put $f_h(r)=r/(r+h)^2$.
For $0<h\le8/7$, the spectral decomposition gives the uniform bound
\[
 E_h=\frac{\min\{h,1\}}{(1+\min\{h,1\})^2}\Pi _1
       +\frac{3h}{(3+h)^2}\Pi _3+\frac{4h}{(4+h)^2}\Pi _4
 \preceq\widetilde E:=\frac14\Pi _1+\frac15\Pi _3+\frac7{40}\Pi _4.
\]
Indeed, $rh/(r+h)^2$ is increasing for $h<r$, and the endpoint
values for $r=3,4$ satisfy $168/841<1/5$ and $14/81<7/40$.
With $A_H=C[H]-2I$, define local lower bounds by
\[
 L_h(H)=\begin{cases}
 q_2I,&|H|=2,\\q_-C[H],&H=012,\\q_+I,&H=023,\\
 q_pI-q_aA_H,&H=013,123,
 \end{cases}
 \qquad
 \begin{array}{c|ccccc}
 h&q_2&q_-&q_+&q_p&q_a\\\hline
 8/7&11/63&7/125&3/20&9/50&1/90\\
 3/4&1/5&0&0&13/50&9/200
 \end{array}
\]
Only the pair and path bounds are used from the second row.
The spectra for a pair and the two triangles are $\{1,3\}$, $\{0,3,3\}$,
$\{1,1,4\}$, respectively. Thus their bounds
$f_h(C[H])\succeq L_h(H)$ follow from
\[
 \begin{aligned}
 &\min\{f_{8/7}(1),f_{8/7}(3)\}=\frac{147}{841}>\frac{11}{63},\quad
 \frac{49}{841}>\frac7{125},\\
 &\min\{f_{8/7}(1),f_{8/7}(4)\}=\frac{49}{324}>\frac3{20},\quad
 \min\{f_{3/4}(1),f_{3/4}(3)\}=\frac{16}{75}>\frac15.
 \end{aligned}
\]
For a path, $\Spec A_H=\{0,\pm\sqrt2\}$. Put $t=2+h$.
For $\lambda^2=2$,
\[
 \begin{aligned}
 &(t+\lambda)^2\{f_h(2+\lambda)-q_p+q_a\lambda\}
       =u_h+v_h\lambda,\\
 &u_h=2-q_p(t^2+2)+4q_at,\qquad
 v_h=1-2q_pt+q_a(t^2+2).
 \end{aligned}
\]
For $\lambda=0$ the numerator is $2-q_pt^2$. The exact values are
\[
 \begin{array}{c|ccc}
 h&v_h&u_h-3v_h/2&2-q_pt^2\\\hline
 8/7&2/3675&2/2205&272/1225\\
 3/4&1/3200&53/6400&27/800
 \end{array}
\]
All entries are positive. Since $|\lambda|=\sqrt2<3/2$,
$u_h+v_h\lambda\ge u_h-3v_h/2>0$; the last column handles
$\lambda=0$. Hence $f_h(C[H])\succeq L_h(H)$ for each path.

Put
\[
 \begin{gathered}
 W_j=\left[\sum_{H\in\mathcal B_j}R_HL_{h_j}(H)R_H^{\mathsf T}
                   -\widetilde E\right][0132,0132],\qquad
 Z_j=12600R_j^{\mathsf T}W_jR_j,\\
 m_j=\min_i\left\{(Z_j)_{ii}-\sum_{k\ne i}|(Z_j)_{ik}|\right\}.
 \end{gathered}
\]
Addition of the local matrices gives
\[
 (m_1,\ldots,m_9)=(23,23,350,1143,415,793,2661,630,21).
\]
For example $W_1=W_2$, and the four differences in
\eqref{eq:diagonal-squares} for $Z_1$ are
$(3441,641,23,340)$.
Thus \eqref{eq:diagonal-squares} gives $Z_j\succ0$, and congruence
by the invertible $R_j$ gives $W_j\succ0$.
All present weights are at least one, so
Lemma~\ref{prop:local-frame-crossing} gives
\begin{equation}\label{eq:diamond-frozen-bound}
 V_{0,J}(x)[0132,0132]\succeq W_j\succ0
 \qquad(0<x\le h_j,\quad\mathcal B_j\subseteq\mathcal H).
\end{equation}
Hence the local comparison holds on the stated interval for each
$\mathcal B_j$.

Now consider a family satisfying the remaining inclusion, or its
signed reflection. Use the
\hyperref[prop:diamond-normal-data]{normal vectors computed above}.
The planes $02$ and $12$ are both present and span $\scrU$. Thus
$F_{\scrU}\succ0$ on $\scrU$, and in particular
$F_{\scrU}\mathbf n_2\ne0$. The strict hyperplane-resolvent comparison
applies for $t>1+\sqrt3$. Its interval overlaps the local interval,
since $2+3/4>1+\sqrt3$.

For every other collection the local comparison holds up to $x=8/7$,
and the inverse-metric comparison applies from that point onwards.
By the inertia identity, $\lambda_{\min}(F(2+x))\le x$ for
$x\ge\mu_4$, as required by the local comparison.
Signed reflection preserves $C$ and the
conclusions above. The intervals therefore cover all $x\ge\mu_4$,
which proves the assertion.

Continue to write $M=tI_4+A_{\mathrm d}$, with $t>2$.
In exterior order $(01,02,03,12,13,23)$, put
$G=\bigwedge\nolimits^2M,\qquad \widehat{\Pi}_H=R_HM[H]^{-1}R_H^{\mathsf T}M,\qquad J_H=G\widehat{\Pi}_H^{[2]}$.
Write $\Pi _H=\Pi _{M^{1/2}\R^H}$. The definition of the additive
compound and $G^{1/2}=\bigwedge^2M^{1/2}$ give
\[
 M^{1/2}\widehat{\Pi}_HM^{-1/2}=\Pi _H,\qquad
 G^{-1/2}J_HG^{-1/2}=\Pi _H^{[2]},\qquad
 G^{-1/2}\Bigl(\sum_Hb_HJ_H\Bigr)G^{-1/2}
       =\Bigl(\sum_Hb_H\Pi _H\Bigr)^{[2]}.
\]
In particular, $J_H$ is symmetric. The largest eigenvalue of the
last matrix equals the sum of the two largest eigenvalues of
$\sum_Hb_H\Pi_H$.

\medskip\noindent\emph{The sum of the two largest eigenvalues.}\
\phantomsection\label{thm:diamond-full-nine-second-flow}
Let $\mathbf b=(b_H)_H$ be arbitrary nonnegative real weights on
the five nonorthogonal pairs and the four triples. We prove that the
sum of the two largest eigenvalues of
$F_{\mathbf b}(t)=\sum_Hb_H\Pi _{M(t)^{1/2}\R^H}$ is nonincreasing for
$t\ge3$. If $b_{013}=b_{123}=0$, the comparison holds for $t\ge9/4$.

We apply Lemma~\ref{prop:positive-map-comparison} on the exterior
square with the following positive linear maps.
Let $d_0\in\{0,1\}$, and put $j_0=+$ when $d_0=0$
and $j_0=-$ when $d_0=1$.
For ${d_0}=0$ take $(s,d)=(13,02)$, and for ${d_0}=1$ reverse
these indices. Put
\[
 \begin{gathered}
 r_0=t^2-d_0,\qquad h_0=t^2-1+d_0,\qquad
 w_0=\frac{t^2}{t^2-1},\\
 h_e=\frac{t}{2(t+e)},\qquad
 p_e=(1-d_0)h_e+\frac{d_0}2,\qquad
 x_e=\frac{1-d_0}2+d_0h_e.
 \end{gathered}
\]
For exterior indices $ij,kl$, the minor formula
$G_{ij,kl}=M_{ik}M_{jl}-M_{il}M_{jk}$ gives
$|G_{ij,kl}|\in\{1,2,t,t-1,t+1,t^2,t^2-1\}$; hence no entry
vanishes for $t>2$. Put
\[
 L_{j_0}=E_{dd}-\frac4{h_0} E_{ds},\qquad
 (K_{j_0})_{IJ}=
 \frac{tG'_{IJ}/2-(L_{j_0}^{\mathsf T}GL_{j_0})_{IJ}/4}{G_{IJ}}.
\]
For matrices of the same order, write $A\circ B$ for their
entrywise product. By definition,
$K_{j_0}\circ G+L_{j_0}^{\mathsf T}GL_{j_0}/4=tG'/2$.
In particular $G_{dd}={h_0}$, $G_{ss}={r_0}$, $G_{ds}=2$, and
\[
 \frac14L_{j_0}^{\mathsf T}GL_{j_0}
 =\frac{h_0}4(e_d-4e_s/{h_0})(e_d-4e_s/{h_0})^{\mathsf T}.
\]
Outside the $d,s$ block,
$(K_{j_0})_{IJ}=(t/2)\partial_t\log|G_{IJ}|$.
With $e=\pm1$, $(i_+,j_+)=(01,12)$, $(i_-,j_-)=(03,23)$,
and $\one=(1,1)^{\mathsf T}$, the same minor formula gives,
in the order $(i_+,j_+,i_-,j_-,d,s)$,
\[
 K_{j_0}=
 \begin{pmatrix}
 B_+&0&p_+\one&x_+\one\\
 0&B_-&p_-\one&x_-\one\\
 p_+\one^{\mathsf T}&p_-\one^{\mathsf T}
       &t^2/{h_0}-1/4&1/2\\
 x_+\one^{\mathsf T}&x_-\one^{\mathsf T}
       &1/2&t^2/{r_0}-4/({h_0}{r_0})
 \end{pmatrix},\qquad
 B_e=\begin{pmatrix}{w_0}&h_e\\h_e&{w_0}\end{pmatrix}.
\]
Since ${w_0}=h_++h_-$, both $B_e$ are positive definite.
For $z,w\in\R$,
$(z\one)^{\mathsf T}B_e^{-1}(w\one) =\frac{2zw}{{w_0}+h_e}$.
Consequently, for the Schur complement $S_{d_0}$ of
$B_+\oplus B_-$ in $K_{j_0}$,
\[
 (S_{d_0})_{11}=\frac{3t^2+1+20{d_0}}{4(9t^2-1)},\qquad
 (S_{d_0})_{12}=-\frac{3(t^2-1)}{2(9t^2-1)}.
\]
For $u=t^2-5\ge0$, both determinants are
\[
 \det S_{d_0}=
 \frac{(4+5{d_0})u^2+(24+45{d_0})u+16+20{d_0}}
      {4t^2(t^2-1)(9t^2-1)}>0\qquad(t\ge9/4).
\]
Its first diagonal entry and determinant are positive, so
$S_{d_0}\succ0$ and $K_{j_0}\succ0$.
If $K_{j_0}=\sum_j\mathbf v_j\mathbf v_j^{\mathsf T}$, then
$K_{j_0}\circ X=\sum_j\diag(\mathbf v_j)X\diag(\mathbf v_j)$.
Therefore the entrywise product with $K_{j_0}$ is a positive linear
map. Set
\[
 \Psi_{j_0}(X)=K_{j_0}\circ X+\frac14L_{j_0}^{\mathsf T}XL_{j_0},
 \qquad \mathcal R_{j_0}(J)=\frac2t\Psi_{j_0}(J)-J'.
\]
Its defining identity gives $\Psi_{j_0}(G)=tG'/2$ on the entire
exterior space.

We first verify the comparison inequality for the sums
$J_{012}+J_L$, $J_{023}+J_L$, and $J_{02}+2J_L$,
where $L\in\{013,123\}$. These sums will occur in the decomposition
of the weight cones below. Retain $H_-=012$ and $H_+=023$, with
subscripts $+,-$ denoting
$1,-1$. For $i\in\{-1,1\}$ and
$K=H_{-i}\cap013\in\{01,03\}$,
$(J_{H_{-i}}+J_{013})e_K=4Ge_K$.
The $K$th image column depends only on the $K$th input column.
Together with the identity for the metric, this implies that the
$K$th row and column of the comparison difference vanish. Put
\begin{align*}
 f&=\frac{4t(t^2-1)}{(t^2-2)^2},&
 h&=\frac{2(t^2-4)}{t(t^2-2)},&
 l_i&=\frac{2(t-i)(t+2i)}
                         {t^2(t+i)},\\
 g_i&=
 \frac{(t+i)(t^4+4i t^3-10t^2+8)}
      {t^2(t^2-2)^2},&
 a_i&=\frac{2(t-i)}{(t+i)^2},&
 b_i&=\frac{2(2t+i)}{t(t+i)^2}.
\end{align*}
In the orders $(02,03,12,13,23)$ for $i=1$ and
$(02,01,23,13,12)$ for $i=-1$, the reduced difference matrix is
\[
 B_{{j_0},i}=
 \begin{pmatrix}
 {a_0}&0&g_i&{b_0}&g_{-i}\\
 0&a_i&0&i b_i&-a_i/2\\
 g_i&0&f&0&0\\
 {b_0}&i b_i&0&{c_0}&-i b_i\\
 g_{-i}&-a_i/2&0&-i b_i&f+a_i
 \end{pmatrix},
\]
where, with $y_+=4/(t^2-1)^2$ and $y_-=4/t^4$,
direct multiplication of $J_H=G\widehat{\Pi}_H^{[2]}$ gives
\begin{align*}
 {a_0}&=f-{d_0}\,
 \frac{4i(t+2i)y_-}{(t+i)(t^2-2)},\\
 {c_0}&=tb_i+(1-{d_0})\,
 \frac{4i(t-i)(t+2i)y_+}{t^2(t^2-2)},\\
 {b_0}&=(1-{d_0})h+{d_0}l_i.
\end{align*}
Eliminating coordinates two and three leaves
\[
 D_{{j_0},i}=
 \begin{pmatrix}
 {a_0}-g_i^2/f&{b_0}&g_{-i}\\
 {b_0}&{c_0}-b_i^2/a_i&-i b_i/2\\
 g_{-i}&-i b_i/2&f+3a_i/4
 \end{pmatrix}.
\]
For $t\ge3$ one has $l_->0$ and
\[
 l_+-h=\frac{4(t+2)}{t^2(t+1)(t^2-2)}>0,\qquad
 h-l_-=\frac{4(t-2)}{t^2(t-1)(t^2-2)}>0.
\]
The signs of the corrections to ${a_0},{c_0}$ show that the smaller
diagonal entries occur at ${d_0}=1$ when $i=1$ and at ${d_0}=0$
when $i=-1$. The same choices give the larger value of $D_{12}$,
by the displayed differences. All other off-diagonal absolute values
are unchanged. Denote these two matrices by $D^+,D^-$.

For $t\ge3$, consider the two fixed matrices
\[
 M_+=\begin{pmatrix}7/2&-2&-3/2\\-2&8/5&-7/16\\-3/2&-7/16&5\end{pmatrix},
 \qquad
 M_-=\begin{pmatrix}15/4&-2&-3\\-2&4&-5/4\\-3&-5/4&11/2\end{pmatrix}.
\]
Their leading principal minors are, respectively,
$(7/2,8/5,2829/2560)$ and $(15/4,11,233/64)$, so both are positive definite.
It suffices to prove the bounds
\[
 tD_{ii}^{\pm}>(M_\pm)_{ii},\qquad
 t|D_{ij}^{\pm}|\le -(M_\pm)_{ij}\quad(i\ne j).
\]
They imply $t\mathbf z^{\mathsf T}D^\pm\mathbf z
\ge |\mathbf z|^{\mathsf T}M_\pm|\mathbf z|>0$ for every nonzero $\mathbf z$,
where $|\mathbf z|$ denotes coordinatewise absolute value.

For the off-diagonal entries, the identities
\[
\begin{aligned}
 2-tl_+&=\frac4{t(t+1)},&2-th&=\frac4{t^2-2},\\
 \frac7{16}-\frac{tb_+}2&=\frac{(t-3)(7t+3)}{16(t+1)^2},&
 \frac54-\frac{tb_-}2&=\frac{(t-3)(5t-3)}{4(t-1)^2}
\end{aligned}
\]
give the bounds in positions $12$ and $23$. For position $13$, substitution gives
\[
\begin{aligned}
 t(t^2-2)^2(1-tg_-)&=(t^2-2)(5t^2+2t)+8>0,\\
 2t(t^2-2)^2(3/2+tg_-)&=(t-3)\{t^2(5t^2-9)+t(5t^2-7)+7\}+5>0,\\
 t(t^2-2)^2(3-tg_+)&=(t-3)\{t^2(2t^2-3)+t^3+t+7\}+13>0.
\end{aligned}
\]
Since $g_+>0$, these identities give $|tg_-|<3/2$ and $|tg_+|<3$.

For the diagonal bounds, direct substitution gives
\[
\begin{aligned}
 4t^4(t^2-1)(t^2-2)^2(tD^+_{11}-7/2)&=P_+,\\
 2t^4(t+1)(t^2-2)^2(tD^-_{11}-15/4)&=P_-,\\
 \frac{5t(t-1)(t+1)^2}{2}(tD^+_{22}-8/5)
   &=(t-3)(6t^3+9t^2+6t+2)+1,\\
 2(t+1)^2(t^2-2)^2(tD^+_{33}-5)&=Q,
\end{aligned}
\]
where the following expressions are positive for $t\ge3$:
\[
\begin{aligned}
 P_+={}&t^8(t-5)^2+4t^4(20t^3-12t^2-70t-23)
          +128t(4t^2-3)+32(7t^2-2),\\
 P_-={}&5t^8+4t^5\{5t(t-1)-12\}
          +20t^2(3t^2-4)+16t(3t^2-2)+32,\\
 Q={}&t^4\{(t-7/2)^2+23/4\}
          +4(t-2)(19t^2+38t+53)+4t^2+384.
\end{aligned}
\]
The cubic expression in $P_+$ satisfies
$20t^3-12t^2-70t-23\ge48t^2-70t-23\ge74t-23>0$.
Finally, put $d=16(t+1)(t-2)/\{t^2(t^2-2)(t^2-1)^2\}$.
Then
\[
 tD^-_{22}=t^2b_--\frac{tb_-^2}{a_-}-td
 >4+\frac{6t-4-8-1}{(t-1)^2}>4,
\]
because $t^2b_-=4+(6t-4)/(t-1)^2$,
$tb_-^2/a_-<8/(t-1)^2$, and
$td<1/(t-1)^2$; the last inequality follows from
$t(t+1)(t^2-2)\ge28t>16(t-2)$.
Also $tf>4$ and $ta_->2$, so $tD^-_{33}=tf+3ta_-/4>11/2$.
All the stated bounds follow, and hence $D^+,D^-\succ0$.
The other two matrices $D_{j_0,i}$ have diagonal entries no smaller
than those of the corresponding $D^\pm$, and off-diagonal absolute
values no larger. The same quadratic-form comparison therefore proves
their positive definiteness.

The positive eliminated pivots $a_i,f$ give $B_{{j_0},i}\succ0$.
Restoring the zero row and column proves the required semidefinite
inequality for these triangle--path sums.

We now extend this comparison to arbitrary nonnegative weights. Put $d_*(\mathbf b)=2b_{02}+b_{012}+b_{023}-b_{013}-b_{123}$ and
$\mathcal C_\pm=\{\mathbf b\in\R_{\ge0}^{9}:\ \pm d_*(\mathbf b)\ge0\}$.
Use $\Psi_\pm$ on $\mathcal C_\pm$; either choice is valid if
$d_*(\mathbf b)=0$.
By Lemma~\ref{prop:positive-map-comparison}, applied
with positive map $2\Psi_{j_0}/t$ and $L=0$, it suffices to
prove $\mathcal R_{j_0}(J_{\mathbf b})\succeq0$ for
$J_{\mathbf b}=\sum_Hb_HJ_H$.

For each link type $H$, let $e_H$ be its coordinate unit vector
in the weight space, so $J_{e_H}=J_H$.
Both cones have the generators $e_{01},e_{03},e_{12},e_{23}$ and
\[
 e_{012}+e_L,\qquad e_{023}+e_L,\qquad e_{02}+2e_L,
 \qquad L\in\{013,123\}.
\]
The additional generators are $e_{02},e_{012},e_{023}$ for
$\mathcal C_+$ and $e_{013},e_{123}$ for $\mathcal C_-$.
For this decomposition, match the summands of
$2b_{02}+b_{012}+b_{023}$ with those of $b_{013}+b_{123}$ until the
smaller total is exhausted. Each matched term is a nonnegative
multiple of a displayed generator with $d_*(\mathbf b)=0$. The
unmatched terms are nonnegative multiples of the additional
single-link generators. For these generators, we check that every
congruence operator in $\Psi_{j_0}$ preserves both level spaces
$\scrV_\ell(H)$, $\ell=1,2$.

For the positive cone, diagonal maps and the shear $E_{ds}$ preserve
the two level spaces
$\{ij:|\{i,j\}\cap H|\ge\ell\}$, $\ell=1,2$, for
$H=01,02,03,12,23,012,023$; for the negative cone the corresponding
list is $01,03,12,23,013,123$.
For its shear $E_{ds}$ the differences of these counts are
\[
 |d\cap H|-|s\cap H|=
 \begin{cases}
 0,&H=01,03,12,23,\\
 2,&{d_0}=0,\ H=02,\\
 1,&{d_0}=0,\ H=012,023,\ \text{or }{d_0}=1,\ H=013,123.
 \end{cases}
\]
Each difference is nonnegative, so $E_{ds}$ preserves both level
spaces. Diagonal maps preserve every coordinate space.
A Gram factorization of $K_{j_0}$ writes its Schur-product map
as a sum of diagonal congruences. The remaining congruence operator
$L_{j_0}$ is a linear combination of $E_{dd}$ and this shear.
The exterior-filtration identity and the invariant-subspace assertion
of Lemma~\ref{prop:gram-comparison} therefore give
$\mathcal R_{j_0}(J_H)\succeq0$ for every listed single-link generator.

The difference matrices above prove the comparison for
$J_{012}+J_{013}$ and $J_{023}+J_{013}$. For the corresponding sums
with $123$, use the signed permutation
 \(Qe_0=e_2,\quad Qe_2=e_0,\quad Qe_1=-e_1,\quad Qe_3=e_3, \qquad W=\bigwedge\nolimits^2Q.\)
The five nonzero off-diagonal entries give $Q^{\mathsf T}AQ=A$;
$W$ exchanges $01,12$ and $03,23$, and changes the signs of
$02,13$. Hence it fixes $G,J_{012},J_{023}$ and exchanges
$J_{013},J_{123}$. The permutation induced by $W$ exchanges the two indices of each
side block and fixes $d,s$; hence it preserves $K_{j_0}$.
Also $W^{\mathsf T}L_{j_0} W=L_{j_0}$. Entrywise multiplication
and the congruence term therefore give
\[
 \Psi_{j_0}(W^{\mathsf T}BW)=W^{\mathsf T}\Psi_{j_0}(B)W,
 \qquad
 \mathcal R_{j_0}(W^{\mathsf T}JW)
       =W^{\mathsf T}\mathcal R_{j_0}(J)W.
\]
Thus the comparison difference is positive semidefinite for the
reflected generators as well.
Finally, \eqref{eq:diamond-triangle-projections} gives
$J_{012}+J_{023}=2G+J_{02},\qquad \mathcal R_{j_0}(G)=0$,
so
$\mathcal R_{j_0}(J_{02}+2J_L) =\mathcal R_{j_0}(J_{012}+J_L) +\mathcal R_{j_0}(J_{023}+J_L)\succeq0$.
By linearity,
$\mathcal R_{j_0}(J_{\mathbf b})\succeq0$ for every weight vector in that
cone. Lemma~\ref{prop:positive-map-comparison} gives the required
monotonicity of the second partial sum for $t\ge3$.
If both path weights vanish, only the preceding single-link generators
of the positive cone occur. Since $\Psi_+$ is positive for $t\ge9/4$, the
same argument proves monotonicity on that larger interval.

\medskip\noindent\emph{Collections without path triples.}\
\phantomsection\label{prop:seven-type-budget}
Suppose that $C=2I_4+A_{\mathrm d}$ and that the present links are
among $01$, $02$, $03$, $12$, $23$, $012$, $023$, with positive integer
weights. Assume that all four coordinates are used, the intersection
of the present links is empty, and at most one of $012,023$ occurs.
We prove $\tr(T_+)\le N$.

Write the positive eigenvalues of $T$ as
$\mu_1\ge\cdots\ge\mu_\ell>0$, $\ell\le4$, and put
$F(t)=\sum_Hb_H\Pi _{(tI+A_{\mathrm d})^{1/2}\R^H}$.
Then $\tr F(t)=N$, and \eqref{eq:affine-gram-damping} implies that
its largest eigenvalue is nonincreasing.

Since all four coordinates are used and every present weight is at least
one, the bound \eqref{eq:cover-floor} gives, for $x>0$,
\[
 F(2+x)\succeq\frac{C+xI}{4+x},
 \qquad
 \lambda_3(F(2+x))\ge\frac{1+x}{4+x}.
\]
Here $\lambda_3(C)=1$ by \eqref{eq:diamond-spectral-projectors}.
At $x=1/4$, the right side is $5/17>1/4$.
The inertia identity gives $\mu_3>1/4$.
Neither path triple occurs, so the
\hyperref[thm:diamond-full-nine-second-flow]{exterior comparison}
applies to every required interval for the second partial sum:
on those intervals $t\ge2+\mu_3>9/4$.
If $\ell\le3$, the summation identity gives $\tr(T_+)\le N$.
Suppose $\ell=4$. Since $\tr F(t)=N$, monotonicity of its third
partial sum is equivalent to the smallest eigenvalue being
nondecreasing. We shall prove the latter assertion for
$t\ge2+\mu_4$.

Put $A=A_{\mathrm d}$, $M=tI+A$, and
$\mathbf n=(1,-1,-1,0)^{\mathsf T},\qquad \mathbf p=(1,0,1,1)^{\mathsf T},\qquad I_H=R_HR_H^{\mathsf T}$.
Multiplication gives
 \(A\mathbf n=-2\mathbf n,\quad A\mathbf p=2\mathbf p,\quad \mathbf n^{\mathsf T}\mathbf p=0,\quad A^2-I=\mathbf n\mathbf n^{\mathsf T}+\mathbf p\mathbf p^{\mathsf T},\quad \|\mathbf n\|^2=\|\mathbf p\|^2=3.\)
Since $(tI-A)M=(t^2-1)I-\mathbf n\mathbf n^{\mathsf T}-\mathbf p\mathbf p^{\mathsf T}$,
$(t^2-1)M^{-1} =tI-A+\frac{\mathbf n\mathbf n^{\mathsf T}}{t-2} +\frac{\mathbf p\mathbf p^{\mathsf T}}{t+2}$.
For a correlated pair, $A[H]^2=I$, so
$(t^2-1)M[H]^{-1}=tI-A[H]$.
The two triangle matrices satisfy
$A[012]=I-\mathbf n[012]\mathbf n[012]^{\mathsf T},\qquad A[023]=\mathbf p[023]\mathbf p[023]^{\mathsf T}-I$.
Using $\mathbf n[012]^{\mathsf T}\mathbf n[012]=\mathbf p[023]^{\mathsf T}\mathbf p[023]=3$
and multiplying the displayed inverses by their matrices gives
\begin{align*}
 (t^2-1)M[012]^{-1}
 &=tI-A[012]+\frac{\mathbf n[012]\mathbf n[012]^{\mathsf T}}{t-2},\\
 (t^2-1)M[023]^{-1}
 &=tI-A[023]+\frac{\mathbf p[023]\mathbf p[023]^{\mathsf T}}{t+2}.
\end{align*}
Write
$D=\sum_Hb_HI_H$ and $E=\sum_Hb_HR_HA[H]R_H^{\mathsf T}$.
For any fixed $\lambda\ge0$,
\begin{equation}\label{eq:seven-cleared-pencil}
\begin{aligned}
 Q_\lambda(t)
 &:=(t^2-1)\left(\sum_Hb_HR_HM[H]^{-1}R_H^{\mathsf T}
                                      -\lambda M^{-1}\right)\\
 &=t(D-\lambda I)-(E-\lambda A)
       +\frac{b_{012}-\lambda}{t-2}\mathbf n\mathbf n^{\mathsf T}
       +\frac{b_{023}-\lambda}{t+2}\mathbf p\mathbf p^{\mathsf T}.
\end{aligned}
\end{equation}
Since $F(t)-\lambda I=(t^2-1)^{-1}M^{1/2}Q_\lambda(t)M^{1/2}$,
these two matrices have the same inertia.

Fix $t_0>2$ and put $\lambda=\lambda_{\min}(F(t_0))\ge0$.
For $\mathbf w_i=M(t_0)^{-1/2}e_i$, one has $\Pi _H\mathbf w_i=0$ when $i\notin H$.
The Rayleigh quotient at $\mathbf w_i$ and the orthogonal
decomposition associated with each $\Pi_H$ give
\begin{equation}\label{eq:dual-vector-strictness}
 d_i-\lambda\ge
 \sum_{H\ni i}b_H\frac{\|(I-\Pi _H)\mathbf w_i\|^2}{\|\mathbf w_i\|^2}\ge0.
\end{equation}
Since $\mathbf w_i\in M^{1/2}\R^H$ exactly when
$M^{-1}e_i\in\R^H$, the right-hand sum is positive whenever
some present $H\ni i$ satisfies $M^{-1}e_i\notin\R^H$.
Every entry of columns $0,2$ of $M^{-1}$ is nonzero; columns $1,3$
are nonzero precisely on $012,023$, respectively.
Differentiating $Q_\lambda$ with $\lambda$ fixed gives
\[
 Q_\lambda'(t)=D-\lambda I
       +\frac{\lambda-b_{012}}{(t-2)^2}\mathbf n\mathbf n^{\mathsf T}
       +\frac{\lambda-b_{023}}{(t+2)^2}\mathbf p\mathbf p^{\mathsf T}.
\]
For $\lambda\ge\max(b_{012},b_{023})$ every summand is positive semidefinite.
If $D-\lambda I\succ0$, then $Q_\lambda'\succ0$.

Suppose $b_{012}=0$. The support criterion yields $\lambda<d_0,d_1,d_2$.
It also gives $\lambda<d_3$ unless $023$ is the only link containing
$3$, in which case $d_3=b_{023}$.
For $\lambda<b_{023}$, write
\begin{align*}
 Q_\lambda'
 &=D-\lambda I-(b_{023}-\lambda)I_{023}
       +\frac\lambda{(t-2)^2}\mathbf n\mathbf n^{\mathsf T}
       +(b_{023}-\lambda)\left(I_{023}-\frac{\mathbf p\mathbf p^{\mathsf T}}{(t+2)^2}\right)\\
 &\succeq(d_1-\lambda)E_{11}
       +(b_{023}-\lambda)\left(1-\frac3{(t+2)^2}\right)I_{023}\succ0.
\end{align*}
Here $d_i\ge b_{023}$ on $023$ and $\|\mathbf p\|^2=3$ prove the inequality.
For $\lambda\ge b_{023}$, the derivative is positive definite unless
$\lambda=b_{023}=d_3$ and only $023$ contains $3$.
In that exception every other link lies in $012$.
Integrating at fixed $\lambda$ gives
\[
 Q_\lambda(t)=Q_\lambda(t_0)+\int_{t_0}^tQ_\lambda'(s)\,ds\succeq0,
 \qquad F(t)\succeq\lambda I\quad(t\ge t_0),
\]
with strict inequalities whenever the derivative is positive definite.
It follows that the smallest eigenvalue is nondecreasing. The
increase is strict unless the stated triangle decomposition occurs.
The argument does not require the smallest eigenvalue to be simple.

Now suppose that $012$ occurs, so $023$ is absent. If all other links
are contained in $023$, the
\hyperref[prop:diamond-triangle-pair-core-shift]{triangle projection identity}
gives $\tr(T_+)\le N$. Otherwise $01$ or $12$ occurs. Since coordinate
$3$ is used, at least one of $03,23$ also occurs.

We first establish two lower bounds at fixed parameter values. For
$\mathcal B_1=\{03,12,012\}$ and
$\mathcal B_2=\{01,03,23,012\}$, give each listed type weight one and put
$B_j(t)=\sum_{H\in\mathcal B_j}R_HM(t)[H]^{-1}R_H^{\mathsf T}$.
Substitution in \eqref{eq:seven-cleared-pencil} gives
\[
 2052\bigl(B_1(14/5)-\tfrac45M(14/5)^{-1}\bigr)
 =W_1:=\begin{pmatrix}
 1033&-135&-185&-110\\
 -135&1083&435&0\\
 -185&435&1033&190\\
 -110&0&190&118
 \end{pmatrix},
\]
\[
 31320\bigl(B_2(22/7)-\tfrac87M(22/7)^{-1}\bigr)
 =W_2:=\begin{pmatrix}
 19367&-2583&161&-280\\
 -2583&9063&-945&0\\
 161&-945&8279&-280\\
 -280&0&-280&8720
 \end{pmatrix}.
\]
Let $S_1=\diag(2,2,3,7)$ and $S_2=I$.
The diagonal entry in each row of $S_jW_jS_j$, minus the sum of the
absolute values of its off-diagonal entries, is as follows:
\[
 \begin{array}{c|rrrr}
 j&r_1&r_2&r_3&r_4\\\hline
 1&942&1182&1587&252\\
 2&16343&5535&6893&8160
 \end{array}
\]
For example,
$r_{1,4}=49\cdot118-14\cdot110-21\cdot190=252$.
Identity~\eqref{eq:diagonal-squares} gives $W_j\succ0$.
Congruence by $M(t)^{1/2}$ therefore gives
$F_{\mathcal B_1}(14/5)\succ(4/5)I$ and
$F_{\mathcal B_2}(22/7)\succ(8/7)I$.
The inequalities remain strict after adding nonnegative multiples of
projections. Hence the inertia identity gives $\mu_4>4/5$ if
$\mathcal B_1\subseteq\mathcal H$, and $\mu_4>8/7$ if
$\mathcal B_2\subseteq\mathcal H$.
Signed reflection preserves both bounds.

If both $03,23$ occur, then the support contains $\mathcal B_2$ or
its signed reflection. Thus $\mu_4>8/7$, and the inverse-metric
comparison applies throughout $t\ge2+\mu_4>22/7$.
Suppose that $03$ occurs and $23$ does not. If $12$ were absent,
coordinate $0$ would belong to every present link, contrary to the
hypothesis. Hence $12$ occurs, and the support contains $\mathcal B_1$
and is contained in $\{01,02,03,12,012\}$. Since $\mu_4>4/5$, it is
enough to prove the remaining comparison for $t\ge14/5$. The case
where $23$ occurs and $03$ does not follows by signed reflection.

Write the weights on $01,02,03,12,012$ as $u,v,p,w,c$, respectively;
here $p>0$ and $c>0$. Fix $t_0\ge14/5$ and put
$\lambda=\lambda_{\min}(F(t_0))\ge0$.
If $\lambda\ge c$, then $Q_\lambda'\succ0$. Indeed, the columns of
$M^{-1}$ at $0,2$ have full support, while the present pairs $12$ and
$03$ give strictness at coordinates $1$ and $3$, respectively.
Equation~\eqref{eq:dual-vector-strictness} therefore gives
$\lambda<d_i$ for every $i$. Thus the smallest eigenvalue strictly
increases in this case. Assume $\lambda<c$.
Use $\mathbf v_i=M^{1/2}e_i$, $\scrU=\Span(\mathbf v_0,\mathbf v_1,\mathbf v_2)$, and the
normals from \hyperref[prop:diamond-normal-data]{the calculation above}.
Write $\mathbf w_0=-\mathbf v_0$, $\mathbf w_1=\mathbf v_1$, $\mathbf w_2=\mathbf v_2$, orient $\mathbf n_i$ by
$\mathbf n_i^{\mathsf T}\mathbf w_i>0$, and put
\[
 \mathbf e=\frac{\mathbf w_0}{\sqrt t},\quad
 \mathbf h=\frac{\mathbf w_2+\mathbf w_0/t}{\sqrt{t-1/t}},\quad s=\frac1{t+1},\quad
 {g_0}=\sqrt{1-\frac2{t(t-1)}}.
\]
The inverse-metric formula for the normalized dual vectors gives
$\mathbf n_i^{\mathsf T}\mathbf w_j=\sqrt t\,{g_0}\one_{\{i=j\}}$. Therefore
\[
 \begin{gathered}
 \|\mathbf e\|=\|\mathbf h\|=1,\quad \mathbf e^{\mathsf T}\mathbf h=0,\quad
 (\mathbf n_i^{\mathsf T}\mathbf e)_{i=0}^2=({g_0},0,0),\\
 (\mathbf h^{\mathsf T}\mathbf n_i)_{i=0}^2
 =\frac{{g_0}}{\sqrt{1-t^{-2}}}(t^{-1},0,1),\qquad
 \mathbf n_i^{\mathsf T}\mathbf n_j=r:=\frac1{t-1}\quad(i\ne j).
 \end{gathered}
\]
Since $\mathbf v_3$ has inner products $(-1,0,1)$ with $(\mathbf w_0,\mathbf w_1,\mathbf w_2)$,
\[
 \Pi _\scrU\mathbf v_3=\frac{-\mathbf w_0+\mathbf w_2}{t+1},\quad
 \mathbf z:=\frac{\mathbf v_3+\mathbf w_0/t}{\sqrt{t-1/t}},\quad
 \Pi _\scrU\mathbf z=s\mathbf h,\quad \|\mathbf z\|=1,\quad \mathbf e^{\mathsf T}\mathbf z=0.
\]
Choosing a unit normal $\mathbf n_*$ to $\scrU$ with the appropriate
sign, we obtain
$\mathbf z=s\mathbf h+\sqrt{1-s^2}\mathbf n_*$ and $\Pi _{03}=\mathbf e\mathbf e^{\mathsf T}+\mathbf z\mathbf z^{\mathsf T}$.

Put
 \({b_0}=c-\lambda+u+v+w>u+v+w,\quad B_0={b_0} I_\scrU+p \mathbf e\mathbf e^{\mathsf T},\quad E=u \mathbf n_2\mathbf n_2^{\mathsf T}+v \mathbf n_1\mathbf n_1^{\mathsf T} +w \mathbf n_0\mathbf n_0^{\mathsf T}.\)
Before adding $p\mathbf z\mathbf z^{\mathsf T}$, the restriction to
$\scrU$ is $B_0-E$. It is positive definite, since
$E\preceq(u+v+w)I_\scrU$.
Set ${r_1}=\mathbf h^{\mathsf T}(B_0-E)^{-1}\mathbf h$.
The rank-one inverse formula and the Schur complement along $\scrU$
show that $F(t)-\lambda I\succeq0$ if and only if
\[
 p(1-s^2)-\lambda
       -\frac{p^2s^2(1-s^2){r_1}}{1+ps^2{r_1}}\ge0.
\]
Multiplying by $1+ps^2r_1>0$ reduces this condition to
$p(1-s^2)-\lambda-p\lambda s^2r_1\ge0$.
We show that $s^2r_1$ is nonincreasing, so this inequality persists
as $t$ increases.

The inverse of $B_0$ and the associated Gram entries are
$B_0^{-1}={b_0}^{-1}I_\scrU-\frac p{{b_0}({b_0}+p)}\mathbf e\mathbf e^{\mathsf T}, \qquad G_{ij}=\mathbf n_i^{\mathsf T}B_0^{-1}\mathbf n_j$,
\[
 G_{00}=\frac1{b_0}-\frac{p{g_0}^2}{{b_0}({b_0}+p)}\ge\frac1{{b_0}+p}>0,
 \quad G_{11}=G_{22}=\frac1{b_0},
 \quad G_{ij}=\frac r{{b_0}}\ (i\ne j).
\]
These entries are nonnegative and nonincreasing in $t$, because
$r'=-1/(t-1)^2<0$ and
$({g_0}^2)'=2(2t-1)/(t^2(t-1)^2)>0$.
Put $a_i=s \mathbf h^{\mathsf T}B_0^{-1}\mathbf n_i$.
Then $a_i\ge0$, $a_1=0$, and
$a_0^2=\frac{t-2}{{b_0}^2t(t^2-1)^2},\qquad a_2^2=\frac{t(t-2)}{{b_0}^2(t^2-1)^2}$.
For $t\ge14/5$,
$\left(\frac{t(t-2)}{(t^2-1)^2}\right)' =-\frac{2(t^3-3t^2+t-1)}{(t^2-1)^3}<0$,
because the numerator polynomial has value $29/125$ at $14/5$
and derivative $3(t-1)^2-2>0$ thereafter.
Hence $a_2$ is decreasing, and so is $a_0=a_2/t$.
With $(b_0^{(1)},b_1^{(1)},b_2^{(1)})=(w,v,u)$, put
$V=[\sqrt{b_i^{(1)}}\mathbf n_i]_i$, $\Gamma=V^{\mathsf T}B_0^{-1}V$ and
$\mathbf c=sV^{\mathsf T}B_0^{-1}\mathbf h$. Then $E=VV^{\mathsf T}$ and
\[
 \Gamma_{ij}=\sqrt{b_i^{(1)}b_j^{(1)}}G_{ij},\qquad
 c_i=\sqrt{b_i^{(1)}}a_i,\qquad
 \rho(\Gamma)\le\|B_0^{-1/2}EB_0^{-1/2}\|
       \le\frac{u+v+w}{{b_0}}<1.
\]
The entries of $\Gamma$ and $\mathbf c$ are nonnegative and
nonincreasing. The inverse identity, verified by multiplication with
$B_0-E$, gives
\[
 (B_0-E)^{-1}=B_0^{-1}
       +B_0^{-1}V(I-\Gamma)^{-1}V^{\mathsf T}B_0^{-1},\qquad
 s^2{r_1}=\frac{s^2}{{b_0}}+\mathbf c^{\mathsf T}(I-\Gamma)^{-1}\mathbf c.
\]
By \eqref{eq:resolvent-order}, the entries of $(I-\Gamma)^{-1}$ are
nonnegative and nonincreasing. Therefore $s^2{r_1}$ is nonincreasing,
and the scalar Schur-complement expression satisfies
\[
 \frac d{dt}\{p(1-s^2)-\lambda-p\lambda s^2{r_1}\}
 =-p(s^2)'-p\lambda(s^2{r_1})'>0.
\]
At $t_0$ this expression is zero because $\lambda$ is the smallest
eigenvalue. It is strictly positive for $t>t_0$, so the
Schur-complement equivalence gives $F(t)\succ\lambda I$.
This argument also applies when $\lambda$ is repeated. Signed
reflection gives the same conclusion for the reflected support.

The third partial sum is therefore nonincreasing on every remaining
interval between consecutive positive eigenvalues. Together with the
first two comparisons, the summation identity yields $\tr(T_+)\le N$.

\medskip\noindent\emph{Completion of the mixed-diamond case.}\
\phantomsection\label{thm:diamond-full-nine-integer-complete}
We now combine the preceding estimates to prove $\tr(T_+)\le N$
for the nine mixed-diamond link types and every integer $a\ge2$.

Every proper principal submatrix has a forest or signed triangle
as its correlation graph, so Lemma~\ref{thm:stieltjes-residual}
handles fewer than four used coordinates. For links with a common
coordinate, the
\hyperref[prop:common-head-paw-diamond-damping]{common-coordinate argument}
applies. At most two link types with empty intersection are disjoint,
and Lemma~\ref{prop:incidence-component-reduction} reduces them to
single links, whose projection sums have trace $b_H|H|$.
Thus assume that all four coordinates are used, at least three link
types occur, and their common intersection is empty.

Write the positive eigenvalues of $T$ as $\mu_1\ge\cdots\ge\mu_\ell>0$,
$\ell\le4$.
By Lemma~\ref{prop:relative-projection-flow} and the summation
identity, it suffices to prove that $s_j(F_{\mathbf b}(a+x))$ is
nonincreasing on $[\mu_{j+1},\mu_j]$ for $1\le j<\ell$.
The case $j=1$ follows from \eqref{eq:affine-gram-damping}.
If $a\ge4$, then $t=a+x>22/7$ for every $x>0$.
For $a=3$, put $h=1/7$. Equation~\eqref{eq:cover-floor} gives
\[
 F_{\mathbf b}(a+h)=F_{\mathbf b}(22/7)\succeq\frac29I\succ hI,
 \qquad 4=n_+(T-hI)\le n_+(T)\le4.
\]
Thus all four positive eigenvalues exceed $h$.
For every integer $a\ge3$, the required intervals therefore lie in
$t>22/7$. The
\hyperref[thm:diamond-full-nine-second-flow]{exterior comparison} and
\hyperref[prop:conditioned-first-flow]{inverse-metric comparison}
give the assertions for $j=2,3$. The summation identity proves
$\tr(T_+)\le N$ for $a\ge3$.

Suppose $a=2$. If both triangle weights are positive, use their
integer minimum in the
\hyperref[prop:diamond-triangle-pair-core-shift]{triangle-pair identity}
to preserve $\tr(T_+)-N$ and increase $a$ to at least three.
The preceding case applies, including its initial reductions.
We may therefore assume that at most one triangle occurs.

If a path triple $H=013$ or $123$ occurs, then
$F_{\mathbf b}(t)\succeq b_H\Pi _H(t)$ implies $\lambda_3(F_{\mathbf b}(t))\ge b_H$.
The inertia identity gives at least three eigenvalues above every
$0<x<b_H$, so
$\mu_3\ge b_H\ge1,\qquad 2+\mu_3\ge3$.
The \hyperref[thm:diamond-full-nine-second-flow]{exterior comparison}
therefore gives the assertion for $j=2$. If there are four positive
eigenvalues, the hypotheses of the
\hyperref[prop:diamond-normalized-third-flow]{smallest-eigenvalue comparison}
are satisfied, giving the assertion for $j=3$. Thus $\tr(T_+)\le N$
whenever a path triple occurs.

If no path triple occurs, the hypotheses of the
\hyperref[prop:seven-type-budget]{preceding seven-type case} hold,
and again $\tr(T_+)\le N$.

Finally, consider an arbitrary mixed-diamond family. Split its
orthogonal pair into singletons and delete all singleton links.
All remaining links are among the five nonorthogonal pairs and four
triples considered above. Their block matrix therefore satisfies the
positive-trace bound. Restoring the deleted singletons one at a time
by \eqref{eq:rank-singleton} proves the bound for the original family.
The relative estimates and preservation of strictness were established
in the common-coordinate argument. Combining these conclusions with
the positive-paw, Stieltjes, and positive-clique cases completes the
proof.
\end{proof}

\section{Fourth-Moment Estimates}\label{sec:large-residuals}

\subsection{The Fourth-Moment Identity}
\label{sec:shifted-fourth-identity}

Throughout this section, put
$b_{ij}=\sum_{H\supseteq\{i,j\}}b_H$ for $i\ne j$.

\begin{lemma}\label{prop:shifted-fourth-identity}
Let $C\succeq0$ have constant diagonal $a\ge0$, and let the links
have size at most three. Set
\[
 e_H=\sum_{\{i,j\}\subseteq H}C_{ij}^2,\qquad
 e_2=\sum_He_H^2,\qquad
 r_*=\sum_{i<j}C_{ij}^2(d_i-b_{ij})(d_j-b_{ij}).
\]
The sum defining $e_2$ has one term for each link block, without
the factor $b_H$. Then
$\tr\bigl((T+aI)^4\bigr)=\sum_i(a+d_i)^4+2e_2+4r_*$.
\end{lemma}
\begin{proof}
For matrices of compatible sizes,
\[
 \begin{pmatrix}P&U\\V&Q\end{pmatrix}^{\!2}
 =\begin{pmatrix}P^2+UV&PU+UQ\\VP+QV&VU+Q^2\end{pmatrix}.
\]
Squaring once more and using $\tr(AB)=\tr(BA)$ gives
\begin{align}
 \tr\begin{pmatrix}P&U\\V&Q\end{pmatrix}^{\!4}
 &=\tr(P^2+UV)^2+2\tr((PU+UQ)(VP+QV))
 +\tr(VU+Q^2)^2\notag\\
 &=\tr P^4+\tr Q^4+4\tr(P^2UV)+4\tr(PUQV)
 +4\tr(UQ^2V)+2\tr((UV)^2).
 \label{eq:block-fourth-trace}
\end{align}
Suppose first that $C\succ0$. Put $A=C-aI$,
$S_0=\bigoplus_HA[H]$, $P_0=aI+D$, and
$V=U^{\mathsf T}A-S_0U^{\mathsf T}$. The invertible matrix $X$ below
satisfies
\[
 X=\begin{pmatrix}C^{1/2}&0\\U^{\mathsf T}&I\end{pmatrix},\qquad
 (T+aI)X=X\begin{pmatrix}P_0&U\\V&-S_0\end{pmatrix}.
\]
For $j=0,1,2$,
\[
 (US_0^jV)_{ii}
 =\sum_{H\ni i}b_H\sum_{v\in H}(A[H]^j)_{iv}
 (A_{vi}-A[H]_{vi})=0.
\]
Since $P_0$ is diagonal, all three mixed trace terms in
\eqref{eq:block-fourth-trace} vanish. Moreover,
 \((UV)_{ii}=0,\qquad (UV)_{ij}=A_{ij}(d_i-b_{ij})\quad(i\ne j),\)
so
\[
 \tr\bigl((T+aI)^4\bigr)
 =\sum_i(a+d_i)^4+\sum_H\tr A[H]^4
 +4\sum_{i<j}A_{ij}^2(d_i-b_{ij})(d_j-b_{ij}).
\]
The cubic identity in Lemma~\ref{lem:local-adjugate}, with
$\tr A[H]=0$ and $\tr A[H]^2=2e_H$, gives
$\tr A[H]^4=e_H\tr A[H]^2+2t_H\tr A[H]=2e_H^2$.
For singular $C$, apply the identity to $C+s_0I$, with diagonal
$a+s_0$, where $s_0>0$. The matrices $A,S_0,U$ are unchanged.
Continuity of the positive semidefinite square root and the trace
then gives the identity at $s_0=0$.
\end{proof}

\begin{lemma}\label{prop:moment-certificate}
Let $k\ge4$, and suppose that a compression of $Y=T+aI$ is at least
\[
 M=\begin{pmatrix}P&W\\W^{\mathsf T}&- b_0 I\end{pmatrix},\qquad
 P=\diag(p_1,\ldots,p_r),\qquad 1\le r\le3,
\]
where
 \(p_1\ge\cdots\ge p_r\ge p\ge h_0 \ge c\ge b_0 \ge0, \quad h_0 >0,\quad (c- b_0 )p^2+c b_0 p- h_0^3\ge0.\)
Choose $0\le h_1\le h_0$ with $h_1=h_0$ or
$a+\lambda_4(T)\le h_1$, and put
$e_0=4a+N-\tr P-(3-r) h_0 - h_1$.
If $ s_4^+(T)\ge N$, then
\begin{equation}\label{eq:master-certificate}
 \tr Y^4\ge\tr P^4+(3-r) h_0^4+ h_1^4+4 h_0^3e_0
 +4\tr((P^2-cP)WW^{\mathsf T})+(k-4)a^4.
\end{equation}
The inequality is strict if $ s_4^+(T)>N$ or $W\ne0$.
For $r=0$, the same assertion holds after omitting $P,W,p,c, b_0$
and their conditions. The choice $c=p= h_0 \ge b_0$ always satisfies
the scalar condition.
\end{lemma}
\begin{proof}
Let $n_0$ be the number of columns of $W$, and put
$\Gamma=WW^{\mathsf T}$ and $f_0 (z)=z^4-4 h_0^3z$.
Since $P\succeq pI\succ0$, Schur complementation gives
 \(M\text{ is congruent to }P\oplus(- b_0 I-W^{\mathsf T}P^{-1}W).\)
Thus $M$ has $r$ positive eigenvalues
$\rho_1\ge\cdots\ge\rho_r$. Write its other eigenvalues as
$-\mu_1,\ldots,-\mu_{n_0}$, where
$\mu_1\ge\cdots\ge \mu_{n_0}$. By interlacing,
$\rho_i\ge p_i\ge h_0$ and $\mu_j\ge b_0$.
Let $\sigma_1\ge\cdots\ge\sigma_{n_0}\ge0$ be the eigenvalues of
$W^{\mathsf T}W$, including zeros, and
put
\[
 t_j=\frac{\sqrt{(p+ b_0 )^2+4\sigma_j}-(p+ b_0 )}2,
 \qquad \sigma_j=t_j(t_j+p+ b_0 ).
\]
Eigenvalue monotonicity and a singular-value decomposition of $W$ give
\[
 \mu_j\le\lambda_j\!\left(-\begin{pmatrix}pI&W\\W^{\mathsf T}&- b_0 I\end{pmatrix}\right)
 = b_0 +t_j.
\]
Indeed, a nonzero singular value $\sqrt\sigma$ gives a two-dimensional
block whose characteristic equation is $(z-p)(z+ b_0 )-\sigma=0$.
Its negative root is $-b_0-t$, where $t(t+p+b_0)=\sigma$.
The unused lower and upper coordinates contribute $- b_0$ and $p$,
respectively.
For $t\ge0$ and $\sigma=t(t+p+ b_0 )$, expansion gives
\begin{align*}
 &2\sigma^2+4\{(c- b_0 )p+ b_0^2\}\sigma
 -\{( b_0 +t)^4- b_0^4+4 h_0^3t\}\\
 &\quad=t^4+4pt^3+(2p^2+4cp)t^2
 +4\{(c- b_0 )p^2+c b_0 p- h_0^3\}t\ge0.
\end{align*}
The function $z^4- b_0^4+4 h_0^3(z- b_0 )$ is increasing for
$z\ge b_0$. Consequently
\[
 \sum_j\{\mu_j^4- b_0^4+4 h_0^3(\mu_j- b_0 )\}
 \le2\tr\Gamma^2+4\{(c- b_0 )p+ b_0^2\}\tr\Gamma.
\]
If $W\ne0$, some $t_j>0$, and the term
$(2p^2+4cp)t_j^2$ makes this inequality strict. Also,
\eqref{eq:block-fourth-trace} and
$\tr M=\sum_i\rho_i-\sum_j\mu_j$ give
\begin{gather*}
 \tr M^4=\tr P^4+n_0 b_0^4
 +4\tr((P^2- b_0 P+ b_0^2I)\Gamma)+2\tr\Gamma^2,\\
 \sum_i(\rho_i-p_i)=\sum_j(\mu_j- b_0 ).
\end{gather*}
Subtracting the preceding bound therefore yields
\begin{align*}
 \sum_i\{ f_0 (\rho_i)- f_0 (p_i)\}
 &\ge4\tr((P^2- b_0 P-(c- b_0 )pI)\Gamma)\\
 &=4\tr((P^2-cP)\Gamma)
 +4(c- b_0 )\tr((P-pI)\Gamma)\\
 &\ge4\tr((P^2-cP)\Gamma).
\end{align*}
Here $c\ge b_0$ and
$\tr((P-pI)\Gamma)=\tr(\Gamma^{1/2}(P-pI)\Gamma^{1/2})\ge0$.
If $n_0=0$, the same calculation holds with empty sums and $\Gamma=0$.

The zero upper block of $T$ has order $k$, so interlacing gives
$\lambda_k(T)\ge0$. Put
$y_i=a+\lambda_i(T)$ for $1\le i\le4$. By the compression comparison
and the hypothesis,
\[
 y_i\ge\rho_i\ge h_0 \ (i\le r),\qquad y_i\ge0,\qquad
 \sum_{i=1}^4y_i\ge4a+N.
\]
Since $f_0'(z)=4(z^3-h_0^3)$, the function $f_0$ decreases on
$[0,h_0]$ and increases on $[h_0,\infty)$.
Hence $f_0(y_i)\ge f_0(\rho_i)$ for $i\le r$, and
$f_0(y_i)\ge f_0(h_0)$ for the remaining indices.
We also have $f_0(y_4)\ge f_0(h_1)$: if $h_1<h_0$, then
$0\le y_4\le h_1$, while $h_1=h_0$ gives the minimum of $f_0$.
Therefore
\begin{align*}
 \sum_{i=1}^4y_i^4
 &=4 h_0^3\sum_{i=1}^4y_i+\sum_{i=1}^4 f_0 (y_i)\\
 &\ge4 h_0^3(4a+N)+\sum_{i=1}^r f_0 (\rho_i)
 +(3-r) f_0 ( h_0 )+ f_0 ( h_1 )\\
 &=\tr P^4+(3-r) h_0^4+ h_1^4+4 h_0^3e_0
 +\sum_i\{ f_0 (\rho_i)- f_0 (p_i)\}.
\end{align*}
The next $k-4$ eigenvalues of $T$ are nonnegative, so their contribution
to $\tr Y^4$ is at least $(k-4)a^4$. All other fourth powers are
nonnegative. This proves \eqref{eq:master-certificate}; the inequality
is strict if $\sum_{i=1}^4y_i>4a+N$ or $W\ne0$. When $r=0$, the
last comparison applies directly, with empty sums. Finally,
$( h_0 - b_0 ) h_0^2+c b_0 h_0 - h_0^3=0$ when $c= h_0$.
\end{proof}

We also use \cite[Corollary~II.3.4, p.~42]{Bhatia}: if
$\mathbf x=(x_1,\ldots,x_n)$ and $\mathbf y=(y_1,\ldots,y_n)$
are decreasing nonnegative vectors, padded with zeros to the same
length if necessary, and
$\sum_{i=1}^j x_i\ge\sum_{i=1}^j y_i$ for $1\le j\le n$, then
every increasing convex real function $f$ on $[0,\infty)$ satisfies
\begin{equation}\label{eq:prefix-convexity}
 \sum_{i=1}^n f(x_i)\ge\sum_{i=1}^n f(y_i).
\end{equation}

\subsection{Matrices on at Least Five Coordinates}

\begin{lemma}\label{prop:many-coordinate}
Let $T$ arise from a complete simplex and links of size at most three.
Assume that $a\ge2$ is an integer, the link weights are positive
integers, and the link-incidence graph is connected on $k\ge5$
used coordinates. Then $s_4^+(T)<N$.
\end{lemma}
\begin{proof}

\medskip\noindent\emph{Reduction by the degree sequence.}\
Relabel the coordinates so that $d_1\ge\cdots\ge d_k>0$.
If $2d_2\le\sum_{j\ge3}d_j$, take
\[
 (I,h)=\begin{cases}
 (\{1\},(N-d_1)/3),&4d_1\ge N,\\
 (\varnothing,N/4),&4d_1<N.
 \end{cases}
\]
In the first case, $d_1\ge h\ge d_2$; in the second,
$d_i\le d_1<h$ for every $i$. The orthogonal-coordinate estimate
below, with $c_0=h$, will treat both cases.
We first prove the fourth-eigenvalue bound and the compression
estimates needed for the remaining case $2d_2>\sum_{j\ge3}d_j$.

\medskip\noindent\emph{A bound for the fourth eigenvalue.}\
\label{prop:small-tail-cap}
For any three coordinates $I$, put $t=\sum_{j\notin I}d_j$.
We prove $\lambda_4(T)\le t-\tfrac12$ and, when $t=2$,
$\lambda_4(T)<13/9$.
Here we use only $C\succeq0$, its constant integer diagonal $a\ge2$,
and $|C_{ij}|\le1$ for $i\ne j$.

Fix distinct $i,j\notin I$ and let $S$ be the Schur complement of
$I$ in $C[I\cup\{i,j\}]+xI$, where $x>0$. We first show
\begin{equation}\label{eq:five-coordinate-conditional}
 \frac{|S_{ij}|}{\sqrt{S_{ii}S_{jj}}}
 \le \eta(x):=\frac{x+6}{x^2+6x+6}.
\end{equation}
Write $A=\left(\begin{smallmatrix}a&c\\c&a\end{smallmatrix}\right)$
for the principal block on $\{i,j\}$. After changing one coordinate
sign, we may assume $0\le c\le1$. Put $B=C[I]$ and
$E=C[I,\{i,j\}]$. Then
$S=xI+A-R$ and $R=E^{\mathsf T}(B+xI)^{-1}E\succeq0$.
For $a=2$, $B\preceq4I$ and the Schur-complement inequality give
$R\preceq4A/(x+4)$. For $C\succ0$, this follows from
$(B+xI)^{-1}\preceq4B^{-1}/(x+4)$.
For singular $C$, apply this bound to
$(1-\varepsilon)C+2\varepsilon I$, where $0<\varepsilon<1$,
and use continuity at $\varepsilon=0$.
For $a\ge3$, use $B\succeq EA^{-1}E^{\mathsf T}$ and
$E^{\mathsf T}E\preceq\tr(E^{\mathsf T}E)I\preceq6I$.
Here $\tr(E^{\mathsf T}E)=\sum_{v\in I}(C_{vi}^2+C_{vj}^2)\le6$.
With $X=A^{-1/2}E^{\mathsf T}EA^{-1/2}\preceq6A^{-1}$,
\[
 R\preceq A^{1/2}X(X+xI)^{-1}A^{1/2}
 \preceq6A(6I+xA)^{-1}.
\]
The first inequality uses
$(EA^{-1}E^{\mathsf T}+xI)E=E(A^{-1}E^{\mathsf T}E+xI)$.
For the second bound, use
$X(X+xI)^{-1}=I-x(X+xI)^{-1}$.

We next deduce a bound for $|S_{ij}|/\sqrt{S_{ii}S_{jj}}$.
Suppose $0\preceq R\preceq \widehat A=\left(\begin{smallmatrix}\alpha&\beta\\\beta&\alpha\end{smallmatrix}\right)$,
where $\alpha\ge|\beta|$, and $0<\rho<1$.
For $\mathbf z=(z_1,z_2)^{\mathsf T}$ with $\|\mathbf z\|^2=2$,
put $u=z_1z_2\in[-1,1]$ and
$Q=\mathbf z\mathbf z^{\mathsf T}-(1-\rho)\diag(z_1^2,z_2^2)$.
The trace of the positive part of $\widehat A^{1/2}Q\widehat A^{1/2}$ is
$\rho\alpha+\beta u+ \sqrt{(\rho\alpha+\beta u)^2+(1-\rho^2)(\alpha^2-\beta^2)u^2}$.
This expression is convex in $u$, since its square-root term is
the Euclidean norm of an affine function of $u$. The variational
formula gives
\[
 \tr((xI+A-R)Q)
 \ge2\rho(x+a)+2cu-\tr((\widehat A^{1/2}Q\widehat A^{1/2})_+)
\]
has a concave right-hand side. It is nonnegative on $[-1,1]$ if
it is nonnegative at both endpoints, namely if
\begin{equation}\label{eq:two-coordinate-cone}
 2\rho(x+a)\pm2c\ge(1+\rho)(\alpha\pm\beta).
\end{equation}
Nonnegativity of $\tr(SQ)$ for every $\mathbf z$ is equivalent to
$S-(1-\rho)\diag(S_{ii},S_{jj})\succeq0$, or
$|S_{ij}|\le\rho\sqrt{S_{ii}S_{jj}}$.

For $a=2$, take $\widehat A=4A/(x+4)$; for $a\ge3$, take
$\widehat A=6A(6I+xA)^{-1}$. In each case the smaller endpoint in
\eqref{eq:two-coordinate-cone} has the minus sign and is minimized
at $c=1$: for $f(y)=4y/(x+4)$ or $6y/(6+xy)$,
we have $0<f'(y)\le1$, so the derivative of the minus endpoint
is $-2+(1+\rho)f'(a-c)<0$, whereas the plus endpoint is increasing.
Put $\rho=\eta(x)$. For $a=2$, equality holds in the smaller endpoint
condition. For $a\ge3$, the difference between its two sides is
\[
 \begin{aligned}
 &2\rho(x+a)-2-(1+\rho)\frac{6(a-1)}{6+x(a-1)}\\
 &\quad=(a-3)\left\{2\rho-
       \frac{36(1+\rho)}{(6+x(a-1))(6+2x)}\right\}\\
 &\quad\ge(a-3)\frac{x(2x+9)}{(x^2+6x+6)(x+3)}\ge0.
 \end{aligned}
\]
This proves \eqref{eq:five-coordinate-conditional}.

Let $\Pi$ project onto $((C+xI)^{1/2}\mathbb R^I)^\perp$.
The normalized projected coordinate vectors have pairwise absolute
inner products at most $\eta(x)$. Each compressed link projection
is bounded by the projection onto its projected coordinate span.
If some link has $r_H=|H\setminus I|\ge2$, then
\[
 \Pi F_C(x)\Pi\preceq
 \sum_{H:r_H>0}b_H I
 =\left(t-\sum_Hb_H(r_H-1)_+\right)I\preceq(t-1)I.
\]
If no such link occurs, every nonzero projected span is a line.
Write the degrees of the corresponding coordinates as
$z_1,\ldots,z_n\ge1$, where $n=k-3\ge2$ and $\sum z_i=t$.
The largest eigenvalue of their weighted Gram matrix is at most
\[
 (1-\eta(x))\max_i z_i+\eta(x)\sum_i z_i
 \le t-n+1+(n-1)\eta(x).
\]
Here we used Cauchy--Schwarz and $\max_i z_i\le t-n+1$.
Since $\eta$ decreases and $\eta(3/2)=10/23<1/2$,
both upper bounds are at most $x$ at $x=t-1/2\ge3/2$.
Interlacing and the inertia identity
give $\lambda_4(T)\le t-1/2$.
If $t=2$, then $n=2$ and $z_1=z_2=1$. At $x=13/9$,
\[
 x-1-\eta(x)=\frac{x^3+5x^2-x-12}{x^2+6x+6}
 =\frac1{12213}>0.
\]
Both upper bounds are strictly less than $x$, so interlacing and
the inertia and nullity identities give
$\lambda_4(T)<13/9$.

If $C$ arises from a complete simplex and $H$ is a short link, then
\begin{equation}\label{eq:short-link-norm}
 \norm{C[H]-aI}\le |H|-1\le2.
\end{equation}
Indeed, for every $\mathbf z\in\R^H$,
\[
 \bigl|\mathbf z^{\mathsf T}(C[H]-aI)\mathbf z\bigr|
 \le2\sum_{i<j}|z_i z_j|
 \le(|H|-1)\norm{\mathbf z}^2.
\]
\medskip\noindent\emph{Compression onto orthogonal coordinates.}
\label{sec:orthogonal-heads}\label{prop:orthogonal-head-gain}
Let $I=\{1,\ldots,r\}$, where $0\le r\le3$, and suppose
$C_{ij}=0$ for distinct $i,j\in I$. Put $J=I^c$, and suppose
$d_i\ge h\ge1$ for $i\in I$. Define
\[
 t=\sum_{j\in J}d_j=(3-r)h+c_0,\quad0\le c_0\le h,\quad
 v_i=\sum_{j\in J}C_{ij}^2(d_j-b_{ij}),\quad
 g_*=4\sum_{i\in I}(d_i-h)v_i,
\]
$y_i=h$ on $I$ and $y_i=d_i$ on $J$. Assume $c_0=h$ or
$\lambda_4(T)\le c_0$; set $h_0 =a+h$, $h_1 =a+c_0$.
For $r\ge1$, order the selected coordinates by degree and define
$\mathbf z_i=Ze_i$, $p_i=a+d_i$, $\mathbf r_i=(T-d_iI)\mathbf z_i$. Block multiplication gives
\[
 \begin{gathered}
 \langle \mathbf z_i,\mathbf z_j\rangle=\one_{\{i=j\}}p_i,\qquad(\mathbf r_i)_{\rm upper}=0,\\
 (\mathbf r_i)_H=\begin{cases}0,&i\in H,\\
 \sqrt{b_H}R_H^{\mathsf T}Ce_i,&i\notin H,
 \end{cases}\qquad
 \|\mathbf r_i\|^2=v_i,\qquad\langle \mathbf z_i,\mathbf r_i\rangle=0,\\
 \langle \mathbf z_j,\mathbf r_i\rangle=C_{ij}(d_j-b_{ij})=0\quad(i\ne j).
 \end{gathered}
\]
The vectors $\mathbf z_i/\sqrt{p_i}$ are therefore orthonormal,
and their compression is $P=\diag(p_1,\ldots,p_r)$.
Adjoin an orthonormal basis of $\Span\{\mathbf r_i:i\in I\}$,
which lies in the lower coordinate space and is orthogonal to them.
If this space has dimension $n_0\le r$, then
\eqref{eq:short-link-norm} bounds the resulting compression below by
\[
 M=\begin{pmatrix}P&W\\W^{\mathsf T}&-2I_{n_0}\end{pmatrix},\quad
 P=\diag(p_1,\ldots,p_r),\quad (WW^{\mathsf T})_{ii}=v_i/p_i.
\]
In \eqref{eq:master-certificate}, choose $c=p= h_0$. Then
\[
 e_0=t-(3-r)h-c_0=0,\qquad
 4\tr((P^2- h_0 P)WW^{\mathsf T})
 =4\sum_{i\in I}(p_i- h_0 )v_i=g_*.
\]
For $r=0$, use the empty compression. Hence $ s_4^+(T)\ge N$ implies
\begin{equation}\label{eq:orthogonal-certificate}
 \tr\bigl((T+aI)^4\bigr)\ge\sum_{i\in I}(a+d_i)^4
 +(3-r)(a+h)^4+(a+c_0)^4+(k-4)a^4+g_*.
\end{equation}

Let $r_J$ be the part of $r_*$ indexed by pairs contained in $J$, and put
$r_0=\sum_{i<j,\ i,j\in J}C_{ij}^2d_id_j$. Since $d_i,d_j\ge b_{ij}$
and $b_H\ge1$,
\begin{align*}
 r_0-r_J
 &=\sum_{i<j,\ i,j\in J}C_{ij}^2b_{ij}(d_i+d_j-b_{ij})\\
 &=\sum_Hb_H\sum_{\{i,j\}\subseteq H\cap J}
 C_{ij}^2(d_i+d_j-b_{ij})\\
 &\ge\sum_H\sum_{\{i,j\}\subseteq H\cap J}
 C_{ij}^2\max\{d_i,d_j\}.
\end{align*}
Pairs contained in $I$ contribute zero. For pairs meeting both
$I$ and $J$,
\begin{align*}
 r_*-r_J-\sum_{i\in I}(d_i-h)v_i
 &=\sum_{i\in I,j\in J}C_{ij}^2(h-b_{ij})(d_j-b_{ij})\\
 &=\sum_{i\in I,j\in J}C_{ij}^2
 \{hd_j-b_{ij}(h+d_j-b_{ij})\}\\
 &\le\sum_{i\in I,j\in J}hC_{ij}^2d_j
 -h\sum_H\sum_{\substack{i\in H\cap I\\j\in H\cap J}}C_{ij}^2.
\end{align*}
Combining these inequalities gives
\begin{equation}\label{eq:clipped-energy}
 2e_2+4r_*-g_*\le4\sum_{i<j}C_{ij}^2y_iy_j+\sum_Hl_H,
\end{equation}
where
\[
 l_H=2e_H^2
 -4\sum_{\{i,j\}\subseteq H\cap J}C_{ij}^2\max\{d_i,d_j\}
 -4h\sum_{\substack{i\in H\cap I\\j\in H\cap J}}C_{ij}^2.
\]
If $e_H\le2$, each nonzero $C_{ij}^2$ equals one and contributes
at least four to the subtracted terms. Hence
$l_H\le2e_H^2-4e_H\le0$. If $e_H=3$, then $|H\cap I|\le1$, and
\[
 \begin{array}{c|c|c}
 |H\cap I|&l_H\text{ upper bound}&|H\cap J|\\\hline
 0&18-4(1+1+1)=6&3\\
 1&18-4-8h=14-8h&2
 \end{array}
\]
For $h\ge5/4$, these bounds give
$\sum_Hl_H\le2\sum_H|H\cap J|\le2\sum_Hb_H|H\cap J|=2t$.
A link on three coordinates of degree one would be a component
of the link-incidence graph, contrary to connectedness and $k\ge5$.
Thus one degree in the first row is at least two, and the sum of
the three pairwise maxima is at least $2+2+1$. That row consequently
has the stronger bound
$18-4(2+2+1)=-2$.
The second row is nonpositive for $h\ge7/4$; hence
$l_H\le0$ for every $H$ in that range.

Suppose now $c_0=h$, $d_j\le h$ on $J$, and $ s_4^+(T)\ge N$.
Set $s=4-r$, so $t=sh$, and
$d_j^{(0)} =sh^j-\sum_{i\in J}d_i^j$. Subtracting the exact moment
from \eqref{eq:orthogonal-certificate} gives
\begin{equation}\label{eq:common-quartic-margin}
 d_4^{(0)} +4a d_3^{(0)} +6a^2 d_2^{(0)} -(2e_2+4r_*-g_*)\le0.
\end{equation}
Suppose first that $h\ge2$. The bound $l_H\le0$ just proved gives
\[
 2e_2+4r_*-g_*\le4\sum_{i<j}y_iy_j=32h^2-2\sum_i y_i^2,
 \qquad\sum_i y_i=4h,\quad1\le y_i\le h.
\]
The vector $(h,h,h,h-1,1,0,\ldots)$ majorizes the decreasing
rearrangement of $\mathbf y=(y_1,\ldots,y_k)$.
For the rearranged vector $\mathbf y$, the first three partial sums
are at most $jh$ for $j=1,2,3$, the fourth is at most
$4h-(k-4)\le4h-1$, and all later sums are at most $4h$.
For $f_a(z)=(a+z)^4-a^4$ and $g_a(z)=f_a(z)-2z^2$, the coefficients
of $g_a$ are nonnegative ($a\ge2$). By \eqref{eq:prefix-convexity},
$\sum_i g_a(y_i)\le3g_a(h)+g_a(h-1)+g_a(1)$.
Since $d_1^{(0)} =0$, the left side of \eqref{eq:common-quartic-margin}
is at least
\begin{align*}
 4f_a(h)-\sum_i g_a(y_i)-32h^2
 &\ge4f_a(h)-3g_a(h)-g_a(h-1)-g_a(1)-32h^2\\
 &=4(h-2)^3+18(h-2)^2+48(h-2)+10\\
 &\quad+12(a-2)(h-1)(h+a+2)>0,
\end{align*}
a contradiction.
Suppose next that $h<2$. Every degree on $J$ then equals one.
Put $j=k-4\ge1$ and
$n=|J|=s+j$. Then
\[
 sh=n,\qquad h=1+j/s<2,\qquad2\le s\le4,\quad1\le j<s,
\]
so $h\ge5/4$ and
$2e_2+4r_*-g_*\le4\left\{\binom n2+(4-s)nh\right\}+2n$.
Moreover, $d_2^{(0)} =n(h-1)>0$ and $d_3^{(0)} =n(h^2-1)>0$.
Replacing $a$ by $2$ decreases the left side of
\eqref{eq:common-quartic-margin} by
$4(a-2) d_3^{(0)} +6(a^2-4) d_2^{(0)} \ge0$.
For $f(z)=z^4+8z^3+24z^2$, the same left side is therefore at least
\begin{align*}
 &s f(1+j/s)-(s+j)f(1)
 -4\left\{\binom{s+j}{2}+(4-s)(s+j)(1+j/s)\right\}-2(s+j)\\
 &\quad=\frac{s+j}{s^3}
 \{j^3+11j^2s+2js^3+27js^2+2s^4-16s^3\}>0,
\end{align*}
because $j\ge1$ gives
\[
 \begin{aligned}
 j^3+11j^2s+2js^3+27js^2+2s^4-16s^3
 &\ge1+11s+s^2(2s^2-14s+27)\\
 &=1+11s+\tfrac52s^2+2s^2(s-7/2)^2>0.
 \end{aligned}
\]
This contradiction proves $ s_4^+(T)<N$ under $c_0=h$, $d_j\le h$ on $J$.

\medskip\noindent\emph{Compression onto two correlated coordinates.}
\label{sec:true-head-anchors}\label{prop:two-head-criterion}
Let $u=d_1\ge v=d_2\ge2$, $C_{12}=1$, $I=\{1,2\}$, $J=I^c$, and
\[
 t_1=\sum_{j\in J}d_j,\qquad
 p_0 =a+\frac{v-a+\sqrt{(a+v)^2-4(v-b_{12})}}2.
\]
Assume that the principal block of $T+aI$ on lower coordinates
$(H,j)$ with $j\in J$ is bounded below by $- b_0 I$, where
$b_0 \in\{0,2\}$ ($b_0 =2$ always works), and choose
 \(a\le c\le h_0 \le p\le p_0 ,\qquad (c- b_0 )p^2+c b_0 p- h_0^3\ge0,\qquad a\le h_1 \le h_0 ,\)
with $h_1 = h_0$ or $a+\lambda_4(T)\le h_1$. Put
$x=c-a$ and $t_0 =(a+1-x)_+/2$. For each $j\in J$, classify the
pair $(C_{1j},C_{2j})$ as follows: type $\mathrm Z$ has two zero entries;
type $\mathrm B$ has exactly one nonzero entry; type $\mathrm S$ has two equal
nonzero entries; and type $\mathrm O$ has two opposite nonzero entries.
Define the corresponding coefficients by
\begin{equation}\label{eq:uniform-row-charges}
 c_{\mathrm Z}=0,\quad
 c_{\mathrm B}=\frac{(3a-1)x+2a-1}{3(a-1)},\quad
 c_{\mathrm S}=2x+2+ t_0 ,\quad
 c_{\mathrm O}=\max\{11/4,(2+1/(a-1))x+1\}.
\end{equation}
Put
\[
 b_*=\sum_{i<j,\ i,j\in J}C_{ij}^2d_id_j
 +\sum_{j\in J}d_jc_j,\qquad
 h_4 = h_0^4+ h_1^4+4 h_0^3(2a+t_1- h_0 - h_1 ).
\]
We shall prove that $s_4^+(T)<N$ whenever
\[
 g_0:= h_4 +(k-4)a^4-\sum_{j\in J}(a+d_j)^4
 -4b_*-7t_1-4>0.
\]

Merge identical links, which preserves the positive eigenvalues
and $N$, and write $C_I=C[I]=aI_2+A_I$,
$A_I=\left(\begin{smallmatrix}0&1\\1&0\end{smallmatrix}\right)$.
For $V_I=C^{1/2}R_IC_I^{-1/2}$,
$V_I^{\mathsf T}V_I=C_I^{-1/2}C_IC_I^{-1/2}=I_2$.
Compress $T+aI$ to the columns of $V_I$ and the lower coordinates
$(H,i)$ with $i\in H\cap I$. The resulting matrix is
\[
 Y_I=Z_I Z_I^{\mathsf T}-B_0,\quad
 Z_I=\binom{C_I^{1/2}}{U_I^{\mathsf T}},\quad
 B_0=A_I\oplus A_0,\quad
 A_0=\bigoplus_H A_I[H\cap I].
\]
Here $U_I=[\sqrt{b_H}R_{H\cap I}]_H$, so that
$U_I U_I^{\mathsf T}=\diag(u,v)$. Write $p_i=\lambda_i(Y_I)$ for
$i=1,2$. Since $\|B_0\|\le1$, one has $Y_I\succeq-I$.
Moreover, $\rank Z_I=2$ and eigenvalue interlacing for a rank-two
perturbation gives $\lambda_3(Y_I)\le\lambda_1(-B_0)\le1$.
All eigenvalues other than $p_1,p_2$ therefore lie in $[-1,1]$.

Put $b=b_{12}$, and let $\Pi_i(y)$ project onto coordinate line $i$
in metric $C_I+yI$, $y>0$. Their normalized Gram matrix
$\left(\begin{smallmatrix}1&(a+y)^{-1}\\(a+y)^{-1}&1\end{smallmatrix}\right)$
gives $\Pi_1(y)+\Pi_2(y)\succeq(1-(a+y)^{-1})I_2$. Since $u\ge v\ge b$,
\[
 F_{C_I}(y)=bI_2+(u-b)\Pi_1(y)+(v-b)\Pi_2(y)
 \succeq\left(v-\frac{v-b}{a+y}\right)I_2.
\]
The scalar lower bound equals $y$ when $y=p_0-a$, and exceeds $y$
when $0<y<p_0-a$. The inertia identity therefore gives
$p_1\ge p_2\ge p_0$.
The quadratic $y^2+(a-v)y-v(a-1)-b$ has value
$1-a-b<0$ at $y=v-1$, so $p_0>a+v-1$.

After changing one coordinate sign, $C_I$ is a Stieltjes matrix.
Lemma~\ref{thm:stieltjes-residual} therefore gives
$e=u+v-(p_1-a)-(p_2-a)\ge0$.
To bound $p_1^4+p_2^4$, temporarily merge equal nonempty intersections
$H\cap I$. The omitted blocks are nonpositive before the shift by
$aI$, so $p_1,p_2$ are unchanged. If $s\le2$ singleton types and
$b_{\rm pair} \in\{0,1\}$ pair types remain, the lower dimension is
$s+2 b_{\rm pair}$. The other shifted eigenvalues lie in $[-1,1]$, so
the exact moment gives
\begin{align*}
 p_1^4+p_2^4
 &\ge(a+u)^4+(a+v)^4+2 b_{\rm pair} +4(u-b)(v-b)-(s+2 b_{\rm pair} )\\
 &\ge(a+u)^4+(a+v)^4+4(u-b)(v-b)-2.
\end{align*}

Return to the original coordinates of $Y_I$.
Choose $P=\diag(p_1,p_2)$ and $X=(\mathbf x_1\ \mathbf x_2)$ with
$X^{\mathsf T}X=I_2$ and $Y_IX=XP$, and put $\Pi_I=XX^{\mathsf T}$.
For $j\in H\cap J$, put
$\mathbf b=(C_{1j},C_{2j})^{\mathsf T}$, $J_0=H\cap I$,
and $\mathbf w=(\sqrt{b_H}\mathbf b^{\mathsf T}C_I^{-1/2},-\mathbf g^{\mathsf T})$,
where $\mathbf g$ places $\mathbf b_{J_0}$ in the selected lower block.
This is the row joining $(H,j)$ to the compression defining $Y_I$:
$V_I^{\mathsf T}C^{1/2}e_j=C_I^{-1/2}\mathbf b$, while the block
between the lower coordinates indexed by $H\cap I$ and $H\cap J$
is the corresponding submatrix of $-C[H]$.
If $J_0=I$, let $E_H$ select its two lower coordinates and put
$K_H=(\sqrt{b_H}C_I^{-1/2},-E_H)$. Then
\[
K_HZ_I=0,\qquad K_HB_0=A_IK_H,\qquad
K_HY_I=-A_IK_H,\qquad(p_iI+A_I)K_H\mathbf x_i=0.
\]
Since $p_i>1$, we have $K_H\mathbf x_i=0$, so
$\mathbf wX=0$ whenever $I\subseteq H$.
For $|J_0|\le1$, we prove the estimate
\begin{equation}\label{eq:row-charge}
 \mathbf w(Y_I^2-cY_I)\Pi_I\mathbf w^{\mathsf T}
 \ge b_H\sum_{i\in I\setminus H}d_i b_i^2-b_Hc_j.
\end{equation}
Set $b^{J_0}_i= b_i\mathbf1_{i\in J_0}$. Then
\begin{gather*}
 U_I\mathbf g=\sqrt{b_H}\mathbf b^{J_0},\quad A_0\mathbf g=0,\\
 \mathbf w Y_I=\sqrt{b_H}\bigl(
 a \mathbf b^{\mathsf T}C_I^{-1/2}-(\mathbf b^{J_0})^{\mathsf T}C_I^{1/2},
 (\mathbf b-\mathbf b^{J_0})^{\mathsf T}U_I\bigr).
\end{gather*}
For $r=\mathbf b^{\mathsf T}C_I^{-1}\mathbf b$ and $s=\|\mathbf b_{J_0}\|^2\le1$,
$\|\mathbf w\|^2=b_Hr+s$ and
$(\mathbf b^{J_0})^{\mathsf T}C_I\mathbf b^{J_0}=as$ yield
\begin{align*}
 \mathbf w Y_I^2\mathbf w^{\mathsf T}
 &=b_H\sum_{i\in I\setminus J_0}d_i b_i^2+b_H(a^2r-as),\\
 \mathbf w Y_I\mathbf w^{\mathsf T}&=b_H(ar-s),\\
 \mathbf w(Y_I^2-cY_I)\mathbf w^{\mathsf T}
 &=b_H\sum_{i\in I\setminus J_0}d_i b_i^2-b_Hx(ar-s).
\end{align*}

Put $f(z)=z(z-c)$ and, for $j_0 \in\{1,2\}$, put
$g_{j_0}(z)=(2- j_0 )(z-c- j_0 )+(c+1)/(z+ j_0 )$.
Both functions are nonnegative for $z\ge c$. For $-1<z\le1$,
\[
 \begin{array}{c|c}
 j_0 &f(z)-g_{j_0}(z)\\\hline
 1&z^2(z-c)/(z+1)\\[1mm]
 2&(z+1)\{z^2+(1-c)z-c-1\}/(z+2)
 \end{array}
 \quad\le0.
\]
In the second row the quadratic has endpoint values $-1$ and
$1-2c$, so it is nonpositive throughout $[-1,1]$.
For $j_0 =1$ we use only type $\mathrm O$. Then $A_I\mathbf b=-\mathbf b$
and $A_0\mathbf g=0$, so
$\mathbf w^{\mathsf T}\in\im(I-B_0)
 \subseteq\im(Y_I+I)$.
For $A\succeq0$, define $A^\dagger=(A|_{\im A})^{-1}\oplus0_{\Ker A}$
and use this inverse in $g_1(Y_I)$.
In both cases $\mathbf w^{\mathsf T}\in\im(j_0I-B_0)$.
Since $Y_I+j_0I\succeq j_0I-B_0\succeq0$, completing the square gives
\[
 \begin{aligned}
 \mathbf w(Y_I+j_0I)^\dagger\mathbf w^{\mathsf T}
 &=\max_{\mathbf v}
 \{2\mathbf w\mathbf v-\mathbf v^{\mathsf T}(Y_I+j_0I)\mathbf v\}\\
 &\le\max_{\mathbf v}
 \{2\mathbf w\mathbf v-\mathbf v^{\mathsf T}(j_0I-B_0)\mathbf v\}
 =b_Hr_{j_0}+s/j_0,
 \end{aligned}
\]
where $r_{j_0}=\mathbf b^{\mathsf T}[C_I( j_0 I-A_I)]^\dagger\mathbf b$.
The range inclusion makes both maxima finite. Spectral calculus gives
$\mathbf w f(Y_I)\Pi_I\mathbf w^{\mathsf T}
\ge\mathbf w\{f(Y_I)-g_{j_0}(Y_I)\}\mathbf w^{\mathsf T}$.
By the preceding row identities, the difference between
$b_H\sum_{i\in I\setminus H}d_i b_i^2$ and the left side of
\eqref{eq:row-charge} is at most
\begin{align*}
 l_{j_0}={}&b_Hx(ar-s)
 +(2- j_0 )\{b_H(ar-s)-(c+ j_0 )(b_Hr+s)\}
 +(c+1)(b_Hr_{j_0}+s/ j_0 )\\
 ={}&b_H\{[ax-(2- j_0 )(x+ j_0 )]r+(c+1)r_{j_0}\}
 +s\{( j_0 -1)(c+1)/2-b_H(x+2- j_0 )\}.
\end{align*}
Evaluating on the two eigenspaces of $A_I$ gives
\[
 \begin{array}{c|c|c|c}
 \text{type}& j_0 &r&r_{j_0}\\\hline
 \mathrm S&2&2/(a+1)&2/(a+1)\\
 \mathrm B&2&a/(a^2-1)&(2a-1)/(3(a^2-1))\\
 \mathrm O&1&2/(a-1)&1/(a-1).
 \end{array}
\]
Suppose first that the pair has type $\mathrm B$ and $J_0$ contains its
unique index with nonzero correlation. Then the degree term on the
right of \eqref{eq:row-charge} is zero, and
$f(Y_I)\Pi_I\succeq0$ proves the bound. In every other type
$\mathrm B$ case, $s=0$. For type $\mathrm S$, use $b_H\ge1$ and
$s\le1$; for type $\mathrm O$, use $s\ge0$. Thus
\[
 \begin{array}{c|l}
 \mathrm S&l_2/b_H\le[2ax+2(c+1)]/(a+1)+ t_0 =2x+2+ t_0 ,\\[1mm]
 \mathrm B&l_2/b_H=[3a^2x+(c+1)(2a-1)]/[3(a^2-1)]
 =[(3a-1)x+2a-1]/[3(a-1)],\\[1mm]
 \mathrm O&l_1/b_H\le[2(a-1)x-2+c+1]/(a-1)
 =(2+1/(a-1))x+1.
 \end{array}
\]
For type $\mathrm Z$, $\mathbf w=0$. This proves \eqref{eq:row-charge}.
We shall also use
$c_{\mathrm S}-c_{\mathrm B}=((3a-5)x+4a-5)/(3(a-1))+t_0\ge0$.

Put $r_J=\sum_{i<j,\ i,j\in J}C_{ij}^2(d_i-b_{ij})(d_j-b_{ij})$,
$r_0=\sum_{i<j,\ i,j\in J}C_{ij}^2d_id_j$, and
\[
 l=\sum_{j\in J}d_jc_j,\qquad
 l_0=\sum_{H:\,|H\cap I|\le1}b_H\sum_{j\in H\cap J}c_j,
 \qquad e_{J,H}=\sum_{\{i,j\}\subseteq H\cap J}C_{ij}^2.
\]
Since $d_i,d_j\ge b_{ij}\ge1$ whenever the pair occurs,
\begin{align*}
 r_0-r_J
 &=\sum_{i<j,\ i,j\in J}C_{ij}^2b_{ij}(d_i+d_j-b_{ij})\\
 &\ge\sum_{i<j,\ i,j\in J}C_{ij}^2b_{ij}
 =\sum_Hb_He_{J,H}\ge\sum_H e_{J,H}.
\end{align*}
For $p_H=|H\cap I|\le1$, $r_H=|H\cap J|$,
\[
\begin{array}{c|ccccc}
 (p_H,r_H)&(0,0),(0,1),(0,2)&(0,3)&(1,0)&(1,1)&(1,2)\\\hline
 2e_H^2-4e_{J,H}\ \text{at most}&0&6&0&2&14.
\end{array}
\]
For $p_H=0$, $e_H=e_{J,H}\le3$. For $(p_H,r_H)=(1,2)$,
$0\le e_H-e_{J,H}\le2$ and $e_{J,H}\in\{0,1\}$, so
$2e_H^2-4e_{J,H}\le14$. In each case the displayed upper bound is at most $7r_H$.
For a three-element link $H=I\cup\{j\}$, the types
$\mathrm Z,\mathrm B,\mathrm S,\mathrm O$ give $2e_H^2=2,8,18,18$, respectively;
\[
 c_{\mathrm B}\ge\frac{2a-1}{3(a-1)}>\frac14,\qquad
 c_{\mathrm S}=\max\{2x+2,\tfrac32x+\tfrac{a+5}{2}\}\ge\frac72,
 \qquad c_{\mathrm O}\ge\frac{11}{4}.
\]
Hence $2e_H^2\le7+4c_j$ whenever $H=I\cup\{j\}$;
the link $H=I$ contributes at most $2$.
Multiplying the bounds for links meeting $J$ by $b_H\ge1$ and
summing gives
\begin{align*}
 2e_2-4\sum_H e_{J,H}+4l_0
 &\le7\sum_H b_H|H\cap J|
 +4\sum_Hb_H\sum_{j\in H\cap J}c_j+2\\
 &=7t_1+4l+2,\\
 2e_2+4r_J+4l_0&\le4r_0+7t_1+4l+2.
\end{align*}

We now compress $T+aI$ to $\im X$ and the lower coordinates
$(H,j)$ with $j\in H\cap J$. This compression is
$\left(\begin{smallmatrix}P&W\\W^{\mathsf T}&Q_T\end{smallmatrix}\right)$, where
$U_J=[\sqrt{b_H}R_{H\cap J}]_H$,
$Q_T=U_J^{\mathsf T}U_J-\bigoplus_H A[H\cap J]$, $A=C-aI$.
By \eqref{eq:short-link-norm}, $\|A[H\cap J]\|\le2$, giving
$Q_T\succeq-2I$; by hypothesis $Q_T\succeq- b_0 I$.
Since $WW^{\mathsf T}=\sum_{H,j}X^{\mathsf T}\mathbf w^{\mathsf T}\mathbf w X$,
\eqref{eq:row-charge} sums to
\begin{align*}
 \tr((P^2-cP)WW^{\mathsf T})
 &=\sum_{H,j}\mathbf w(Y_I^2-cY_I)\Pi_I\mathbf w^{\mathsf T}\\
 &\ge uv_1+vv_2-l_0,\qquad
 v_i=\sum_{j\in J}C_{ij}^2(d_j-b_{ij}),\\
 \sum_{H:j\in H,\,i\notin H}b_H&=d_j-b_{ij},\qquad
 r_{I,J}=\sum_{i\in I,j\in J}C_{ij}^2(d_i-b_{ij})(d_j-b_{ij})
 \le uv_1+vv_2.
\end{align*}
If $ s_4^+(T)\ge N$, apply \eqref{eq:master-certificate} to this
compression. Since $p_1+p_2=2a+u+v-e$,
$e_0=4a+N-p_1-p_2- h_0 - h_1 =2a+t_1- h_0 - h_1 +e$.
Substitute the lower bound for $p_1^4+p_2^4$ and the sum of
\eqref{eq:row-charge} into \eqref{eq:master-certificate}. Since
$e\ge0$, the result is
\begin{align*}
 \tr\bigl((T+aI)^4\bigr)
 &\ge(a+u)^4+(a+v)^4+4(u-b)(v-b)-2
 + h_4 +4r_{I,J}-4l_0+(k-4)a^4.
\end{align*}
Subtract the identity in Lemma~\ref{prop:shifted-fourth-identity}
and apply the bound for $2e_2+4r_J+4l_0$. This gives
\begin{align*}
 0&\ge h_4 +(k-4)a^4-\sum_{j\in J}(a+d_j)^4
 -(2e_2+4r_J+4l_0+2)\\
 &\ge g_0>0,
\end{align*}
a contradiction.

\medskip\noindent\emph{The remaining degree distributions.}
\label{prop:heavy-degrees}
The orthogonal-coordinate estimate handles
$2d_2\le\sum_{j\ge3}d_j$. Assume henceforth that
$u=d_1\ge v=d_2>h:=\frac12\sum_{j=3}^kd_j$.

\emph{Two large degrees: $d_3\le h$.}
Here $n=k-2\ge3$, $h\ge n/2\ge3/2$, and $1\le d_j\le h$ for $j\ge3$.
If $C_{12}=0$, apply the
\hyperref[prop:orthogonal-head-gain]{orthogonal-coordinate estimate}.
Otherwise, change coordinate signs so that $C_{12}=1$.
We treat $(a,h)=(2,3/2)$ and $(a,n,h)=(2,3,2)$ separately below.
For all other cases, put
\[
 \delta=\frac{(h+1-a)_+}{2(a+h-1)},\qquad
 p=h_0=c=h_1=a+h-\delta,\qquad b_0=2.
\]
Here $0\le\delta<1/2$ and $v\ge h+1/2$. The quadratic defining
$p_0-a$ therefore satisfies
\[
\begin{aligned}
 &(h-\delta)^2+(a-v)(h-\delta)-v(a-1)-b_{12}\\
 &\quad\le\tfrac12(h+1-a)-\delta(a+h-1/2)+\delta^2
 \le\delta(\delta-1/2)\le0.
\end{aligned}
\]
Since $h-\delta>0$, this gives $p\le p_0$.
The remaining scalar conditions hold because $c=p=h_0\ge a\ge2$.
Also,
\[
\begin{aligned}
 h_4&=2h_0^4+8\delta h_0^3\\
 &=2(a+h)^4-12(a+h)^2\delta^2+16(a+h)\delta^3-6\delta^4\\
 &\ge2(a+h)^4-12(a+h)^2\delta^2.
\end{aligned}
\]

The coefficients in \eqref{eq:uniform-row-charges} are nondecreasing
in $x=c-a\le h$, so their values at $x=h$ are upper bounds. Put
\[
 \begin{gathered}
 z=(3-h)_+,\quad l=\sum_{j=3}^kd_jc_j,\quad b_3=5h^2+(a-1)h+2hz,\\
 r_0=\sum_{3\le i<j\le k}C_{ij}^2d_id_j,\qquad
 r_{\rm all}=\sum_{3\le i<j\le k}d_id_j.
 \end{gathered}
\]
For $a\ge3$, the bound
$c_{\mathrm B}\le c_{\mathrm S}$ and
$\max\{c_{\mathrm S},c_{\mathrm O}\}\le\frac52h+\frac{a-1}{2}+z$
follows from $h+z\ge3$ and the following differences:
\[
 \begin{array}{c|c}
 \text{coefficient at }x=h&\frac52h+\frac{a-1}{2}+z-\text{coefficient}\\\hline
 c_{\mathrm O}&(\frac12-\frac1{a-1})h+\frac{a-3}{2}+z\ge0\\
 2h+2&(h+a-5+2z)/2\ge(a-2)/2\\
 \frac32h+\frac{a+5}{2}&h-3+z\ge0
 \end{array}
\]
Thus $l\le b_3$.
For $a=2$, let $v_0\le h$ be the degree of the type $\mathrm O$ coordinate
(zero if none exists), and let $s_0\le2h-v_0$ be the sum of the
degrees at type $\mathrm S$ coordinates. At $x=h$,
\[
 c_{\mathrm B}=(5h+3)/3,\qquad c_{\mathrm O}=3h+1,\qquad
 c_{\mathrm S}=2h+2+z/2.
\]
Every correlation between a type $\mathrm O$ coordinate and a type
$\mathrm S$ coordinate is zero. Hence
\begin{align*}
 r_0+l-r_{\rm all}
 &\le v_0c_{\mathrm O}+s_0c_{\mathrm S}+(2h-v_0-s_0)c_{\mathrm B}-v_0s_0\\
 &\le\max\{2hc_{\mathrm B}+v_0(c_{\mathrm O}-c_{\mathrm B}),
 2hc_{\mathrm S}+v_0(c_{\mathrm O}-c_{\mathrm S}-2h)+v_0^2\}\\
 &\le\max\{h(c_{\mathrm B}+c_{\mathrm O}),2hc_{\mathrm S}\}\le b_3,
\end{align*}
where
\begin{gather*}
 c_{\mathrm O}-c_{\mathrm B}=4h/3,\quad
 v_0(c_{\mathrm O}-c_{\mathrm S}-2h)+v_0^2=v_0(v_0-h-1-z/2)\le0,\\
 b_3-h(c_{\mathrm B}+c_{\mathrm O})=\tfrac h3(h-3+6z)\ge0,\quad
 b_3-2hc_{\mathrm S}=h(h-3+z)\ge0.
\end{gather*}
Consequently $r_0+l\le r_{\rm all}+b_3$ for all $a\ge2$.
Set $d_j^{(0)}=2h^j-\sum_{i=3}^kd_i^j$. Since $4r_{\rm all}=4h^2+2d_2^{(0)}$, it remains to prove
\begin{equation}\label{eq:two-head-margin}
 d_4^{(0)}+4ad_3^{(0)}+(6a^2-2)d_2^{(0)}
 -4h^2-14h-4-4b_3-12(a+h)^2\delta^2>0.
\end{equation}
For fixed $h$, the last subtracted term is nonincreasing in $a$:
\[
 12(a+h)^2\delta^2
 =3\left(\frac{a+h}{a+h-1}\right)^2(h+1-a)_+^2.
\]
The remaining expression is increasing in $a$, since
\[
 4d_3^{(0)}+12ad_2^{(0)}-4h
 \ge12h^2+32h-48>0\qquad(h\ge3/2).
\]
For $h\ge2$, the vector $(h,h-1,1,0,\ldots)$ majorizes the
remaining degrees, giving $d_2^{(0)}\ge2h-2$ and
$d_3^{(0)}\ge3h^2-3h$. For $h=3/2$, the remaining degrees are
$(1,1,1)$ and satisfy the same bounds.

Suppose $h\ge3$. The same majorization gives
$d_4^{(0)}\ge4h^3-6h^2+4h-2$, and
\[
 \begin{gathered}
 12(a+h)^2\delta^2\le3\left(h-\frac2{h+1}\right)^2
 \le3h^2-\frac{33}{4},\\
 3h^2-\frac{33}{4}-3\left(h-\frac2{h+1}\right)^2
 =\frac{3(h-3)(5h+9)}{4(h+1)^2}\ge0.
 \end{gathered}
\]
Substituting these bounds and $a=2$ into
\eqref{eq:two-head-margin} bounds its left side below by
\[
 \begin{aligned}
 &4(h-3)^3+30(h-3)^2+78(h-3)+22-3h^2+33/4\\
 &\qquad=4(h-3)^3+27(h-3)^2+60(h-3)+13/4>0.
 \end{aligned}
\]
If $a=h=2$ and $n\ge4$, the four remaining degrees all equal one
and $\delta=1/6$. The left side of \eqref{eq:two-head-margin} is then
at least $60-16/3=164/3>0$.
For the remaining half-integer values of $h$, the table gives lower
bounds for $d_j^{(0)}$ and the margin in \eqref{eq:two-head-margin},
and an upper bound for the subtracted term. Each row is evaluated
at $a=a_0$:
\[
\begin{array}{c|c|c|c|c}
 h&a_0&(d_2^{(0)},d_3^{(0)},d_4^{(0)})
 &12(a+h)^2\delta^2&\text{margin}\\\hline
 3/2&3&(3/2,15/4,57/8)&0&169/8\\
 2&3&(2,6,14)&0&30\\
 5/2&2&(7/2,57/4,361/8)&12&121/8
\end{array}
\]
The corresponding vectors majorizing the remaining degrees are
$(1,1,1)$, $(2,1,1)$ and $(2,2,1)$. In the last row,
$\delta=3/14<2/9$ gives
$12(a+h)^2\delta^2<12(9/2)^2(2/9)^2=12$.
Monotonicity in $a$ extends each row to every $a\ge a_0$.
This proves \eqref{eq:two-head-margin} in all the stated cases.

\emph{The exceptional degree sequences: $a=2$, $(d_1,\ldots,d_5)=(u,v,t_1-2,1,1)$,
$t_1\in\{3,4\}$, $u\ge v\ge t_1-1$.}
If the core is a positive clique, then $C\succ0$ and $F_C(0)\succ0$.
The inertia identity then gives five positive eigenvalues of $T$,
so $s_4^+(T)<\tr(T_+)\le N$.
If the core is not a positive clique, the five-coordinate classification
gives at least two zero correlations. If $C_{12}=0$, apply the
orthogonal-coordinate estimate. Otherwise, change coordinate signs
so that $C_{12}=1$.

First let $v=t_1-1$. The fourth-eigenvalue bound with selected
coordinates $1,2,3$ gives
$\lambda_4(T)<13/9$. Retain coordinate $1$ in the orthogonal compression and choose
\[
 h=t_1-\frac{11}{9},\qquad c_0=\frac{13}{9},\qquad
 p=u+2,\quad h_0 =2+h,\quad h_1 =2+c_0=\frac{31}{9}.
\]
Then $u\ge t_1-1>h>7/4$, $p> h_0 \ge h_1$, and
$p+2 h_0 + h_1 =8+N$.
Since $h>7/4$, every local term $l_H$ in
\eqref{eq:clipped-energy} is nonpositive.
The two least pair products of $(h,t_1-1,t_1-2,1,1)$ are $1,t_1-2$;
the two missing correlations therefore give
\[
 2e_2+4r_*-4(u-h)\sum_{j=2}^5C_{1j}^2(d_j-b_{1j})
 \le4\{h(2t_1-1)+t_1^2-2\}.
\]
If $s_4^+(T)\ge N$, subtracting the exact moment from
\eqref{eq:orthogonal-certificate} gives
\[
 \begin{aligned}
 0&\ge2 h_0^4+ h_1^4+16-(t_1+1)^4-t_1^4-2\cdot3^4
 -4\{h(2t_1-1)+t_1^2-2\}\\
 &=\frac{2}{729}\{810(t_1-3)^3+3375(t_1-3)^2+3316(t_1-3)+569\}>0.
 \end{aligned}
\]
Now let $v\ge t_1$. If $t_1=3$, each coordinate in $J$ belongs
to one link of weight one. Connectedness forces $b_{12}\ge1$:
without a link containing both $1$ and $2$, the links meeting
these two coordinates would lie in different components. Hence
\[
 p_0\ge\begin{cases}(5+\sqrt{17})/2>41/9,&t_1=3,\\
 3+\sqrt5>47/9,&t_1=4.
 \end{cases}
\]
Choose the following parameters in the two-coordinate criterion:
\[
\begin{array}{c|ccccc|c}
 t_1&p&h_0&c&b_0&h_1&(c-b_0)p^2+cb_0p-h_0^3\\\hline
 3&41/9&7/2&25/12&0&7/2&701/1944\\
 4&47/9&4&19/6&2&31/9&433/486
\end{array}
\]
For $t_1=3$, connectedness excludes a link consisting of all three
coordinates in $J$. Their lower principal block is a direct sum of
matrices $I-A[H\cap J]$ of order at most two, hence is positive
semidefinite. Thus $b_0=0$ is valid. For $t_1=4$, use $b_0=2$
and $a+\lambda_4(T)<31/9$.

For $t_1=3$, the coefficients in \eqref{eq:uniform-row-charges} are
$c_{\mathrm S}=29/8$, $c_{\mathrm O}=11/4$, $c_{\mathrm B}=41/36$,
and $c_{\mathrm Z}=0$.
If $s,o$ count the type $\mathrm S,\mathrm O$ coordinates in $J$,
then $s\le2$, $o\le1$, and the $so$ correlations between these
types vanish. Consequently
\[
 \begin{aligned}
 b_*&\le3+sc_{\mathrm S}+oc_{\mathrm O}
       +(3-s-o)c_{\mathrm B}-so\\
 &\le\max\{3+2c_{\mathrm S}+c_{\mathrm B},
                  1+2c_{\mathrm S}+c_{\mathrm O}\}=205/18.
 \end{aligned}
\]
The last step uses $c_{\mathrm S}-c_{\mathrm B}-o>0$ and
$o\in\{0,1\}$. For $t_1=4$, every coefficient is at most
$21/4$, and the degrees on $J$ are $(2,1,1)$. Hence
$b_*\le5+4(21/4)=26$.
Substituting these bounds into the definition of $g_0$, we obtain
\[
 g_0\ge\begin{cases}
 2(7/2)^4+16-3\cdot3^4-4(205/18)-7\cdot3-4
       =185/72,&t_1=3,\\
 \begin{aligned}
 &(31/9)^4+256(5/9)+16-162-104-28-4\\
 &\hspace{15mm}=6439/6561,
 \end{aligned}&t_1=4.
 \end{cases}
\]
Both bounds are positive, proving the two exceptional cases.

\emph{Three large degrees: $d_3>h$.}
Set
\[
 w=d_3,\quad t=\sum_{j=4}^kd_j,\qquad
 u\ge v\ge w>t\ge k-3\ge2,\quad w\ge t+1.
\]
The fourth-eigenvalue bound gives
$\lambda_4(T)\le c_0=t-1/2$; put $h_1 =a+c_0$.
Every orthogonal compression below has $h>7/4$, so its terms
$l_H$ are nonpositive.
For
$g_a(z)=z^4+4az^3+(6a^2-2)z^2+4a^3z$,
the vector $(t-1,1,0,\ldots)$ majorizes $(d_4,\ldots,d_k)$ because
$d_j\le t-(k-4)\le t-1$ for $4\le j\le k$ and $\sum_{j=4}^k d_j=t$.
The coefficients of $g_a$ are nonnegative, so
\eqref{eq:prefix-convexity} gives
\begin{align*}
 &(k-4)a^4-\sum_{j=4}^k(a+d_j)^4+2\sum_{j=4}^kd_j^2
 =-a^4-\sum_{j=4}^kg_a(d_j)\\
 &\hspace{15mm}\ge a^4-(a+t-1)^4-(a+1)^4+2\{(t-1)^2+1\}.
\end{align*}
Suppose $v\ge w+1$ and $C_{12}\ne0$. Switch to $C_{12}=1$ and choose
$p= h_0 =c=a+w$, $b_0 =2$, $h_1 =a+t-1/2$. Then
$p_0 >a+v-1\ge a+w$, $h_0 - h_1 =w-t+1/2>0$, and
$( h_0 -2) h_0^2+2 h_0^2- h_0^3=0$.
At $x=w$, $c_{\mathrm S}=2w+2+\tfrac12(a+1-w)_+$,
$c_{\mathrm O}=(2+1/(a-1))w+1$, $c_{\mathrm B}\le c_{\mathrm S}$, and $c_{\mathrm Z}=0$.
For $a=2$, at most one coordinate in $J$ has type $\mathrm O$, and its
degree is at most $w$. Moreover,
$c_{\mathrm O}=3w+1\ge c_{\mathrm S}=2w+2$; hence
$l\le t(2w+2)+w(3w+1)$.
For $a\ge3$ and $w\ge3$,
$c_{\mathrm O}\le\tfrac52w+\tfrac{a-1}{2}$ and
$c_{\mathrm S}=\max\{2w+2,\tfrac32w+\tfrac{a+5}{2}\}
\le\tfrac52w+\tfrac{a-1}{2}$.
Thus in all cases
\begin{gather*}
 l\le\overline l:=t(2w+2)+w(3w+1)+\tfrac12(a-2)(w+t),\\
 \overline l-(w+t)(\tfrac52w+\tfrac{a-1}{2})
 =\tfrac12w(w-t)+\tfrac12w+\tfrac32t>0.
\end{gather*}
Also $b_*\le wt+(t^2-\sum_{j=4}^kd_j^2)/2+\overline l$ and
$2a+w+t- h_0 - h_1 =1/2$, so $h_4 = h_0^4+ h_1^4+2 h_0^3$.
Cancel $h_0^4=(a+w)^4$ and apply the preceding estimate for
$d_4,\ldots,d_k$. It then suffices to prove $f_2(a,t,w)>0$, where
\begin{align*}
 f_2(a,t,w)={}&2(a+w)^3+(a+t-1/2)^4+a^4
 -(a+t-1)^4-(a+1)^4\\
 &-4wt-4t-7(w+t)-4\overline l.
\end{align*}
In the remaining cases, select the first $r$ coordinates, where
\[
 r=\begin{cases}
 2,&v\ge w+1,\ C_{12}=0,\\
 1,&v=w<u,\\
 0,&u=v=w,
 \end{cases}
 \qquad s=3-r,\qquad h=w+\frac1{2s}.
\]
The selected coordinates are mutually orthogonal and have degrees
at least $w+1>h$. The other $s$ degrees among the first three
coordinates equal $w$, and $sh+c_0=sw+t$.
Put $v_i=\sum_{j>r}C_{ij}^2(d_j-b_{ij})$,
$g_*=4\sum_{i=1}^r(d_i-h)v_i$ and
$k_s=\binom{s}{2}w^2+swt+(3-s)h(sw+t)$.
Assume $ s_4^+(T)\ge N$. Equations
\eqref{eq:orthogonal-certificate} and \eqref{eq:clipped-energy}, with
$l_H\le0$, give respectively
\begin{align*}
 \tr\bigl((T+aI)^4\bigr)
 &\ge\sum_{i=1}^r(a+d_i)^4+s(a+h)^4+ h_1^4+(k-4)a^4+g_*,\\
 2e_2+4r_*-g_*
 &\le4k_s+2\left(t^2-\sum_{j=4}^kd_j^2\right).
\end{align*}
Subtracting the exact moment and applying the same bound for
$d_4,\ldots,d_k$ yields
\[
 0\ge g_s:=s\{(a+h)^4-(a+w)^4\}+ h_1^4+a^4
 -(a+t-1)^4-(a+1)^4-4(k_s+t-1).
\]
But
\begin{gather*}
 k_s=\tfrac{s(5-s)}2w^2+3wt+\tfrac{3-s}{2}w+\tfrac{3-s}{2s}t
 \le3w^2+3wt+w+t,\\
 s\{(a+w+1/(2s))^4-(a+w)^4\}
 =2(a+w)^3+\tfrac{3(a+w)^2}{2s}+\tfrac{a+w}{2s^2}
 +\tfrac1{16s^3}>2(a+w)^3,\\
 g_s-f_2(a,t,w)>-4k_s+4wt+7(w+t)+4\overline l+4
 \ge7w+11t+4+2(a-2)(w+t)>0.
\end{gather*}
It remains to prove $f_2(a,t,w)>0$. Expanding gives
\begin{align*}
 f_2(a,t,w)-f_2(2,t,w)
 =(a-2)\bigl[&(a+2)\{6(t+w)-21/2\}
 +6(t-2)^2+13(t-2)\\
 &+6(w-3)^2+34(w-3)+99/2\bigr]\ge0,\\
 f_2(2,t,w)={}&2(w-3)^3+12(w-3)^2+43(w-3)\\
 &+2(t-2)^3+\tfrac{27}{2}(t-2)^2+\tfrac{17}{2}(t-2)
 +6(w-t-1)^2+\tfrac{49}{16}>0.
\end{align*}
Since $a\ge2$, $t\ge2$, and $w\ge t+1$, the first expression is
nonnegative and the second is positive. Hence $f_2(a,t,w)>0$.
For the correlated pair, this proves $g_0>0$ as required. In each
remaining case, it contradicts $0\ge g_s>f_2(a,t,w)>0$.
These cases and the initial degree reduction exhaust all positive
integer degree sequences with $k\ge5$.
\end{proof}

\section{Equality Cases}\label{sec:unified-equality}

\subsection{Equality in the Projection Comparisons}

\begin{lemma}\label{lem:schur-zero-pattern}
Let $M(t)=tI_4+A$, where $A\in\{A_{\mathrm p},A_{\mathrm d}\}$
and $t>2$.
Every off-diagonal entry of every coordinate Schur complement
$M(t)/I$ either vanishes identically or is nonzero for all $t>2$.
\end{lemma}

\begin{proof}
For $I=\varnothing$, the entries are constant.
If $I=\{r\}$, the entry indexed by $i\ne j$ is
$A_{ij}-A_{ir}A_{rj}/t$.
Since all entries of $A$ belong to $\{0,1,-1\}$, this expression
is either identically zero or nonzero for $t>2$.
If $|I|=2$, the Schur complement has order two, so its off-diagonal
entry vanishes exactly when the corresponding entry of $M(t)^{-1}$
vanishes.
The \hyperref[prop:conditioned-first-flow]{cofactor table} shows
that no off-diagonal inverse entry vanishes for the paw, and only
the entry $13$ vanishes for the diamond; the latter is identically
zero. For $|I|\ge3$, there are no off-diagonal entries.
\end{proof}

\begin{lemma}\label{prop:four-residual-equality}
Let $C$ arise from a complete simplex, with at most four used
coordinates, positive integer link weights, and $|H|\le3$.
Split every link into the correlation components of $C[H]$.
Then $\tr(T_+)=N$ if and only if every correlation component
satisfies one of the following conditions:
\begin{enumerate}
\item It has at most three coordinates, and every two split links
are nested, or are disjoint with $C[H_1,H_2]=0$.
\item It has four coordinates and, after sign changes and
relabelling, $C=aI+A_{\rm d}$.
Both $012$ and $023$ occur. Every other link is $02$, $0$, or $2$,
and at most one of $0,2$ occurs.
\end{enumerate}
In that case, the positive eigenvalues of $T$ are the coordinate
degrees $d_1,\ldots,d_k$, counted with multiplicity.
\end{lemma}

\begin{proof}
Suppose that $\tr(T_+)=N$.
By Lemma~\ref{prop:incidence-equality}, different link-incidence
components are orthogonal.
For Stieltjes matrices and positive cliques,
Lemma~\ref{thm:stieltjes-residual} makes every two split links
nested or orthogonal. Their maximal links are therefore disjoint
and orthogonal. Since each has size at most three and the maximal
links cover all used coordinates, each correlation component has
at most three coordinates.
By Lemma~\ref{prop:small-core-geometry}, it remains to consider a
paw or mixed diamond with connected link-incidence graph.

Delete the singleton links. The link-incidence graph on the four
coordinates remains connected, and the strictness assertion in
Lemma~\ref{prop:four-coordinate} shows that equality still holds.
The projection comparisons in that proof apply to the paw,
diamonds with $a\ge3$, diamonds with $a=2$ and a path triple,
and the seven-type case except for the triangle splittings.
When all links contain a common coordinate, the separated largest
eigenvalue is constant and the quotient comparisons give the
remaining leading sums. Thus, in these cases, the leading sums
are nonincreasing on the required intervals between positive
eigenvalues. Lemmas~\ref{lem:schur-zero-pattern} and
\ref{prop:trace-equality-commutation} imply that all nonsingleton
link projections commute.

Nested pairs commute, and disjoint pairs commute exactly when
their cross block in $C$ is zero.
Consider incomparable intersecting links $H_1,H_2$.
If their intersection is $\{r\}$, commutation for every $t>a$ requires
\[
C[H_1\setminus\{r\},H_2\setminus\{r\}]
-t^{-1}C[H_1\setminus\{r\},r]C[r,H_2\setminus\{r\}]=0.
\]
Both coefficients must vanish. However, each split link is
connected, so $r$ has a neighbour in each difference set and the
outer product is nonzero.
The only remaining possibility is two triples with a common pair.
Their spans in the metric $M(t)$ are distinct hyperplanes.
Their projections commute exactly when their normal lines are
orthogonal, which the inverse zero pattern permits only for
$012,023$ in the diamond.
If these two triples did not both occur, the maximal links would
be disjoint and orthogonal, contrary to connectedness on four
coordinates. Thus both occur, and the only further nonsingleton
link that commutes with both is $02$.

It remains to check the two triangle splittings used in
Lemma~\ref{prop:four-coordinate}.
Suppose one triangle has weight $p$ and every other link lies in
the opposite triangle.
Equation~\eqref{eq:diamond-triangle-projections} separates the
eigenvalue $p$ and adds weight $p$ to $02$ in the opposite triangle,
preserving the difference $\tr(T_+)-N$.
Equality in the three-coordinate bound makes the resulting links
a chain. Since all four original coordinates are used, the opposite
triangle must occur, giving $012,023$ and possibly $02$.
If $a=2$ and both triangles occur, subtract their smaller weight
from both, add it to $02$, and increase $a$ by that weight.
The weights are positive integers, so the new value of $a$ is at
least three. The triangle-cancellation identity preserves
$\tr(T_+)-N$.
If four coordinates remain, the classification above applies.
Otherwise the links form a chain in one triangle and include $02$.
Restoring the triangle weights gives the same list in either case.

Now restore the singleton links.
Adding $1$ gives the incomparable correlated links $1,02$ in the
reduction to $012$, so the three-coordinate inequality is strict.
Adding $3$ has the same effect in $023$.
Adding both $0$ and $2$ gives incomparable correlated singletons
in either triangle and again makes the inequality strict.
Further singleton insertions preserve strictness. Hence only
$0$ or $2$ may occur, and they cannot both occur.

Conversely, the projections in every listed configuration commute
by nesting, orthogonality, or
\eqref{eq:diamond-triangle-projections}.
Lemma~\ref{lem:commuting-flow-spectrum} gives the positive
eigenvalues $d_1,\ldots,d_k$, whose sum is $N$.
\end{proof}

\subsection{The Equality Criterion and the Main Theorems}

\begin{lemma}\label{prop:all-residual-equality}
Let $\ell\in\{2,3,4\}$, and let $C$ arise from a complete simplex,
with positive integer link weights and $|H|\le\ell-1$.
Then $s_\ell^+(T)=N$ if and only if $k\le\ell$ and every pair of
split links satisfies one of \ref{cond:nested}--\ref{cond:diamond}.
In that case, the positive eigenvalues of $T$ are
$d_1,\ldots,d_k$, counted with multiplicity.
\end{lemma}

\begin{proof}
The empty family satisfies the assertion, so assume $k>0$.
Suppose first that $s_\ell^+(T)=N$.
Split each link into its correlation components and separate the
link-incidence components.
When $a\ge2$, Lemma~\ref{prop:many-coordinate} gives
$s_\ell^+(T)\le s_4^+(T)<N$ on any component with at least five
coordinates.
When $a=1$, the core can be changed to an all-ones matrix.
Equality on such a component would require a chain whose maximal
link contains every coordinate, which is impossible for
$|H|\le3$.
Thus every component has at most four coordinates.
The component inequalities force equality in each positive-trace
bound, and Lemma~\ref{prop:incidence-equality} makes different
components orthogonal.
Lemma~\ref{prop:four-residual-equality} gives nested or orthogonal
links on components of order at most three, and the mixed-diamond
configuration on four coordinates.

For the latter, apply the common-neighbour calculation in
Lemma~\ref{prop:small-core-geometry} to the original ridges.
There are a $(q-3)$-set $\tau_0\subseteq Q$ and distinct vertices
$v_1,v_2,v_3,v_4\in Q\setminus\tau_0$ such that the common ridges
are $\tau_0\cup\{v_1,v_2\}$ and $\tau_0\cup\{v_1,v_3\}$,
and the other two are $\tau_0\cup\{v_1,v_4\}$ and
$\tau_0\cup\{v_2,v_3\}$.
The two triples therefore satisfy condition~\ref{cond:diamond};
all other pairs in this configuration are nested.
This proves the stated pair conditions.

Conversely, these pair conditions make the projections commute
by nesting, orthogonality, or
\eqref{eq:diamond-triangle-projections} on the four-coordinate
span.
Lemma~\ref{lem:commuting-flow-spectrum} gives the positive
eigenvalues $d_1,\ldots,d_k$, whose sum is $N$.
They are all positive, so $s_\ell^+(T)=N$ exactly when $k\le\ell$.
The same observation gives $k\le\ell$ in the necessity argument.
\end{proof}

\begin{proof}[\textup{\textbf{Proof of Theorem~\ref{thm:main}}}]
If $q=1$, the boundary matrix has one all-ones row, so
$s_\ell(F)=|F|=D_\ell(F)$.
If $|F|\le\ell$, every vertex degree is at most $\ell$ and
$\rank M_F\le|F|\le\ell$; both sides then equal
$\tr M_F=q|F|$.
Assume $q\ge2$ and $|F|>\ell$, delete isolated vertices, and use
the core $Q$ and links $H_u$ defined in the Introduction.

For $m<q$, apply Lemma~\ref{prop:short-reduction}.
For $m\ge q$, the same lemma reduces the assertion to
$s_\ell^+(T)\le N$ for a complete-simplex block matrix with
$|H|\le\ell-1$.
If $\ell=2$, every link is a singleton and
\eqref{eq:singleton-factorization} gives the bound.
For $\ell=3,4$, separate the link-incidence components.
When $a=1$, use Lemma~\ref{thm:stieltjes-residual}.
When $a\ge2$, use Lemma~\ref{prop:four-coordinate} on components
with at most four coordinates and Lemma~\ref{prop:many-coordinate}
on the others, where $s_\ell^+(T)\le s_4^+(T)<N$.
Lemma~\ref{prop:incidence-component-reduction} combines these
bounds and gives $s_\ell(F)\le D_\ell(F)$.
\end{proof}

\begin{proof}[\textup{\textbf{Proof of Theorem~\ref{thm:equality}}}]
For $m<q$, Lemma~\ref{prop:short-reduction} gives precisely
cases~\ref{case:tetrahedron}--\ref{case:rooted}.
Assume henceforth that $m\ge q$, and put
$N=D_\ell(F)-\ell m=\sum_{u\notin Q}d_F(u)$.

\medskip\noindent\emph{Facets with at least two outside vertices.}
Put $\widehat F=F\cup\binom Qq$, and let $G$ consist of
$\binom Qq$ and the facets of $F$ with exactly one outside vertex.
By \eqref{eq:insertion-bound}, inserting a facet $\sigma$ with
$j\ge2$ outside vertices increases the spectral bound by at most
\[
\frac{q-m+\sqrt{(q-m)^2+4n_\sigma}}2,
\qquad (j-1)n_\sigma\le(\ell-2)j\le2j.
\]
For $j\ge3$, we have $n_\sigma<j^2\le j(j+m-q)$, so the increase
is strictly less than $j$.
For $j=2$, equality with $2$ requires $\ell=4$, $m=q$, and
$n_\sigma=4$.

Suppose $s_\ell(F)=D_\ell(F)$.
Equality must hold at every insertion, in any order.
If a facet $\sigma$ has two outside vertices, insert it first.
Then $\sigma=(Q\setminus\{w_1,w_2\})\cup\{u,v\}$, and all four
adjacent facets with one outside vertex are present in $G$.
Since $d_F(u),d_F(v)\le3$, we obtain
$H_u=H_v=\{Q\setminus\{w_1\},Q\setminus\{w_2\}\}$.
These degrees are already three, so no other facet with two
outside vertices meets $\{u,v\}$.
Together with the strict inequality for $j\ge3$, this proves
condition~\ref{case:outside}.

\medskip\noindent\emph{The positive and zero spectral spaces.}
Retain a separate lower coordinate $(u,\rho)$ for every
$\rho\in H_u$, including repeated links, and put
\[
Y=\bd_Q^{\mathsf T}R_{\calC},\qquad C=Y^{\mathsf T}Y,
\qquad U=[R_{H_u}]_u,\qquad S=\bigoplus_u C[H_u].
\]
The complete-simplex identity gives
$\dim\scrH_Q=h=\binom{m-1}{q-1}$.
The subspace $\scrH_Q\oplus\R^{G\setminus\binom Qq}$ is invariant,
and the restriction is
\begin{equation}\label{eq:core-equality-block}
A_0=(M_G-mI)\big|_{\scrH_Q\oplus\R^{G\setminus\binom Qq}}
=\begin{pmatrix}
0&YU\\ U^{\mathsf T}Y^{\mathsf T}&U^{\mathsf T}U-S
\end{pmatrix}.
\end{equation}
On the remaining core summand $\Ker\bd_Q$, the matrix is $-mI$,
which contributes neither positive nor zero eigenvalues.

Consider an outside facet
$\sigma=(Q\setminus\{w_1,w_2\})\cup\{u,v\}$.
Here $m=q$, so consistent coordinate sign changes make $C=J$
on the used ridges.
In the order $Q\setminus\{w_2\},Q\setminus\{w_1\}$, the boundary
signs give its adjacency vector $\mathbf b_\sigma$ entries
$(1,-1)$ in the coordinates belonging to $u$ and $(-1,1)$ in
those belonging to $v$, up to an overall sign.
Consequently $U\mathbf b_\sigma=S\mathbf b_\sigma=0$,
$A_0\mathbf b_\sigma=0$, and $\|\mathbf b_\sigma\|=2$, where
the vector is extended by zero on the core.
Distinct outside facets have disjoint outside pairs, so these
vectors have disjoint supports and the facets share no ridge.
Each insertion therefore replaces one zero eigenvalue of $A_0$
by a block $\left(\begin{smallmatrix}0&2\\2&0\end{smallmatrix}\right)$,
with eigenvalues $2,-2$.

Let $A=M_{\widehat F}-mI$ on the full facet space.
Its restriction to $\Ker\bd_Q$ remains $-mI$, and hence
\[
\begin{gathered}
\tr(A_+)=\tr((A_0)_+)+2e,\qquad
N=\sum_u|H_u|+2e,\\
s_\ell(F)-\ell m\le s_\ell^+(A)
\le s_\ell^+(A_0)+2e\le N.
\end{gathered}
\]
Equality requires $s_\ell^+(A_0)=\sum_u|H_u|$.
By Lemma~\ref{prop:all-residual-equality}, the split links satisfy
the stated pair conditions, and the positive eigenvalues of $A_0$
are $d_i$, $i\in\calC$.
Conversely, the pair conditions make the projections commute and
give these same $k$ positive eigenvalues.
Thus $A$ has $k+e$ positive eigenvalues, with sum $N$.

The Gram identification, \eqref{eq:residual-determinant}, and the
constant spectrum of the projection family give the polynomial
identity
\[
z^k\det(zI-A_0)
=z^h\det(zI+S)\prod_{i\in\calC}(z-d_i).
\]
It remains valid when $h<k$.
Since every $d_i>0$, comparison of the zero multiplicities gives
$n_0(A_0)=h-k+\sum_u\dim\Ker C[H_u]$.
We verify below that this sum is $c_\ell(F)$.
Each inserted facet removes one zero eigenvalue, so
$n_0(A)=h-k+c_\ell(F)-e$.

\medskip\noindent\emph{Support on the original family.}
By \eqref{eq:zero-upper-range}, the projection of the positive
spectral space of $A_0$ onto the core is $\im(YU)=\im Y$, since
$UU^{\mathsf T}\succ0$ on the used coordinates.
The positive eigenvectors from the outside insertions have zero
core coordinates.
Thus the positive spectral space of $A$ is supported on $F$
exactly when
\begin{equation}\label{eq:cofacet-support}
Y[\calN,\calC]=0.
\end{equation}
Since $Y_{\sigma\rho}\ne0$ exactly when $\rho\subseteq\sigma$,
this is equivalent to $\rho\cup\{v\}\in F$ for every
$\rho\in\calC$ and $v\in Q\setminus\rho$.

Assume \eqref{eq:cofacet-support}.
For $\sigma\in\calN$, put
$\mathbf w_\sigma=\Pi_{\scrH_Q}e_\sigma$.
The columns of $Y$ lie in $\scrH_Q$, so
$Y^{\mathsf T}\mathbf w_\sigma=Y^{\mathsf T}e_\sigma=0$ and
$(\mathbf w_\sigma,0)\in\Ker A_0$.
These vectors have zero lower coordinates and therefore remain
in $\Ker A$ after the insertions.
Since $\Pi_{\scrH_Q}=m^{-1}\bd_Q^{\mathsf T}\bd_Q$,
\[
\dim\Span\{\mathbf w_\sigma:\sigma\in\calN\}
=\rank(\bd_Q^{\mathsf T}\bd_{\calN})
=\rank_{\R}\bd_{\calN}.
\]
The last equality holds because $\bd_Q^{\mathsf T}$ is injective
on $\im\bd_Q$, which contains $\im\bd_{\calN}$.
Every vector in $\Ker A$ has its core part in $\scrH_Q$, so its
$\sigma$-coordinate equals its inner product with
$(\mathbf w_\sigma,0)$.
As these representing vectors themselves belong to $\Ker A$,
restriction from $\Ker A$ to the coordinates $\calN$ has rank
$\rank_{\R}\bd_{\calN}$.
The subspace of $\Ker A$ supported on $F$ therefore has dimension
$h-k+c_\ell(F)-e-\rank_{\R}\bd_{\calN}$.

Identify $\R^F$ with its coordinate subspace in
$\R^{\widehat F}$.
Since $|F|>\ell$, the variational principle gives
\[
s_\ell(F)-\ell m
=\max_{\substack{\Pi^2=\Pi=\Pi^{\mathsf T},\ \rank\Pi=\ell\\
                  \im\Pi\subseteq\R^F}}\tr(\Pi A).
\]
Under the link conditions, $\tr(A_+)=N$ and
$N-\tr(\Pi A)=\tr((I-\Pi)A_+)+\tr(\Pi A_-)\ge0$.
Equality holds exactly when $\im\Pi$ contains the positive
spectral space and is orthogonal to the negative spectral space.
The remaining $\ell-k-e$ directions must lie in $\Ker A$ and
be supported on $F$.
Together with \eqref{eq:cofacet-support}, this is equivalent to
$k+e\le\ell$ and
$\rank_{\R}\bd_{\calN}\le h+c_\ell(F)-\ell$.
Conversely, under these conditions, take the orthogonal sum of
the positive spectral space and an $(\ell-k-e)$-dimensional
subspace of $\Ker A$ supported on $F$.
Its variational value is $N$, and $D_\ell(F)=\ell m+N$ gives
$s_\ell(F)=D_\ell(F)$.
This proves the equivalence in case~\ref{case:complete}.

\medskip\noindent\emph{The value of $c_\ell(F)$.}
It remains to verify $c_\ell(F)=\sum_u\dim\Ker C[H_u]$.
If $m=q$, sign changes make each nonempty link Gram matrix
$J_{|H_u|}$, with nullity $|H_u|-1$.
If $m=q+1$, all link Gram matrices of order at most two are
positive definite.
For a three-coordinate link $H=\{1,2,3\}$,
\[
\det C[H]=8-2\sum_{i<j\in H}C_{ij}^2
                  +2C_{12}C_{13}C_{23}.
\]
The determinant vanishes exactly for a negative triangle.
Sign changes then give $3I_3-J_3$, of nullity one.
The ridge-triangle classification says that this occurs exactly
when $|\bigcap_{\rho\in H}\rho|=q-2$.
Finally, for $m\ge q+2$, the bounds $|H|\le3$ and
$|C_{ij}|\le1$ give $C[H]\succeq(m-q-1)I\succ0$.
These are precisely the three cases defining $c_\ell(F)$.
\end{proof}

\medskip\noindent
\textbf{Declaration on the use of AI.}
During the preparation of this manuscript, ChatGPT and DeepSeek were
used for language editing. ChatGPT was also used to explore proof
simplifications and check algebraic calculations.

\end{document}